\documentclass[11pt]{article}
\usepackage{amsmath, amssymb,verbatim,appendix}
\usepackage[mathscr]{eucal}
\usepackage{amscd}
\usepackage{amsthm}
\usepackage{stmaryrd}
\usepackage{enumerate}
\usepackage{comment}
\usepackage[latin1]{inputenc}
\usepackage{tikz}
\usetikzlibrary{shapes,arrows}
\usepackage{cite}
\usepackage{url}
\usepackage[margin=1in]{geometry}

\AtBeginDocument{%
   \def\MR#1{}
}

\newtheorem{theorem}{Theorem}
\newtheorem{lemma}{Lemma}
\newtheorem{corollary}{Corollary}
\newtheorem{proposition}{Proposition}
\newtheorem{remark}{Remark}
\newtheorem{question}{Question}
\newtheorem{conjecture}{Conjecture}
\newtheorem{observation}{Ovservation}

\newtheorem*{theorem3}{Theorem 3}

\newtheorem*{theorem5}{Theorem 5}

\newtheorem*{theorem7}{Theorem 7}

\theoremstyle{definition}
\newtheorem{example}{Example}
\newtheorem{definition}{Definition}

\newcommand{\abar}{\bar{a}}
\newcommand{\bbar}{\bar{b}}
\newcommand{\cbar}{\bar{c}}

\newcommand{\ebar}{\bar{e}}

\newcommand{\xbar}{\bar{x}}

\newcommand{\vbar}{\bar{v}}

\newcommand{\tildeM}{\tilde{M}}
\newcommand{\tildeC}{\tilde{C}}

\newcommand{\prob}{\textnormal{prob}}
\newcommand{\diff}{\textnormal{diff}}
\newcommand{\cop}{\textnormal{cop}}
\newcommand{\dist}{\textnormal{dist}}
\newcommand{\ex}{\textnormal{ex}}
\newcommand{\FLAW}{\textnormal{FLAW}}

\newcommand{\calL}{\mathcal{L}}
\newcommand{\calF}{\mathcal{F}}
\newcommand{\calH}{\mathcal{H}}
\newcommand{\calC}{\mathcal{C}}
\newcommand{\calM}{\mathcal{M}}

\newcommand{\calA}{\mathcal{A}}
\newcommand{\calB}{\mathcal{B}}
\newcommand{\calD}{\mathcal{D}}
\newcommand{\calP}{\mathcal{P}}

\newcommand{\calE}{\mathcal{E}}

\newcommand{\calR}{\mathcal{R}}

\newcommand{\tildeE}{\tilde{E}}

\def\Forb{\operatorname{Forb}}

\title{
Structure and enumeration theorems for hereditary properties in finite relational languages
}

\author{
Caroline Terry}
\date{}

\begin{document}

\maketitle

\begin{abstract}

Given a finite relational language $\calL$, a hereditary $\calL$-property is a class of finite $\calL$-structures which is closed under isomorphism and model theoretic substructure.  This notion encompasses many objects of study in extremal combinatorics, including (but not limited to) hereditary properties of graphs, hypergraphs, and oriented graphs.  In this paper, we generalize certain definitions, tools, and results form the study of hereditary properties in combinatorics to the setting of hereditary $\calL$-properties, where $\calL$ is any finite relational language with maximum arity at least two.  In particular, the goal of this paper is to generalize how extremal results and stability theorems can be combined with standard techniques and tools to yield approximate enumeration and structure theorems.  We accomplish this by generalizing the notions of extremal graphs, asymptotic density, and graph stability theorems using structures in an auxiliary language associated to a hereditary $\calL$-property.  Given a hereditary $\calL$-property $\calH$, we prove an approximate asymptotic enumeration theorem for $\calH$ in terms of its generalized asymptotic density.  Further we prove an approximate structure theorem for $\calH$, under the assumption of that $\calH$ has a stability theorem.  The tools we use include a new application of the hypergraph containers theorem (Balogh-Morris-Samotij \cite{Baloghetal1}, Saxton-Thomason \cite{saxton-thomason}) to the setting of $\calL$-structures, a general supersaturation theorem for hereditary $\calL$-properties (also new), and a general graph removal lemma for $\calL$-structures proved by Aroskar and Cummings in \cite{AroskarCummings}.

\end{abstract}

%************************************************************************
\section{Introduction}

%************************************************************************

The study of hereditary properties of combinatorial structures is an important topic within the field of extremal combinatorics.  Out of the many results in this line of research has emerged an pattern for how to prove approximate asymptotic enumeration and structure results.  The aim of this paper is to provide a general framework in which to view these results and to formalize this pattern of proof.

\subsection{Background}

A nonempty class of graphs $\mathcal{P}$ is called a \emph{hereditary graph property} if it is closed under isomorphism and induced subgraphs.  Given a hereditary graph property $\mathcal{P}$, let $\mathcal{P}_n$ denote the set of elements of $\mathcal{P}$ with vertex set $[n]$.  There has been extensive investigation into the properties of $\calP_n$, where $\calP$ is a hereditary property of graphs and $n$ is large, see for instance \cite{Alekseev1, Alekseev2, BBW1, BBW2, BBW3, BoTh1, BoTh2,ZitoSch}.  The main questions addressed in these papers concern \emph{enumeration} (finding an asymptotic formula for $|\calP_n|$) and \emph{structure} (understanding what properties elements of $\calP_n$ have with high probability).  Given a graph $H$, $\Forb(H)$ (respectively $\Forb_{ind}(H)$) is the class of finite graphs omitting $H$ as a non-induced (respectively induced) subgraph.  For any graph $H$, both $\Forb(H)$ and $\Forb_{ind}(H)$ are hereditary graph properties.  Therefore, work on hereditary graph properties can be seen as generalizing the many structure and enumeration results about graph properties of the form $\Forb(H)$ and $\Forb_{ind}(H)$, for instance those appearing in \cite{HPS, KPR, ProSte3, ProSte2, ProSte1}.  From this perspective, the study of hereditary graph properties has been a central area of research in extremal combinatorics.

There are many results which extend the investigation of hereditary graph properties to other combinatorial structures.  Examples of this include  \cite{BaBoMo1} for tournaments, \cite{BaBoMo} for oriented graphs and posets, \cite{DN, KNR} for $k$-uniform hypergraphs, and \cite{Ishigami} for colored $k$-uniform hypergraphs.  The results in \cite{BM, BM2, PersonSchacht, NR} investigate asymptotic enumeration and structure results for specific classes of $H$-free hypergraphs, which are examples of hereditary properties of hypergraphs.  Similarly, the results in \cite{OKTZ} concern specific examples of hereditary properties of digraphs.  The results in \cite{MT} for metric spaces are similar in flavor, although they have not been studied explicitly as instances of hereditary properties. Thus, extending the investigation of hereditary graph properties to other combinatorial structures has been an active area of research for many years.  

From this investigation, patterns have emerged for how to prove these kinds of results, along with a set of standard tools, including extremal results, stability\footnote{This use of the word stability refers to a type of theorem from extremal combinatorics and is unrelated to the model theoretic notion of stability.}  theorems, regularity lemmas, supersaturation results, and the hypergraph containers theorem.  In various combinations with extremal results, stability theorems, and supersaturation results, Szemer\'{e}di's regularity lemma has played a key role in proving many results in this area, especially those extending results for graphs to other settings.  A sampling of these are \cite{ABBM, BBS, ProSte2} for graphs, \cite{AlonYuster} for oriented graphs, \cite{BM, BM2, DN, Ishigami, KNR, NR, PersonSchacht} for hypergraphs, and \cite{MT} for metric spaces.   The hypergraph containers theorem, independently developed in \cite{Baloghetal1, saxton-thomason}, has been used in many recent papers in place of the regularity lemma.   Examples of this include \cite{BBNLMS, BLPS, KKOT, Stegeretal, Baloghetal1, saxton-thomason} for graphs, \cite{OKTZ} for digraphs, and \cite{BW} for metric spaces. In these papers, the commonalities in the proofs are especially clear.  Given an extremal result, there is clear outline for how to prove an approximate enumeration theorem.  If on top of this, one can characterize the extremal structures and prove a corresponding stability theorem, then there is a clear outline for how to prove an approximate structure theorem.  The goal of this paper is to make these proof outlines formal using generalizations of tools, definitions, and theorems from these papers to the setting of structures in finite relational languages.

\subsection{Summary of Results}
Given a first-order language $\mathcal{L}$, we say a class $\mathcal{H}$ of $\mathcal{L}$-structures has the \emph{hereditary property} if for all $A\in \mathcal{H}$, if $B$ is a model theoretic substructure of $A$, then $B\in \mathcal{H}$.  
\begin{definition}
Suppose $\calL$ is a finite relational language.  A  \emph{hereditary $\calL$-property} is a nonempty class of finite $\mathcal{L}$-structures which has the hereditary property and which is closed under isomorphism.  
\end{definition}
This is the natural generalization of existing notions of hereditary properties of various combinatorial structures.  Indeed, for appropriately chosen $\mathcal{L}$, almost all of the results cited so far are for hereditary $\calL$-properties, including all hereditary properties of graphs, $k$-uniform hypergraphs, colored $k$-uniform hypergraphs, directed graphs, and posets, as well as the the metric spaces from \cite{MT}.  Thus hereditary $\calL$-properties are the appropriate objects to study in order to generalize many of the results we are interested in.

We now give a description of our results.  The precise statements require extensive preliminaries and appear in Section \ref{mainresults}.  Fix a finite relational language $\calL$ with maximum arity $r\geq 2$.  Given a hereditary $\calL$-property $\calH$, we will define an invariant associated to $\mathcal{H}$, called the \emph{asymptotic density of $\calH$}, denoted by $\pi(\calH)$ (see Definition \ref{pidef}).  Our first main result, Theorem \ref{enumeration}, gives an asymptotic enumeration of $\calH_n$ in terms of $\pi(\calH)$, where $\calH_n$ denotes the set of elements in $\calH$ with domain $[n]$.  

\begin{theorem3}[Enumeration]
Suppose $\mathcal{H}$ is a hereditary $\calL$-property.  Then the following hold. 
\begin{enumerate}
\item If $\pi(\mathcal{H})>1$, then $|\mathcal{H}_n|= \pi(\mathcal{H})^{{n\choose r}+o(n^r)}$.
\item If $\pi(\mathcal{H})\leq 1$, then $|\mathcal{H}_n|=2^{o(n^r)}$.
\end{enumerate}
\end{theorem3}

\noindent The tools we use to prove this theorem include a general supersaturation theorem for $\calL$-structures (Theorem \ref{GENSUPERSAT}) and a new adaptation of the hypergraph containers theorem to the setting of $\calL$-structures (Theorems \ref{version1}).

\begin{theorem7}[informal]  A hereditary $\calL$-property always has a supersaturation theorem.
\end{theorem7}

\noindent The proof of Theorem \ref{GENSUPERSAT} uses our hypergraph containers theorem for $\calL$-structures (Theorem \ref{version1}) and a powerful generalization by Aroskar and Cummings \cite{AroskarCummings} of the triangle removal lemma (Theorem \ref{triangleremoval2}).  Our proof strategies for these theorems draw on a series of enumeration results for combinatorial structures which employ the hypergraph containers theorem, namely those in \cite{saxton-thomason, Baloghetal1, OKTZ, BW, Stegeretal}.  

We will also define generalizations of extremal graphs (Definition \ref{genexdef}) and graph stability theorems (Definition \ref{stabdef}).  We will prove that the existence of a stability theorem, along with an understanding of extremal structure, always yield an approximate structure theorem.  This result, Theorem \ref{stab}, generalizes arguments appearing in many papers, including \cite{OKTZ, BM, BM2, PersonSchacht, MT}.  

\begin{theorem5}[informal] Stability theorem + Characterization of extremal structures $\Rightarrow$ Approximate Structure.
\end{theorem5}

\noindent The main tool used to prove Theorem \ref{stab} is a second adaptation of the hypergraph containers theorem to the setting of $\calL$-structures, namely Theorem \ref{COROLLARY2}.  Our adaptations of the hypergraph containers theorem, Theorems \ref{version1} and \ref{COROLLARY2}, rely on the original hypergraph containers theorem of \cite{Baloghetal1, saxton-thomason}, the general triangle removal lemma in \cite{AroskarCummings}, and the model theoretic tools developed in this paper.

In the last main section of the paper, we consider how our results relate to theorems about discrete metric spaces.  Given integers $r\geq 3$ and $n\geq 1$, let $M_r(n)$ be the set of metric spaces with distances all in $[r]$ and underlying set $[n]$.  We will reprove structure and enumeration theorems from \cite{MT} using the machinery of this paper, along with combinatorial ingredients from \cite{MT}.  We include this example because it demonstrates interesting behavior with regards to the existence of a stability theorem.  In particular, we will prove that when $r$ is even, the hereditary property associated to $\bigcup_{n\in \mathbb{N}}M_r(n)$ has a stability theorem in the sense of Definition \ref{stabdef}, but when $r$ is odd, this is not the case.  This formalizes an intuitive difference between the even and odd cases observed in \cite{MT}.  It is important to note that our results apply to languages with relations of arity larger than two, and to structures with non-symmetric relations.  To illustrate this we have included appendices explaining how our results apply to  examples in the settings of colored hypergraphs (Appendix \ref{coloredhg}), directed graphs (Appendix \ref{dgsec}), and triangle-free hypergraphs (Appendix \ref{trifreesec}). 

We now clarify what the results in this paper do and what they do not do.  Our main theorem, Theorem \ref{enumeration}, gives an enumeration theorem for a hereditary $\calL$-property in terms of its asymptotic density.  However, determining the asymptotic density of a specific hereditary $\calL$-property is often a hard combinatorial problem which this paper does not address.  Similarly, while Theorem \ref{stab} shows that a stability theorem and an understanding of extremal structure implies an approximate structure theorem, proving a specific family $\calH$ has a stability theorem and understanding its extremal structures are often difficult problems in practice.  These difficulties are not addressed by the results in this paper.   The role of this paper is to generalize how extremal results and stability theorems give rise to approximate structure and enumeration theorems.  

While our proofs are inspired by and modeled on those appearing in \cite{saxton-thomason, Baloghetal1, OKTZ, BW, Stegeretal}, our results are more than just straightforward generalizations of existing combinatorial theorems.  We use new tools called $\calL_{\calH}$-templates (see Section \ref{tildeLstructures}) and an application of the hyergraph containers theorem to a hypergraph whose vertices and edges correspond to certain atomic diagrams (see Theorem \ref{VERSION2}). These technical tools and their appearances in our results are non-obvious and of independent interest from a model theoretic perspective.  We also provide a simplified version of the generalized triangle removal lemma appearing in \cite{AroskarCummings}, by noting that a simpler notion of the distance between $\calL$-structures may be used.

\subsection{Conclusion}

The work in this paper is significant from the perspective of combinatorics for three main reasons.  First, problems in finite combinatorics are most often approached one by one, and techniques developed for specific structures often do not translate well into other contexts.  While this style of approach is necessary to solve problems, it creates the impression that generalization is not possible.  This work serves as an example that searching to generalize results and proofs within finite combinatorics can be highly fruitful.  Second, this work will save researchers time by allowing them to avoid repeating arguments which now appear here in a general context.  Third, we believe this paper gives the correct general framework in which to view these questions, which may aid in finding answers to open problems in the area.  

This work is also of significance from the model theoretic perspective due to the following connection to logical $0$-$1$ laws.  Suppose $\calL$ is a finite language, and for each $n$, $F(n)$ is a set of $\calL$-structures with domain $[n]$.  We say $F:=\bigcup_{n\in \mathbb{N}}F(n)$ has a \emph{$0$-$1$ law} if for every first-order $\calL$-sentence $\phi$, the limit
$$
\mu(\phi):=\lim_{n\rightarrow \infty}\frac{|\{G\in F(n): G\models \phi\}|}{|F(n)|}
$$
exists and is equal to $0$ or $1$.  If $F$ has a $0$-$1$ law, then $T_{as}(F):=\{\phi:  \mu(\phi)=1\}$ is a complete consistent first-order theory.  There are many interesting model theoretic questions related to $0$-$1$ laws and almost sure theories.  For instance, it is not well understood in general why some classes of finite structure have $0$-$1$ laws and why others do not.  One source of known $0$-$1$ laws are asymptotic structure results from extremal combinatorics.  For instance, fix $s \geq 3$ and suppose for each $n$, $F(n)$ is the set of all graphs with vertex set $[n]$ omitting the complete graph $K_s$ on $s$ vertices.  In \cite{KPR} Kolaitis, Pr\"{o}mel and Rothschild show $F:=\bigcup_{n\in \mathbb{N}}F(n)$ has a $0$-$1$ law.  Their proof relies crucially on first proving asymptotic structure and enumeration theorems for $F$.  In particular, they show that if $S(n)$ is the set of $(s-1)$-partite graphs on $[n]$, then $S(n)\subseteq F(n)$ for all $n$ and
\begin{align}\label{KPRfact}
\lim_{n\rightarrow \infty}\frac{|S(n)|}{|F(n)|}=1.
\end{align}
They then prove a $0$-$1$ law for $S:=\bigcup_{n\in \mathbb{N}}S(n)$ which combines with (\ref{KPRfact}) to imply $F$ has a $0$-$1$ law.  Other  asymptotic structure results which imply $0$-$1$ laws include \cite{BM2, OKTZ, PersonSchacht, MT} (for details on how these structure results imply $0$-$1$ laws, see \cite{Koponen}).  In these papers, \cite{BM2, OKTZ, PersonSchacht, MT}, the precise structure results (which are needed to prove the $0$-$1$ laws) are proven using \emph{approximate} structure and enumeration theorems as stepping stones.  This trend suggests that a systematic understanding of precise structure and enumeration will use some version of this ``approximate version" stepping stone.  Therefore, understanding approximate structure and enumeration results from a model theoretic perspective is a necessary step in gaining a general understanding of precise structure and enumeration results, and consequently of the logical $0$-$1$ laws which rely on them.

%************************************************************************
\section{Preliminaries}\label{prelims}
%************************************************************************
Our goal here is to include enough preliminaries so that anyone with a rudimentary knowledge of first-order logic will be able to read this paper.  Definitions we expect the reader to understand include: first-order languages, constant and relation symbols, formulas, variables, structures, substructures, satisfaction, and consistency.  We refer the reader to \cite{Dave} for these definitions.

%************************************************************************
\subsection{Notation and Setup}\label{notation}
%************************************************************************

In this subsection we fix some notational conventions and definitions.  We will use the word ``collection'' to denote either a set or a class.  Suppose $\ell \geq 1$ is an integer and $X$ is a set.  Let $Perm(\ell)$ be the set of permutations of $[\ell]$.  We let $\calP(X)$ or $2^X$ denote the power set of $X$.  Given a finite tuple $\xbar=(x_1,\ldots, x_{\ell})$ and $\mu\in Perm(\ell)$, let $\cup \xbar =\{x_1,\ldots, x_{\ell}\}$, $|\xbar|=\ell$, and $\mu(\xbar)=(x_{\mu(1)},\ldots, x_{\mu(\ell)})$.  An \emph{enumeration of $X$} is a tuple $\xbar=(x_1,\ldots, x_{|X|})$ such that $\cup \xbar =X$.  Given $x\neq y\in X$, we will write $xy$ as shorthand for the set $\{x,y\}$.  Set
$$
X^{\underline{\ell}}=\{(x_1,\ldots, x_{\ell}) \in X^{\ell}: x_i\neq x_j\text{ for each }i\neq j\}\quad \hbox{ and }\quad {X\choose \ell}=\{Y\subseteq X: |Y|=\ell\}.
$$
Suppose $\mathcal{L}$ is a finite relational first-order language. Let $r_{\calL}$ denote the maximum arity of any relation symbol in $\calL$.  Given a formula $\phi$ and a tuple of variables $\xbar$, we write $\phi(\xbar)$ to denote that the free variables in $\phi$ are all in the set $\cup \xbar$.  Similarly, if $p$ is a set of formulas, we will write $p(\xbar)$ if every formula in $p$ has free variables in the set $\cup \xbar$. We will sometimes abuse notation and write $\xbar$ instead of $\cup \xbar$ when it is clear from context what is meant.  %do we actually do this?

Suppose $M$ is an $\calL$-structure.  Then $dom(M)$ denotes the underlying set of $M$, and the \emph{size} of $M$ is $|dom(M)|$.  If $\mathcal{L'}\subseteq \mathcal{L}$, $M\upharpoonright_{\mathcal{L}'}$ is the $\mathcal{L}'$-structure with underlying set $dom(M)$ such that for all $\ell\geq 1$, if $\abar \in dom(M)^{\ell}$ and $R$ is an $\ell$-ary relation symbol from $\calL'$, then $M\upharpoonright_{\mathcal{L}'}\models R(\abar)$ if and only if $M\models R(\abar)$.  We call $M\upharpoonright_{\calL'}$ the \emph{reduct of $M$ to $\calL'$}.  Given $X\subseteq dom(M)$, $M[X]$ is the $\mathcal{L}$-structure with domain $X$ such that for all $\ell\geq 1$, if $\abar\in X^{\ell}$ and $R$ is an $\ell$-ary relation symbol from $\mathcal{L}$, then $M[X]\models R(\abar)$ if and only if $M\models R(\abar)$.  We call $M[X]$ the $\mathcal{L}$-structure induced by $M$ on $X$.  Given a tuple $\abar\in dom(M)^{\ell}$, the \emph{quantifier-free type} of $\abar$ is 
$$
qftp^M(\abar)=\{\phi(x_1,\ldots, x_{\ell}): \text{$\phi(x_1,\ldots, x_{\ell})$ is a quantifier-free $\mathcal{L}$-formula and $M\models \phi(\abar)\}$}.
$$
If $\xbar=(x_1,\ldots, x_{\ell})$ and $p(\xbar)$ is a set of quantifier-free $\mathcal{L}$-formulas, then $p$ is called a \emph{quantifier-free $\ell$-type} if there is some $\mathcal{L}$-structure $N$ and a tuple $\abar\in dom(N)^{\ell}$ such that $N\models \phi(\abar)$ for all $\phi(\xbar)\in p$.  In this case we say $\abar$ \emph{realizes $p$ in $N$}.  If there is some $\abar \in dom(N)^{\ell}$ realizing $p$ in $N$, we say $p$ \emph{is realized in $N$}.  A quantifier-free $\ell$-type $p(\xbar)$ is \emph{complete} if for every quantifier-free formula $\phi(\xbar)$, either $\phi(\xbar)$ or $\neg \phi(\xbar)$ is in $p(\xbar)$.  Note that any type of the form $qftp^M(\abar)$ is complete.  All types and formulas we consider will be quantifier-free, so for the rest of the paper, any use of the words type and formula means quantifier-free type and quantifier-free formula.   

If $X$ and $Y$ are both $\calL$-structures, let $X\subseteq_{\mathcal{L}} Y$ denote that $X$ is a $\mathcal{L}$-substructure of $Y$.  Given an $\mathcal{L}$-structure $H$, we say that $M$ is $H$-free if there is no $A\subseteq_{\mathcal{L}} M$ such that $A\cong_{\mathcal{L}}H$.  Suppose $\mathcal{H}$ is a collection of $\mathcal{L}$-structures. We say $M$ is \emph{$\mathcal{H}$-free} if $M$ is $H$-free for all $H\in \mathcal{H}$.  For each positive integer $n$, let $\calH(n)$ denote the collection of all elements in $\calH$ of size $\ell$, and let $\calH_n$ denote the set of elements in $\calH$ with domain $[n]$.  $\calH$ is \emph{trivial} if there is $N$ such that $\calH(n)=\emptyset$ for all $n\geq N$.  Otherwise $\calH$ is \emph{non-trivial}.

We now define a modified version of the traditional type space, which is appropriate for working with families of finite structures instead of with complete first-order theories.  Given $\xbar=(x_1,\ldots, x_{\ell})$, an $\ell$-type $p(\xbar)$ is \emph{proper} if it contains the formulas $x_i\neq x_j$ for each $i\neq j$. 

\begin{definition}
Suppose $\mathcal{F}$ is a collection of $\mathcal{L}$-structures and $\ell \geq 1$ is an integer.  Define $S_{\ell}(\mathcal{F})$ to be the set of all complete, proper, quantifier-free $\ell$-types which are realized in some element of $\mathcal{F}$.  Let $S_{\ell}(\calL)$ denote the set of all complete, proper, quantifier-free $\ell$-types.
\end{definition}

\noindent We would like to emphasize some important differences between this and the usual type space.  First, the elements of these type spaces are proper and contain only quantifier-free formulas. Second, these type spaces are defined relative to families of finite structures instead of complete first-order theories.  

It will at times be convenient to expand our languages to contain constant symbols naming elements of the structures under consideration.  If $V$ is a set, let $C_V$ denote the set of constant symbols $\{c_v: v\in V\}$. Given $\vbar=(v_1,\ldots, v_{\ell}) \in V^{\ell}$, let $c_{\vbar}=(c_{v_1},\ldots, c_{v_{\ell}})$.  Suppose $M$ is an $\calL$-structure.  The \emph{diagram of $M$}, denoted $Diag(M)$, is the following set of sentences in the language $\calL\cup C_{dom(M)}$.
$$
Diag(M)=\{\phi(c_{\abar}): \phi(\xbar) \text{ is a quantifier-free $\calL$-formula, } \cup \abar \subseteq dom(M)\text{, and }M\models \phi(\abar)\}.
$$
\noindent If $A\subseteq dom(M)$, the \emph{diagram of $A$ in $M$} is the following set of sentences in the language $\calL\cup C_{A}$.
$$
Diag^M(A)=\{\phi(c_{\abar}): \phi(\xbar) \text{ is a quantifier-free $\calL$-formula, } \cup \abar \subseteq A\text{, and }M\models \phi(\abar)\}.
$$
Observe that if $A=\{a_1,\ldots, a_r\}\subseteq dom(M)$ and $p(\xbar)\in S_r(\calL)$ is such that $p(\xbar)=qftp^M(a_1,\ldots, a_r)$, then $Diag^M(A)=p(c_{a_1},\ldots, c_{a_r})$.  Given a set of constants $C$, a collection of $\calL$-structures $\calF$, and $\ell\geq 1$, set
$$S_{\ell}(C)=\{p(\cbar): p(\xbar)\in S_{\ell}(\calL)\text{ and } \cbar\in C^{\underline{\ell}}\}\quad \hbox{ and }\quad S_{\ell}(C, \calF)=\{p(\cbar): p(\xbar)\in S_{\ell}(\calF)\text{ and } \cbar\in C^{\underline{\ell}}\}.
$$
We would like to emphasize that if $p(\cbar)\in S_{\ell}(C)$, then $\cbar\in C^{\underline{\ell}}$ is a tuple of $\ell$ \emph{distinct} constants. Note that by this definition, if $A\in {dom(M)\choose \ell}$, then $Diag^M(A)\in S_{\ell}(C_{dom(M)})$.

%%%%%%%%%%%%%%%%%%%%%%%%%%%%%%%%%%%%%%%
\subsection{Hypergraph Containers Theorem.}\label{containerssec}
%%%%%%%%%%%%%%%%%%%%%%%%%%%%%%%%%%%%%%%

In this section we state a version of the hypergraph containers theorem, which was independently developed by Balogh-Morris-Samotij in \cite{Baloghetal1} and by Saxton-Thomason in \cite{saxton-thomason}.  The particular statement we use, Theorem \ref{containers} below, is a simplified version of Corollary 3.6 in \cite{saxton-thomason}.  We will use Theorem \ref{containers} directly in Section \ref{VERSION2pf}.  We also think it will be useful for the reader to compare it to the versions for $\calL$-structures stated in Section \ref{mainresults} (Theorem \ref{version1} and Corollary \ref{COROLLARY2}).

We begin with some definitions.  Recall that if $s\geq 2$ is an integer, an \emph{$s$-uniform hypergraph} is a pair $(V,E)$ where $V$ is a set of \emph{vertices} and $E\subseteq {V\choose s}$ is a set of \emph{edges}.  Suppose $H$ is an $s$-uniform hypergraph.  Then $V(H)$ and $E(H)$ denote the vertex and edge sets of $G$ respectively.  We set $v(H)=|V(H)|$ and $e(H)=|E(H)|$.  Given $X\subseteq V(H)$, $H[X]$ is the hypergraph $(X,E\cap{V(H)\choose s})$.  If $v(H)$ is finite, then the \emph{average degree of $H$} is $d=e(H)s/v(H)$.

\begin{definition}\label{containersdef}
Suppose $s\geq 2$, $H$ is a finite $s$-uniform hypergraph with $n$ vertices and $\tau>0$.
\begin{itemize}
\item For every $\sigma \subseteq V(H)$, the \emph{degree of $\sigma$ in $H$} is $d(\sigma) = |\{e\in E(H): \sigma \subseteq e\}|$.
\item Given $v\in V(H)$ and $j\in [s]$, set $d^{(j)}(v)=\max \{ d(\sigma): v\in \sigma \subseteq V(H), |\sigma|=j\}$.
\item If $d>0$, then for each $j\in [s]$, define $\delta_j=\delta_j(\tau)$ to satisfy the equation
$$
\delta_j \tau^{j-1}nd=\sum_{v\in V(H)}d^{(j)}(v), \qquad \hbox{ and set }\qquad \delta(H,\tau)=2^{{s\choose 2}-1}\sum_{j=2}^{s} 2^{-{j-1\choose 2}}\delta_j.
$$
If $d=0$, set $\delta(H,\tau)=0$. $\delta(H,\tau)$ is called the \emph{co-degree function}.
\end{itemize}
\end{definition}

\noindent Unless otherwise stated, $n$ is always a positive integer.

\begin{theorem}[\bf{Saxton-Thomason \cite{saxton-thomason}}] \label{containers}
Let $H$ be an $\ell$-uniform hyptergraph with a vertex set $V$ of size $n$.  Suppose $0<\epsilon, \tau<\frac{1}{2}$ and $\tau$ satisfies $\delta(H,\tau)\leq \epsilon/12s!$.  Then there exists a constant $c=c(s)$ and a collection $\calC\subseteq \calP(V)$ such that the following hold.
\begin{enumerate}[(i)]
\item For every independent set $I$ in $H$, there exists $C\in \calC$ such that $I\subseteq C$.
\item For all $C\in \mathcal{C}$, we have $e(H[C])\leq \epsilon e(G)$, and
\item $\log |\mathcal{C}| \leq c\log(1/\epsilon) n\tau \log (1/\tau)$.
\end{enumerate}
\end{theorem}

%%%%%%%%%%%%%%%%%%%%%%%%%%%%%%%%%%%%%%
\subsection{Distance between first-order structures}
%%%%%%%%%%%%%%%%%%%%%%%%%%%%%%%%%%%%%%%

In this section we define a notion of distance between finite first-order structures.  The following is based on definitions in \cite{AroskarCummings}.
\begin{definition}
Suppose $\calL$ is a first-order language, $B$ is a finite $\calL$-structure of size $\ell$, and $M$ is a finite $\calL$-structure of size $L$.  
\begin{itemize}
\item The \emph{set of copies of $B$ in $M$} is $\cop(B,M)=\{A: A\subseteq_{\calL}M \text{ and }A\cong_{\calL}B\}$.
\item The \emph{induced structure density of $B$ in $M$} is $\prob(B,M)=|\cop(B,M)|/{L\choose \ell}$
\item If $\mathcal{B}$ is a set of finite $\calL$-structures, let 
$$
\cop(\mathcal{B}, M)=\bigcup_{B\in \mathcal{B}}\cop(B,M)\qquad \hbox{ and }\qquad \prob(\mathcal{B},M)=\max\{p(B,M): B\in \mathcal{B}\}.
$$
\end{itemize}
If $\calB$ is a \emph{class} of finite $\calL$-structures, define $\cop(\calB, M)=\cop(\calB',M)$ and $\prob(\calB,M)=\prob(\calB',M)$, where $\calB'$ is any set containing one representative of each isomorphism type in $B$. 
\end{definition}

We now state our definition for the distance between two finite first-order structures.  It is a simplified version of the distance notion appearing in \cite{AroskarCummings}.  We will discuss the relationship between the two notions in Section \ref{rphrem}.

\begin{definition}\label{deltaclosedef1}
Let $\mathcal{L}$ be a finite relational first-order language with $r_{\ell}=r\geq 2$.  Suppose $M$ and $N$ are two finite $\mathcal{L}$-structures with the same underlying set $V$ of size $n$.  Let
\begin{align*}
\diff(M,N)&=\{A\in {V\choose r}: \text{ for some enumeration $\abar$ of $A$, }qftp^M(\abar)\neq qftp^N(\abar)\}\text{ and }\\
\dist(M,N)&=\frac{|\diff(M,N)|}{{n\choose r}}
\end{align*}
We say that $M$ and $N$ are \emph{$\delta$-close} if $\dist(M,N)\leq \delta$.
\end{definition}

\noindent Observe that in the notation of Definition \ref{deltaclosedef1}, $\diff(M,N)=\{A\in {V\choose r}: Diag^M(A)\neq Diag^N(A)\}$.  

%%%%%%%%%%%%%%%%%%%%%%%%%%%%%%%%%%%%%%%
\subsection{Facts about hereditary properties}
%%%%%%%%%%%%%%%%%%%%%%%%%%%%%%%%%%%%%%%
Suppose $\calL$ is a finite relational language.  In this subsection we state some well known facts about hereditary $\calL$-properties.  First we recall that hereditary $\calL$-properties are the same as families of structures with forbidden configurations.  This fact will be used throughout the chapter.

\begin{definition}
If $\mathcal{F}$ is a collection of finite $\mathcal{L}$-structures, let $\Forb(\calF)$ be the class of all finite $\calL$-structures which are $\calF$-free.  
\end{definition}
It is easy to check that for any collection $\calF$ of finite $\calL$-structures, $\Forb(\calF)$ is a hereditary $\calL$-property.  The converse to this statement is also true in the sense of Observation \ref{HP} below.  This fact is standard, but we include a proof for completeness.  

\begin{observation}\label{HP}
If $\calH$ is a hereditary $\calL$-property, then there is a class of finite $\mathcal{L}$-structures $\mathcal{F}$ which is closed under isomorphism and such that $\calH=\Forb(\calF)$.
\end{observation}
\begin{proof}
Let $\calF$ be the class of all finite $\calL$-structures $F$ such that $\prob(F, \calH)=0$.  Clearly $\calF$ is closed under $\calL$-isomorphism.  We show $\calH=\Forb(\calF)$.  Suppose $M\in \calH$ but $M\notin \Forb(\calF)$. Then there is some $F'\subseteq_{\calL}M$ and $F\in \calF$ such that $F\cong_{\calL}F'$.  Since $\calF$ is closed under $\calL$-isomorphism, $F'\in \calF$.  But then $M\in \calH$ and $F'\subseteq_{\calL}M$ implies $F'\notin \calF$ by definition of $\calF$, a contradiction. Conversely, suppose $M\in \\Forb(\calF)$ but $M\notin \calH$.  Because $M$ is $\calF$-free, we must have $M\notin \calF$.  By definition of $\calF$, this implies there is some $M'\in \calH$ such that $M\subseteq_{\calL}M'$.  Because $\calH$ has the hereditary property, this implies $M\in \calH$, a contradiction.
\end{proof}

A sentence $\phi$ is universal if it is of the form $\forall \xbar \psi(\xbar)$ where $\psi(\xbar)$ is quantifier-free.  The following well known fact is another reason hereditary $\calL$-properties are natural objects of study.

\begin{observation}
$\calH$ is a hereditary $\calL$-property if and only if there is a set of universal sentences $\Phi$ such that $\calH$ is the class of all finite models of $\Phi$.
\end{observation}

\begin{comment}

\begin{proof}
Suppose first $\calH$ is a hereditary $\calL$-property.  Let $\calF$ be as in Observation \ref{HP} such that $\calH=\Forb(\calF)$.  Let $\{F_i: i\in \omega\}$ contain one representative of each isomorphism class in $\calF$, and let $\theta_i$ be the sentence saying $\forall \xbar (\xbar \ncong_{\calL} F_i)$.  Then clearly $\calH=\Forb(\calF)$ is the class of all finite models of $\Phi:=\{\theta_i: i\in \omega\}$.  Conversely, suppose $\Phi$ is a set of universal sentences and $\calH$ is the class of all finite models of $\Phi$.  Clearly $\calH$ is closed under isomorphism.  Suppose now $B\in \calH$ and $A\subseteq_{\calL}B$.  Let $\phi \in \Phi$.  Then $B\models \phi$ by assumption.  Since $\phi$ is universal, it is preserved under substructures, so $A\models \phi$.  Thus $A\models \phi$ for all $\phi \in \Phi$, so $A\in \calH$.  This shows $\calH$ has the hereditary property, and thus is a hereditary $\calL$-property.
\end{proof}
\end{comment}

The proof of this is straightforward using Observation \ref{HP}, the fact that $\Forb(\calF)$ can be axiomatized using universal sentences for any class $\calF$ of finite $\calL$-structures, and the fact that universal sentences are preserved under taking substructures (see the \L os-Tarski Theorem in \cite{hodges}).   

%************************************************************************
\section{$\calL_{\calH}$-structures}\label{tildeLstructures}
%************************************************************************\chi
{\bf From now on, $\mathcal{L}$ is a fixed finite relational language and $r:=r_{\calL}\geq 2$}.   For this section, $\calH$ is a nonempty collection of finite $\calL$-structures.  In this section we introduce a language $\calL_{\calH}$ associated to $\mathcal{L}$ and $\calH$.  Structures in this new language play key roles in our main theorems.  

\begin{definition}\label{LHdef}
Define $\calL_{\calH}=\{R_p(\xbar): p(\xbar)\in S_r(\calH)\}$ to be the relational language with one $r$-ary relation for each $p(\xbar)$ in $S_r(\calH)$.  
\end{definition}

The goal of this section is to formalize how an $\calL_{\calH}$-structure $M$ with the right properties can serve as a ``template'' for building $\calL$-structures with the same underlying set as $M$. We now give an example of a hereditary property and its corresponding auxiliary language as in Definition \ref{LHdef}.

\begin{example}\label{ex1}
To avoid confusion, we will use $\calP$ to refer to specific hereditary properties in example settings.  Suppose $\calL=\{R_1(x,y), R_2(x,y), R_3(x,y)\}$.  Let $\calP$ be the class of all finite metric spaces with distances in $\{1,2,3\}$, considered as $\calL$-structures in the natural way (i.e. $R_i(x,y)$ if and only if $d(x,y)=i$).  It is easy to see that $\calP$ is a hereditary $\calL$-property.  Observe that since $r_{\calL}=2$, $\calL_{\calP}=\{R_p(x,y): p(x,y)\in S_2(\calP)\}$.  For each $i\in [3]$, set  
$$
q_i(x,y):=\{x\neq y\}\cup \{R_i(x,y), R_i(y,x)\}\cup \{\neg R_j(x,y), \neg R_j(y,x): j\neq i\},
$$
and let $p_i(x,y)$ be the unique quantifier-free $2$-type containing $q_i(x,y)$. Informally, the type $p_i(x,y)$ says ``the distance between $x$ and $y$ is equal to $i$.''  We leave it as an exercise to the reader that $S_2(\calP)=\{p_i(x,y): i\in [3]\}$ (recall $S_2(\calP)$ consists of \emph{proper} types).  Thus $\calL_{\calP}=\{R_{p_i}(x,y): i\in [3]\}$.   
\end{example}

Observe that in an arbitrary $\calL_{\calH}$-structure, the relation symbols in $\calL_{\calH}$ may have nothing to do with the properties of the type space $S_r(\calH)$.  For instance, in the notation of Example \ref{ex1}, we can easily build an $\calL_{\calP}$-structure $M$ so that for some $a,b\in dom(M)$, $M\models R_{p_1}(a,b) \wedge \neg R_{p_1}(b,a)$, even though $p_1(x,y)=p_1(y,x)$ in $S_2(\calP)$.  This kind of behavior will be undesirable for various technical reasons.  We now define the class of $\calL_{\calH}$-structures which are most nicely behaved for our purposes, and where in particular, this bad behavior does not happen.

\begin{definition}
An $\calL_{\calH}$-structure $M$ with domain $V$ is \emph{complete} if for all $A\in {V\choose r}$ there is an enumeration $\abar$ of $A$ and $R_p\in \calL_{\calH}$ such that $M\models R_p(\abar)$.
\end{definition}

\begin{definition}\label{templatedef}
An $\calL_{\calH}$-structure $M$ with domain $V$ is an \emph{$\calL_{\calH}$-template} if it is complete and the following hold.  
\begin{enumerate}
\item If $p(\xbar)\in S_r(\calH)$ and $\abar\in V^r\setminus V^{\underline{r}}$, then $M\models \neg R_p(\abar)$.  
\item If $p(\xbar), p'(\xbar)\in S_r(\calH)$ and $\mu\in Perm(r)$ are such that $p(\xbar)=p'(\mu(\xbar))$, then for every $\abar\in V^{\underline{r}}$, $M\models R_p(\abar)$ if and only if $M\models R_{p'}(\mu(\abar))$. 
\end{enumerate}
\end{definition}

\noindent The idea is that $\calL_{\calH}$-templates are the $\calL_{\calH}$-structures which most accurately reflect the properties of $S_r(\calH)$.

\begin{example}\label{ex2}
Let $\calL$ and $\calP$ be as in Example \ref{ex1}.  Let $G$ be the $\calL_{\calP}$-structure with domain $V=\{u,v,w\}$ such that $G\models \bigwedge_{i=1}^3 (R_{p_i}(u,v) \wedge R_{p_i}(v,u))$,
\begin{align*}
G&\models \neg R_{p_3}(w,v) \wedge \neg R_{p_3}(v,w) \wedge \bigwedge_{i=1}^2 (R_{p_i}(w,v) \wedge R_{p_i}(v,w)),\\
G&\models R_{p_1}(w,u) \wedge R_{p_1}(w,u) \wedge \bigwedge_{i=2}^3 (\neg R_{p_i}(w,u) \wedge \neg R_{p_i}(u,w)),
\end{align*}
and for all $x\in V$, $G\models \bigwedge_{i=1}^3 \neg R_{p_i}(x,x)$. We leave it to the reader to verify $G$ is a $\calL_{\calP}$-template.
\end{example}

While $\calL_{\calH}$-templates are important for the main results of this paper, many of the definitions and facts in the rest of this section will be presented for $\calL_{\calH}$-structures with weaker assumptions.  

\subsection{Choice functions and subpatterns}
In this subsection, we give crucial definitions for how we can use $\calL_{\calH}$-structures to build $\calL$-structures.

\begin{definition}\label{chdef}
Suppose $M$ is an $\calL_{\calH}$-structure with domain $V$. 
\begin{enumerate}
\item Given $A\in {V\choose r}$, the \emph{set of choices for $A$ in $M$} is
$$
Ch_M(A)=\{p(c_{a_1},\ldots, c_{a_r})\in S_r(C_V,\calH): \{a_1,\ldots, a_r\}=A \text{ and }M\models R_p(a_1,\ldots, a_r)\}.
$$
\item A \emph{choice function for $M$} is a function $\chi: {V\choose r}\rightarrow  S_r(C_V,\calH)$ such that for each $A\in {V\choose r}$, $\chi(A)\in Ch_M(A)$. Let $Ch(M)$ denote the set of all choice functions for $M$.
\end{enumerate}
\end{definition}

\noindent In the notation of Definition \ref{chdef}, note $Ch(M)\neq \emptyset$ if and only if $Ch_M(A)\neq \emptyset$ for all $A\in {V\choose r}$.  Observe that $Ch_M(A)\neq \emptyset$ for all $A\in {V\choose r}$ if and only if $M$ is complete.  Therefore $Ch(M)\neq \emptyset$ if and only if $M$ is complete.  

\begin{example}\label{ex3}
Recall that if $x$ and $y$ are distinct elements of a set, then $xy$ is shorthand for the set $\{x,y\}$.  Let $\calL$, $\calP$, $V$, and $G$ be as in Example \ref{ex2}.  Note that $C_V=\{c_u, c_v,c_w\}$ and  
$$
S_2(C_V, \calP)= \{p_i(c_u,c_v):  i\in [3]\}\cup \{p_i(c_v,c_w):  i\in [3]\}\cup \{p_i(c_u,c_w):  i\in [3]\}.
$$
By definition of $G$, $Ch_G(uv)=\{p_1(c_u,c_v), p_2(c_u,c_v), p_3(c_u,c_v)\}$, $Ch_G(vw)=\{p_1(c_v,c_w), p_2(c_v,c_w)\}$, and $Ch_G(uw)=\{p_1(c_u,c_w)\}$.  Therefore $Ch(G)$ is the set of functions 
$$
\chi:\{uv, vw, uw\}\rightarrow \{p_i(c_u,c_v):  i\in [3]\}\cup \{p_i(c_v,c_w):  i\in [3]\}\cup \{p_i(c_u,c_w):  i\in [3]\}
$$
with the properties that $\chi(uv)\in \{p_1(c_u,c_v), p_2(c_u,c_v), p_3(c_u,c_v)\}$, $\chi(vw)\in \{p_1(c_v,c_w), p_2(c_v,c_w)\}$ and $\chi(uw)=p_1(c_u,c_w)$.  Clearly this shows $|Ch(G)|=|Ch_G(uv)||Ch_G(vw)||Ch_G(uw)|=6$.
\end{example}

The following observation is immediate from the definition of $\calL_{\calH}$-template.

\begin{observation}\label{temob}
If $M$ is an $\calL_{\calH}$-template with domain $V$, then for all $\abar\in V^r$ and $R_p\in \calL_{\calH}$, $M\models R_p(\abar)$ if and only if $|\cup \abar|=r$ and $p(c_{\abar})\in Ch_M(\cup \abar)$.
\end{observation}

\noindent The following fact is one reason why $\calL_{\calH}$-templates are convenient.

\begin{proposition}\label{nice}
Suppose $M_1$ and $M_2$ are $\calL_{\calH}$-templates with domain $V$ such that for all $A\in {V\choose r}$, $Ch_{M_1}(A)=Ch_{M_2}(A)$.  Then $M_1$ and $M_2$ are the same $\calL_{\calH}$-structure.
\end{proposition}
\begin{proof}
We show that for all $\abar\in V^r$ and $R_p\in \calL_{\calH}$, $M_1\models R_p(\abar)$ if and only if $M_2\models R_p(\abar)$.  Fix $\abar \in V^r$ and $R_p\in \calL_{\calH}$.  Suppose first that $|\cup \abar|<r$. By part (1) of Definition \ref{templatedef}, $M_1\models \neg R_p(\abar)$ and $M_2\models \neg R_p(\abar)$.  So assume $|\cup \abar|=r$.  By Observation \ref{temob}, $M_1\models R_p(\abar)$ if and only if $p(c_{\abar})\in Ch_{M_1}(\cup \abar)$ and $M_2\models R_p(\abar)$ if and only if $p(c_{\abar})\in Ch_{M_2}(\cup \abar)$.  Since $Ch_{M_1}(\cup \abar)=Ch_{M_2}(\cup \abar)$, this implies $M_1\models R_p(\abar)$ if and only if $M_2\models R_p(\abar)$.
\end{proof}

\noindent The next example shows Proposition \ref{nice} can fail when we are not dealing with $\calL_{\calH}$-templates.

\begin{example}\label{ex4}
Let $\calL$, $\calP$, $V$, and $G$ be as in Example \ref{ex3}.  Let $G'$ be the $\calL_{\calP}$-structure with domain $V$ which agrees with $G$ on $V^2\setminus \{(v,u), (w,v), (w,u)\}$ and where 
$$
G'\models \bigwedge_{i=1}^3 (\neg R_{p_i}(v,u)\wedge \neg R_{p_i}(w,v)\wedge \neg R_{p_i}(w,u)).
 $$
We leave it to the reader to check that for all $xy\in {V\choose 2}$, $Ch_{G'}(xy)=Ch_G(xy)$.  However, $G$ and $G'$ are distinct $\calL_{\calP}$-structures because, for instance, $G\models R_{p_1}(v,u)$ while $G'\models \neg R_{p_1}(v,u)$.  Observe that $G'$ is not an $\calL_{\calP}$-template because $G'\models R_{p_1}(u,v)\wedge \neg R_{p_1}(v,u)$ while $p_1(x,y)=p_1(y,x)$.
\end{example}

\noindent The next definition shows how choice functions give rise to $\calL$-structures.

\begin{definition}
Suppose $M$ is a complete $\calL_{\calH}$-structure with domain $V$, $N$ is an $\mathcal{L}$-structure such that $dom(N)\subseteq V$, and $\chi\in Ch(M)$.
\begin{enumerate}
\item $N$ is a \emph{$\chi$-subpattern} of $M$, denoted $N\leq_{\chi}M$, if for every $A\in {dom(N)\choose r}$, $\chi(A)=Diag^N(A)$.  
\item $N$ is a \emph{full $\chi$-subpattern of $M$}, denoted $N\unlhd_{\chi} M$, if $N\leq_{\chi}M$ and $dom(N)=V$.  
\end{enumerate}
When $N\unlhd_{\chi} M$, we say $\chi$ \emph{chooses $N$}.  We say $N$ is a \emph{subpattern of $M$}, denoted $N\leq_p M$, if $N\leq_{\chi}M$ for some choice function $\chi$ for $M$.  We say $N$ is a \emph{full subpattern of $M$}, denoted $N\unlhd_p M$, if $N\unlhd_{\chi}M$ for some choice function $\chi$ for $M$.  The subscript in $\leq_p$ and $\unlhd_p$ is for ``pattern.''
\end{definition}

\begin{observation}\label{ob00}
Suppose $M$ is a complete $\calL_{\calH}$-structure, $\chi\in Ch(M)$, and $G$ is an $\calL$-structure such that $G\unlhd_{\chi}M$.  If $G'$ is another $\calL$-structure such that $G'\unlhd_{\chi}M$, then $G$ and $G'$ are the same $\calL$-structure.  If $\chi'\in Ch(M)$ satisfies $G\unlhd_{\chi'}M$, then $\chi=\chi'$.
\end{observation}
\begin{proof}
By definition, $G\unlhd_{\chi}M$ and $G'\unlhd_{\chi}M$ imply that $Diag^G(A)=\chi(A)=Diag^{G'}(A)$ for all $A\in {V\choose r}$.  This implies 
$$
Diag(G)=\bigcup_{A\in {V\choose r}}Diag^G(A)=\bigcup_{A\in {V\choose r}}Diag^{G'}(A)=Diag(G'),
$$
which clearly implies $G$ and $G'$ are the same $\calL$-structure.  Similarly, $G\unlhd_{\chi}M$ and $G\unlhd_{\chi'}M$ imply that for all $A\in {V\choose r}$, $\chi(A)=Diag^G(A)=\chi'(A)$.  Thus $\chi=\chi'$.
\end{proof}

\begin{example}\label{ex3}
Let $\calL$, $\calP$, $V$ and $G$ be as in Example \ref{ex2}.  We give two examples of subpatterns of $G$.  Let $\chi$ be the function from ${V\choose 2}\rightarrow S_2(C_V,\calP)$ defined by $\chi(uv)=p_1(c_u,c_v)$, $\chi(vw)=p_2(c_v,c_w)$, and $\chi(uw)=p_1(c_u,c_w)$. Clearly $\chi$ is a choice function for $G$.  Let $H$ be the $\calL$-structure with domain $V$ such that $H\models p_1(u,v)\cup p_2(v,w)\cup p_1(u,w)$.  Then by definition of $H$, $Diag^H(uv)=p_1(c_u,c_v)$, $Diag^H(vw)=p_2(v,w)$ and $Diag^H(uw)=p_1(c_u,c_w)$.  In other words, $H\leq_{\chi}G$.  Since $dom(H)=dom(G)=V$, $H\unlhd_{\chi} G$.  Note that $H$ is a metric space, that is, $H\in \calP$.  

Let $\chi'$ be the function from ${V\choose 2}\rightarrow S_2(C_V,\calP)$ defined by $\chi'(uv)=p_3(c_u,c_v)$, $\chi'(vw)=p_1(c_v,c_w)$, and $\chi'(uw)=p_1(c_u,c_w)$.  Clearly $\chi'$ is a choice function for $G$.  Let $H'$ be the $\calL$-structure with domain $V$ such that $H'\models p_3(u,v)\cup p_1(v,w)\cup p_1(u,w)$.  Then as above, it is easy to see that $H'\leq_{\chi'} G$, and since $dom(H')=V$, $H'\unlhd_{\chi'} G$.  However, $H'$ is \emph{not} a metric space, that is, $H'\notin \calP$.
\end{example}

This example demonstrates although $\calL_{\calH}$-templates are well behaved in certain ways, an $\calL_{\calH}$-template can have full subpatterns that are not in $\calH$.  We will give further definitions to address this in Section \ref{temsec}.

\subsection{Errors and counting subpatterns}  In this subsection we characterize when an $\calL_{\calH}$-structure has the property that every choice function gives rise to a subpattern.  This will be important for counting subpatterns of $\calL_{\calH}$-structures.

\begin{definition}
Given $r< \ell < 2r$, an \emph{error of size $\ell$} is a complete $\calL_{\calH}$-structure $M$ of size $\ell$ with the following properties.  There are $\abar_1, \abar_2\in dom(M)^{\underline{r}}$ such that $dom(M)=\cup \abar_1\bigcup \cup \abar_2$ and for some $p_1(\xbar), p_2(\xbar)\in S_r(\calH)$, $M\models R_{p_1}(\abar_1)\wedge R_{p_2}(\abar_2)$ but $p_1(c_{\abar_1})\cup p_2(c_{\abar_2})$ is unsatisfiable.
\end{definition}

\begin{example}\label{errorex}
Let $\calL=\{E(x,y,z), R_1(x,y), R_2(x,y), R_3(x,y)\}$ consist of one ternary relation $E$ and three binary relations $R_1,R_2,R_3$.  Suppose $\calP$ is the class of all finite $\calL$-structures $M$ such that the restriction of $M$ to $\{R_1,R_2,R_3\}$ is a metric space with distances in $\{1,2,3\}$ (we put no restrictions on how $E$ must behave). For $i\in [3]$, let $p_i(x,y)$ be the quantifier-free $\{R_1,R_2,R_3\}$-type from Examples \ref{ex1}, \ref{ex2}, and \ref{ex3}.  Let $\xbar=(x_1,x_2,x_3)$ and set $q_0(\xbar)=\{E(x_i,x_j,x_k): \{x_i,x_j,x_k\}\subseteq \{x_1,x_2,x_3\}\}$.  Then set $q_1(\xbar)$ and $q_1(\xbar)$ to be the complete quantifier-free types satisfying the following.
\begin{align*}
q_0(\xbar)\cup p_1(x_1,x_2)\cup p_1(x_1,x_3)\cup p_1(x_2,x_3)&\subseteq q_1(\xbar)\text{ and }\\
q_0(\xbar)\cup p_2(x_1,x_2)\cup p_1(x_1,x_3)\cup p_1(x_2,x_3)&\subseteq q_2(\xbar).
\end{align*}
It is easy to check that $q_i(\xbar)\in S_3(\calP)$ for $i=1,2$.  Note $q_1$ and $q_2$ agree about how $E$ behaves, but disagree on how the binary relations in $\calL$ behave. Let $V=\{t,u,v,w\}$ be a set of size $4$.  Choose $G$ to be the $\calL_{\calP}$-structure which satisfies $G\models R_{q_1}(x,y,z)$ if and only if $x$, $y$ and $z$ are distinct, $G\models R_{q_2}(x,y,z)$ if and only if $(x,y,z)=(t,u,v)$, and $G\models \neg R_{q}(x,y,z)$ for all $q\in S_3(\calP)\setminus \{p_1,p_2\}$.  

By construction, $G$ is a complete $\calL_{\calP}$-structure.  Let $\abar_1=(u,v,w)$ and $\abar_2=(u,v,t)$.  Then $dom(G)=\cup \abar_1\bigcup \cup \abar_2$ and $G\models R_{q_1}(\abar_1)\wedge R_{q_2}(\abar_2)$.  However, $p_1(c_u,c_v)\subseteq q_1(c_u,c_v,c_w)=q_1(c_{\abar_1})$ implies $q_1(c_{\abar_1})$ contains the formula $R_1(c_u,c_v)$ while $p_2(c_u,c_v)\subseteq q_2(c_u,c_v,c_w)=q_2(c_{\abar_2})$ implies $q_2(c_{\abar_2})$ contains the formula $\neg R_1(c_u,c_v)$ .  Therefore $q_1(c_{\abar_1})\cup q_2(c_{\abar_2})$ is unsatisfiable, and $G$ is an error of size $4$.
\end{example}

\begin{definition}
Let $\mathcal{E}$ be the class of $\calL_{\calH}$-structures which are errors of size $\ell$ for some $r<\ell<2r$.  
\end{definition} 

An $\calL_{\calH}$-structure $M$ is \emph{error-free} if it is $\calE$-free.  Error-free $\calL_{\calH}$-structures will be important for the following reason.

\begin{proposition}\label{Lrandom}
Suppose $M$ is a complete $\calL_{\calH}$-structure with domain $V$.  Then $M$ is error-free if and only if for every $\chi\in Ch(M)$, there is an $\calL$-structure $N$ such that $N\unlhd_{\chi}M$. 
\end{proposition}

\begin{proof}
Suppose first that there exists a choice function $\chi: {V\choose r}\rightarrow S_r(C_V,\calH)$ such that there are no $\chi$-subpatterns of $M$. This means $\Gamma :=\bigcup_{A\in {V\choose r}} \chi(A)$ is not satisfiable.  So there is an atomic formula $\psi(\xbar)$ and a tuple $\cbar\subseteq C_V^r$ such that $\psi(\cbar)\in \Gamma$ and $\neg \psi(\cbar)\in \Gamma$. For each $A\in {V\choose r}$, because $\chi(A)\in S_r(C_A,\calH)$, exactly one of $\psi(\cbar)$ or $\neg\psi(\cbar)$ is in $\chi(A)$.  This implies  there must be distinct $A_1$, $A_2 \in {V\choose r}$ such that $\cup \cbar\subseteq A_1\cap A_2$ and $\psi(\cbar)\in \chi(A_1)$ and $\neg \psi(\cbar)\in \chi(A_2)$.  Note $A_1\neq A_2$ and $A_1\cap A_2\neq \emptyset$ imply that if $\ell:=|A_1\cup A_2|$, then $r<\ell<2r$.  Let $N$ be the $\calL_{\calH}$-structure $M[A_1\cup A_2]$.  We show $N$ is an error of size $\ell$.  By definition of $\chi$ being a choice function, there are $p_1,p_2\in S_r(C_V,\calH)$ and $\abar_1$, $\abar_2$ such that $\cup \abar_1=A_1$, $\cup \abar_2=A_2$, $p_1(c_{\abar_1})=\chi(A_1)$, $p_2(c_{\abar_2})=\chi(A_2)$, $M\models R_{p_1}(\abar_1)$, and $M\models R_{p_2}(\abar_2)$.  By definition, $N\subseteq_{\calL_{\calH}}M$, thus $N\models R_{p_1}(\abar_1)\wedge R_{p_2}(\abar_2)$.  Note 
$$
\{\psi(\cbar),\neg\psi(\cbar)\}\subseteq \chi(A_1)\cup \chi(A_2)=p_1(c_{\abar_1})\cup p_2(c_{\abar_2})
$$
implies $p_1(c_{\abar_1})\cup p_2(c_{\abar_2})$ is unsatisfiable.  Thus $N\in \calE$ and $N\subseteq_{\calL_{\calH}}M$ implies $M$ is not error-free.

Suppose on the other hand that $M$ is not error-free.  Say $r< \ell< 2r$ and $N$ is an error of size $\ell$ in $M$.  Then $N\subseteq_{\calL_{\calH}}M$ and there are $\abar_1, \abar_2\in dom(N)^{\underline{r}}$ and types $p_1(\xbar), p_2(\xbar)\in S_r(\calH)$ such that $dom(N)=\cup \abar_1\bigcup \cup \abar_2$,  $N\models R_{p_1}(\abar_1)\wedge R_{p_2}(\abar_2)$, and $p_1(c_{\abar_1})\cup p_2(c_{\abar_2})$ is unsatisfiable.  We define a function $\chi: {V\choose r}\rightarrow S_r(C_V,\calH)$ as follows.  Set $\chi(\cup \abar_1)=p_1(c_{\abar_1})$ and $\chi(\cup\abar_2)=p_2(c_{\abar_2})$.  For every $A'\in {V\choose r}\setminus \{A_1,A_2\}$, choose $\chi(A')$ to be any element of $Ch_M(A')$ (note $Ch_M(A')$ is nonempty since $M$ is complete).  By construction, $\chi$ is a choice function for $M$.  Suppose there is $G\unlhd_{\chi} M$.  Then $p_1(c_{\abar_1})=Diag^G(\cup \abar_1)$ and $p_2(c_{\abar_2})=Diag^G(\cup \abar_2)$ implies $G\models p_1(\abar_1)\cup p_2(\abar_2)$, contradicting that $p_1(c_{\abar_1})\cup p_2(c_{\abar_2})$ is unsatisfiable.  Thus $\chi\in Ch(M)$ and there are no $\chi$-subpatterns of $M$.  This finishes the proof. 
\end{proof}

\begin{definition}
Given a finite $\calL_{\calH}$-structure $M$, let $sub(M)=|\{G: G\unlhd_pM\}|$ be the number of full subpatterns of $M$.  
\end{definition}
This definition and the following observation will be crucial to our enumeration theorem.
\begin{observation}\label{ob0}
If $M$ is a complete $\calL_{\calH}$-structure with finite domain $V$, then 
\begin{align*}
sub(M)\leq \prod_{A\in {V\choose r}}|Ch_M(A)|,
\end{align*}
and equality holds if and only if $M$ is error-free.
\end{observation}
\begin{proof}
By definition of a choice function, $|Ch(M)|= \prod_{A\in {V\choose r}}|Ch_M(A)|$.  By definition of subpattern, for each $G\unlhd_pM$, there is $\chi_G\in Ch(M)$ which chooses $G$.  Observation \ref{ob00} implies the map $f:G\mapsto \chi_G$ is a well-defined injection from $\{G: G\unlhd_p M\}$ to $Ch(M)$. Thus 
$$
sub(M)=|\{G: G\unlhd_p M\}|\leq |Ch(M)|=\prod_{A\in {V\choose r}}|Ch_M(A)|.
$$
We now show equality holds if and only if $M$ is error-free.  Suppose first $M$ is error-free.  We claim $f$ is surjective.  Fix $\chi\in Ch(M)$.  Since $M$ is error-free, Lemma \ref{Lrandom} implies that there is an $\calL$-structure $G_{\chi}$ such that $G_{\chi}\unlhd_{\chi} M$.  So $G_{\chi}\in \{G: G\unlhd_pM\}$ implies $f(G_{\chi})$ exists.  By Observation \ref{ob00}, we must have $f(G_{\chi})=\chi$.  Thus $f$ is surjective, and consequently $sub(M)= |Ch(M)|=\prod_{A\in {V\choose r}}Ch_M(A)$.  Conversely, suppose equality holds.  Then $f$ is an injective map from a finite set to another finite set of the same size, thus it must be surjective.  This implies that for all $\chi\in Ch(M)$, there is an $\calL$-structure $G$ such that $G\unlhd_{\chi}M$.  By Lemma \ref{Lrandom}, this implies $M$ is error-free.
\end{proof}

\begin{remark}\label{specialob0}
Suppose $\calL$ contains no relations of arity less than $r$. If $\calM$ is a complete $\calL_{\calH}$-structure with finite domain $V$, then $sub(M)=\prod_{A\in {V\choose r}}|Ch_M(A)|$.
\end{remark}
\begin{proof}
Our assumption on $\calL$ implies $\calE=\emptyset$ by definition.  Thus, if $\calM$ is a complete $\calL_{\calH}$-structure with finite domain $V$, it is error-free, so Observation \ref{ob0} implies $sub(M)=\prod_{A\in {V\choose r}}|Ch_M(A)|$.
\end{proof}

\noindent Remark \ref{specialob0} applies to most examples we are interested in, including graphs, (colored) $k$-uniform hypergraphs for any $k\geq 2$, directed graphs, and discrete metric spaces.

%%%%%%%%%%%%%%%%
\subsection{$\calH$-random $\calL_{\calH}$-structures and $\calL_{\calH}$-templates}\label{temsec}
%%%%%%%%%%%%%%
In this subsection we consider $\calL_{\calH}$-structures with the property that all choice functions give rise to subpatterns in $\calH$.

\begin{definition}\label{Hrand}
 An $\calL_{\calH}$-structure $M$ is \emph{$\mathcal{H}$-random} if it is complete and for every $\chi\in Ch(M)$, there is an $\mathcal{L}$-structure $N\in \calH$ such that $N\unlhd_{\chi} M$.  
\end{definition}  

Observe that by Proposition \ref{Lrandom}, any $\calH$-random $\calL_{\calH}$-structure is error-free.  The difference between being error-free and being $\calH$-random is as follows.  If an $\calL_{\calH}$-structure is error-free, then it must have at least one full subpattern, however some or all its subpatterns may not be in $\calH$.  On the other hand, if an $\calL_{\calH}$-structure is $\calH$-random, then it must have at least one full subpattern, and further, all its full subpatterns must also be in $\calH$. 

\begin{example}\label{ex4}
Let $\calL$, $\calP$, $V$, $H'$ and $G$ be as in Example \ref{ex3}.  Observe that $G$ is \emph{not} $\calP$-random, since $H'\unlhd_pG$, but $H'\notin \calP$.  We now define an $\calL_{\calP}$-structure $G''$ which is $\calP$-random.  Let $G''$ be any $\calL_{\calP}$-structure with domain $V$ such that for all $(x,y)\in V^{\underline{2}}$,  
$$
G''\models R_{p_1}(x,y)\wedge R_{p_2}(x,y)\wedge \neg R_{p_3}(x,y)\wedge \bigwedge_{i=1}^3\neg R_{p_i}(x,x).
$$
It is easy to check that $G''$ is a $\calL_{\calP}$-template and for all $xy\in {V\choose 2}$, $Ch_{G''}(xy)=\{p_1(c_x,c_y), p_2(c_x,c_y)\}$.  Suppose $\chi''\in Ch(G'')$.  Then for all $xy\in {V\choose 2}$, $\chi''(xy)\in \{p_1(c_x,c_y), p_2(c_x,c_y)\}$.  Let $M$ be the $\calL$-structure such that $M\unlhd_{\chi''}G''$, that is, $dom(M)=V$ and for each $xy\in {V\choose 2}$, $M\models p_i(x,y)$ if and only if $\chi''(xy)=p_i(c_x,c_y)$.  Then for all $xy\in {V\choose 2}$, $M\models p_1(x,y)$ or $M\models p_2(x,y)$.  Since there is no way to violate the triangle inequality using distances in $\{1,2\}$, $M$ is a metric space.  Thus we have shown that for every $\chi''\in Ch(G'')$, there is an $\calL$-structure $M\in \calP$ such that $M\unlhd_{\chi''} G''$.  Thus $G''$ is $\calP$-random.  
\end{example}

The most important $\calL_{\calH}$-structures for the rest of the paper are $\calH$-random $\calL_{\calH}$-templates.  We now fix notation for these special $\calL_{\calH}$-structures.

\begin{definition}
Suppose $V$ is a set, and $n$ is an integer.  Then
\begin{itemize}
\item  $\calR(V,\calH)$ is the set of all $\calH$-random $\calL_{\calH}$-templates with domain $V$ and
\item $\calR(n,\calH)$ is the class of all $\calH$-random $\calL_{\calH}$-templates of size $n$.
\end{itemize}
\end{definition}

\noindent In the above notation, $\calR$ is for ``random.''   Note that if $\calH(n)=\emptyset$ for some $n$, then $\calR(n,\calH)=\emptyset$.  

%************************************************************************
\section{Main Results}\label{mainresults}

%************************************************************************
In this section we state the main results of this paper.  Recall that $\calL$ is a fixed finite relational language of maximum arity $r\geq 2$. We now define our generalization of extremal graphs.  By convention, set $\max \emptyset=0$.

\begin{definition}\label{genexdef}
Suppose $\mathcal{H}$ is a collection of finite $\mathcal{L}$-structures.  Given $n$, set
$$
ex(n,\calH)=\max\{sub(M): M\in \calR(n,\calH)\}.
$$
We say $M\in \calR(n,\calH)$ is \emph{extremal} if $sub(M)=ex(n,\mathcal{H})$.  If $V$ is a set and $n\in \mathbb{N}$, then
\begin{itemize}
\item $\calR_{ex}(V,\calH)$ is the set of extremal elements of $\calR(V,\calH)$ and
\item $\calR_{ex}(n,\calH)$ is the class of extremal elements of $\calR(n,\calH)$.  
\end{itemize}
\end{definition}

The main idea is that when $\calH$ is a hereditary $\calL$-property, $\ex(n,\calH)$ is the correct generalization of the extremal number of a graph, and elements of $\calR_{ex}(n,\calH)$ are the correct generalizations extremal graphs of size $n$.  

\begin{definition}\label{pidef}
Suppose $\mathcal{H}$ is a nonempty collection of finite $\mathcal{L}$-structures. When it exists, set
$$
\pi(\mathcal{H})=\lim_{n\rightarrow \infty}\ex(n,\calH)^{1/{n\choose r}}
$$
\end{definition}

\noindent Using techniques similar to those in \cite{BoTh1} we will show the following.

\begin{theorem}\label{densityexists}
If $\mathcal{H}$ is hereditary $\calL$-property, then $\pi(\mathcal{H})$ exists.
\end{theorem}

We now state our approximate enumeration theorem in terms of the asymptotic density. 

\begin{theorem}[Enumeration]\label{enumeration}
Suppose $\mathcal{H}$ is a hereditary $\calL$-property.  Then the following hold. 
\begin{enumerate}
\item If $\pi(\mathcal{H})>1$, then $|\mathcal{H}_n|= \pi(\mathcal{H})^{{n\choose r}+o(n^r)}$.
\item If $\pi(\mathcal{H})\leq 1$, then $|\mathcal{H}_n|=2^{o(n^r)}$.
\end{enumerate}
\end{theorem}

The notion $\pi(\calH)$ is related to many existing notions of asymptotic density for various combinatorial structures, and Theorem \ref{enumeration} can be seen as generalizing many existing enumeration theorems.  Some of these connections will be discussed in Section \ref{end} and Appendices  \ref{coloredhg}, \ref{dgsec}, and \ref{trifreesec}.  We say a hereditary $\calL$-property $\calH$ is \emph{fast-growing} if $\pi(\calH)>1$.  In this case, we informally say $M\in \calR(n,\calH)$ is \emph{almost extremal} if $sub(M)\geq\ex(n,\calH)^{1-\epsilon}$ for some small $\epsilon$.  Our next theorem shows that almost all elements in a fast-growing hereditary $\calL$-property $\calH$ are close to subpatterns of almost extremal elements of $\calR(n,\calH)$.  Given $\epsilon>0$, $n$, and a collection $\calH$ of $\calL$-structures, let
\begin{align*}
E(n,\calH)&=\{G\in \calH_n: G\unlhd_pM\text{ for some $M\in \calR_{ex}(n,\calH)$}\}\text{ and }\\
E(\epsilon, n,\calH)&=\{G\in \calH_n: G\unlhd_pM\text{ for some $M\in \calR(n,\calH)$ with $sub(M)\geq\ex(n,\calH)^{1-\epsilon}$}\}.
\end{align*}
Given $\delta>0$, let $E^{\delta}(n,\calH)$ and $E^{\delta}(\epsilon, n,\calH)$ denote the set of $G\in \calH_n$ which are $\delta$-close to any element of $E(n,\calH)$ and $E(\epsilon, n,\calH)$, respectively.  

\begin{theorem}\label{b4stab}
Suppose $\calH$ is a fast-growing hereditary $\calL$-property.  For all $\epsilon, \delta>0$ there is $\beta>0$ such that for sufficiently large $n$, 
$$
\frac{|\calH_n\setminus E^{\delta}(\epsilon, n,\calH)|}{|\calH_n|} \leq 2^{-\beta {n\choose r}}.
$$
\end{theorem}

We now define our generalization of a graph stability theorem.

\begin{definition}\label{stabdef}
Suppose $\mathcal{H}$ is a nontrivial collection of $\mathcal{L}$-structures.  We say $\mathcal{H}$ \emph{has a stability theorem} if for all $\delta>0$ there is $\epsilon>0$ and $N$ such that $n>N$ implies the following.  If $M\in \calR(n,\calH)$ satisfies $sub(M)\geq\ex(n,\mathcal{H})^{1-\epsilon}$, then $M$ is $\delta$-close to some $M'\in \calR_{ex}(n,\calH)$.
\end{definition}

Our next result, Theorem \ref{stab} below, shows that if a fast-growing hereditary $\calL$-property $\calH$ has a stability theorem, we can strengthen Theorem \ref{b4stab} to say that that almost all elements in $\calH_n$ are approximately subpatterns of elements of $\calR_{ex}(n,\calH)$. 
\begin{theorem}\label{stab}
Suppose $\calH$ is a fast growing hereditary $\calL$-property with a stability theorem.  Then for all $\delta>0$, there is a $\beta>0$ such that for sufficiently large $n$, 
$$
\frac{|\calH_n\setminus E^{\delta}(n,\calH)|}{|\calH_n|} \leq 2^{-\beta {n\choose r}}.
$$
\end{theorem}

When one has a good understanding of the structure of elements in $\calR_{ex}(n, \calH)$, Theorem \ref{stab} gives us a good description of the approximate structure of most elements in $\calH_n$, when $n$ is large.  The main new tool we will use to prove our main theorems is Theorem \ref{version1} below, which is an adaptation of the hypergraph containers theorem to the setting of $\calL$-structures.  

\begin{definition}
If $F$ is an $\mathcal{L}$-structure, let $\tilde{F}$ be the set of $\calL_{\calH}$-structures $M$ such that $F\unlhd_pM$.  If $\mathcal{F}$ is a collection of $\mathcal{L}$-structures, let $\tilde{\mathcal{F}}=\bigcup_{F\in \calF}\tilde{F}$.   
\end{definition}

\begin{theorem}\label{version1}
Suppose $0<\epsilon<1$ and $k\geq r$ is an integer.  Then there exist positive constants $c=c(k,r, \calL,\epsilon)$ and $m=m(k,r)>1$ such that for all sufficiently large $n$ the following holds.   Assume $\mathcal{F}$ is a collection of finite $\mathcal{L}$-structures each of size at most $k$ and $\calB:=\Forb(\calF)\neq \emptyset$.  For any $n$-element set $W$, there is a collection $\mathcal{C}$ of $\calL_{\calB}$-templates with domain $W$ such that 
\begin{enumerate}
\item For all $\mathcal{F}$-free $\mathcal{L}$-structures $M$ with domain $W$, there is $C\in \mathcal{C}$ such that $M\unlhd_pC$,
\item For all $C\in \mathcal{C}$, $\prob(\tilde{\mathcal{F}},C)\leq \epsilon$ and $\prob(\calE,C)\leq \epsilon$.
\item $\log |\mathcal{C}|\leq cn^{r-\frac{1}{m}}\log n$.
\end{enumerate}
\end{theorem}

We will combine Theorem \ref{version1} with a general version of the graph removal lemma proved by Aroskar and Cummings in \cite{AroskarCummings} to prove a supersaturation theorem for hereditary $\calL$-properties (Theorem \ref{GENSUPERSAT} below), and a version of the hypergraph containers theorem for hereditary $\calL$-properties (Theorem \ref{COROLLARY2}  below).

\begin{theorem}[Supersaturation]\label{GENSUPERSAT}
Suppose $\calH$ is a non-trivial hereditary $\calL$-property and $\calF$ is as in Observation \ref{HP} so that $\calH=\Forb(\calF)$.  Then for all $\delta>0$ there are $\epsilon>0$ and $K$ such that for all sufficiently large $n$, if $M$ is an $\calL_{\calH}$-template of size $n$ such that $\prob(\tilde{\calF}(K)\cup \calE(K),M)<\epsilon$, then 
\begin{enumerate}
\item If $\pi(\calH)>1$, then $sub(M)\leq \ex(n,\calH)^{1+\delta}$.
\item If $\pi(\calH)\leq 1$, then $sub(M)\leq 2^{\delta {n\choose r}}$.
\end{enumerate}
\end{theorem}

\begin{theorem}\label{COROLLARY2}
Suppose $\calH$ is a hereditary $\calL$-property.  Then there is $m=m(\calH, r_{\calL})>1$ such that the following holds.  For every $\delta>0$ there is a constant $c=c(\calH,\calL,\delta)$ such that for all sufficiently large $n$ there is a set of $\calL_{\calH}$-templates $\mathcal{C}$ with domain $[n]$ satisfying the following properties.  
\begin{enumerate}
\item For every $H\in \mathcal{H}_n$, there is $C\in \mathcal{C}$ such that $H\unlhd_pC$.  
\item For every $C\in \mathcal{C}$, there is $C'\in \calR([n],\calH)$ such that $dist(C,C')\leq \delta$.
\item $\log |\mathcal{C}|\leq cn^{r-\frac{1}{m}}\log n$.
\end{enumerate}
\end{theorem}

%************************************************************************
\section{Proofs of Main Theorems}

%************************************************************************
%reformulate with $\tilde{mathcal{L}}$-structures instead of subsets of $U_n(F)$.

In this section we prove our main results using Theorems \ref{version1}, \ref{GENSUPERSAT}, and \ref{COROLLARY2}.  For the rest of the section, $\calH$ is a fixed hereditary $\calL$-property.

\begin{lemma}\label{templatelem2}
Suppose $N$ is an $\calL$-structure and $\tilde{N}$ is the $\calL_{\calH}$-structure such that $dom(\tilde{N})=dom(N)$ and for each $\abar \in dom(\tilde{N})^{r}$ and $p(\xbar)\in S_r(\calH)$, $\tilde{N}\models R_p(\abar)$ if and only if $N\models p(\abar)$.  Then $\tilde{N}$ is an $\calL_{\calH}$-template and $N$ is the unique full subpattern of $\tilde{N}$.  
\end{lemma}
\begin{proof}
Let $V=dom(N)=dom(\tilde{N})$.  We first verify $\tilde{N}$ is an $\calL_{\calH}$-template.  By the definition of $\tilde{N}$, for all $A\in {V\choose r}$, $Ch_{\tilde{N}}(A)=\{Diag^N(A)\}$.  Therefore $\tilde{N}$ is complete.  If $\abar\in V^r\setminus V^{\underline{r}}$ and $p\in S_r(\calH)$, then because $p$ is a proper type, $N\nvDash p(\abar)$.  Thus by definition of $\tilde{N}$, $\tilde{N}\models \neg R_p(\abar)$ and $\tilde{N}$ satisfies part (1) of Definition \ref{templatedef}.  Suppose $p(\xbar), p'(\xbar)\in S_r(\calH)$ and $\mu\in Perm(r)$ are such that $p(\xbar)=p'(\mu(\xbar))$.  Then for all $\abar\in V^r$, $\tilde{N}\models R_p(\abar)$ if and only if $N\models p(\abar)$ if and only if $N\models p'(\mu(\abar))$ if and only if $\tilde{N}\models R_{p'}(\mu(\abar))$.  Thus $\tilde{N}$ satisfies part (2) of Definition \ref{templatedef}, so $\tilde{N}$ is an $\calL_{\calH}$-template. Define $\chi: {V\choose r}\rightarrow S_r(C_V,\calH)$ by setting $\chi(A)=Diag^N(A)$ for each $A\in {V\choose r}$.  It is clear that $\chi\in Ch(\tilde{N})$ and $N\unlhd_{\chi}\tilde{N}$.  By definition of $\tilde{N}$,  $\chi$ is the \emph{only} choice function for $\tilde{N}$, so any full subpattern of $\tilde{N}$ must be chosen by $\chi$.  By Observation \ref{ob00}, $\chi$ chooses at most one $\calL$-structure, so $N$ is the unique full subpattern of $\tilde{N}$.
\end{proof}

\noindent We now prove Theorem \ref{enumeration}.  The proof is based on the method of proof in \cite{BoTh1}.

\vspace{3mm}
\noindent {\bf Proof of Theorem \ref{densityexists}.}
Let $b_n=ex(n,\calH)^{1/{n\choose r}}$.  If $\calH$ is trivial, then for sufficiently large $n$, $\calR(n,\calH)=\emptyset$ so by convention, $\ex(n,\calH)=0$.  Thus, for sufficiently large $n$,  $b_n=0$ and $\pi(\mathcal{H})$ exists and is equal to zero.  

Assume now $\mathcal{H}$ is nontrivial.   We show that the sequence $b_n$ is bounded below and non-increasing.  Since $\mathcal{H}$ is non-trivial and has the hereditary property, $\calH_n\neq \emptyset$ for all $n$.  Fix $n\geq 1$ and choose any $N\in \mathcal{H}_n$.  Let $\tilde{N}$ be the $\calL_{\calH}$-structure defined as in Lemma \ref{templatelem2} for $N$.  Then $\tilde{N}$ is an $\calL_{\calH}$-template, and its only full subpattern is $N$.  Since $N\in \calH$, this implies $\tilde{N}\in \calR(n,\calH)$ and $sub(\tilde{N})=1$. So we have shown $b_n\geq 1$ for all $n\geq 1$.

We now show the $b_n$ are non-increasing.   Fix $n\geq 2$. Let $M\in \calR(n,\calH)$ be such that $sub(M)\geq 1$ and let $V=dom(M)$.  Fix $a\in V$ and set $V_a=V\setminus \{a\}$ and $M_a=M[V_a]$.  We claim $M_a\in \calR(n-1,\calH)$.  Because $M$ is an $\calL_{\calH}$-template, the definition of $M_a$ implies $M_a$ is also an $\calL_{\calH}$-template.  Suppose $\chi\in Ch(M_a)$.  We want to show there exists $N_a\in \calH$ with $N_a\unlhd_{\chi}M_a$.  We define a function $\chi':{V\choose r}\rightarrow S_r(C_V,\calH)$ as follows.  For $A\in {V_a\choose r}$, set $\chi'(A)=\chi(A)$, and for $A\in {V\choose r}\setminus A\in {V_a\choose r}$, choose $\chi'(A)$ to be any element of $Ch_{M_a}(A)=Ch_M(A)$ (this is possible since $M$ is complete).  Note that for each $A\in {V_a\choose r}$, $\chi(A)\in Ch_M(A)$, so $\chi'\in Ch(M)$.  Because $M$ is $\calH$-random, there is $N\in \calH$ such that $N\unlhd_{\chi'}M$.  Let $N_a=N[V_a]$.  Because $\calH$ has the hereditary property and $N_a\subseteq_{\calL}N$, $N_a\in \calH$.  For each $A\in {V_a\choose r}$, $Diag^{N_a}(A)=Diag^N(A)=\chi'(A)=\chi(A)$, so $N_a\unlhd_{\chi}M_a$.  Thus we have verified that $M_a\in \calR(n-1,\calH)$.  By definition of $b_{n-1}$, this implies $sub(M_a)^{1/{n-1\choose r}}\leq b_{n-1}$. Because $M_a$ is $\calH$-random, Lemma \ref{Lrandom} implies it is error-free, so Observation \ref{ob0} implies $sub(M_a)=\prod_{A\in {V_a\choose r}} |Ch_{M_a}(A)|$.  Then observe that
$$
sub(M)=\Big(\prod_{a\in V} \prod_{A\in {V_a\choose r}} |Ch_{M_a}(A)|\Big)^{1/(n-r)}= \Big(\prod_{a\in V} sub(M_a)\Big)^{1/(n-r)}.
$$
Since $sub(M_a)\leq b_{n-1}^{n-1\choose r}$, this implies
$$
sub(M)\leq \Big(\prod_{a\in V} b_{n-1}^{n-1\choose r}\Big)^{1/(n-r)} = b_{n-1}^{n{n-1\choose r}/(n-r)}= b_{n-1}^{n\choose r}.
$$
Thus for all $M\in \calR(n,\calH)$, $sub(M)^{1/{n\choose r}}\leq b_{n-1}$. So by definition, $b_n\leq b_{n-1}$. 
\qed

\vspace{3mm}

\noindent The following observations follow from the proof of Theorem \ref{densityexists}.
\begin{observation}\label{ob5}
Assume $\calH$ is a hereditary $\calL$-property.  
\begin{enumerate}[(a)]
\item For all $n$, $\ex(n,\calH)^{1/{n\choose r}}\geq \pi(\calH)$ (since $(b_n)_{n\in \mathbb{N}}$ is non-increasing and converges to $\pi(\calH)$).  
\item Either $\calH$ is trivial and $\pi(\calH)=0$ or $\calH$ is non-trivial and $\pi(\calH)\geq 1$.   
\end{enumerate}
\end{observation}
\vspace{3mm}

\noindent{\bf Proof of Theorem \ref{enumeration}.}
Assume $\calH$ is a hereditary $\calL$-property.  Recall we want to show the following.
\begin{enumerate}
\item If $\pi(\mathcal{H})>1$, then $|\mathcal{H}_n|= \pi(\mathcal{H})^{{n\choose r}+o(n^r)}$.
\item If $\pi(\mathcal{H})\leq 1$, then $|\mathcal{H}_n|=2^{o(n^r)}$.
\end{enumerate}
Assume first that $\calH$ is trivial.  Then by Observation \ref{ob5}(b), $\pi(\calH)=0\leq 1$, so we are in case 2. Since $|\calH_n|=0$ for all sufficiently large $n$, $|\calH_n|=2^{o(n^2)}$ holds, as desired.  Assume now $\calH$ is non-trivial, so $\pi(\calH)\geq 1$ by Observation \ref{ob5}(b).  We show that for all $0<\eta<1$, either $\pi(\calH)=1$ and $|\calH_n|\leq 2^{\eta n^r}$ or $\pi(\calH)>1$ and $\pi(\calH)^{n\choose 2}\leq |\calH_n|\leq \pi(\mathcal{H})^{{n\choose r}+\eta n^r}$.  Fix $0<\eta<1$.  Let $\calF$ be as in Observation \ref{HP} for $\calH$ so that $\calH=\Forb(\calF)$.  Choose $\epsilon>0$ and $K$ as in Theorem \ref{GENSUPERSAT} for $\delta=\eta/4$.   Replacing $K$ if necessary, assume $K\geq r$.  Apply Theorem \ref{version1} to $\epsilon$ and $\calF(K)$ to obtain $m=m(K,r)>1$ and $c=c(K,r, \calL, \epsilon)$.  Assume $n$ is sufficiently large.  Theorem \ref{version1} with $W=[n]$ and $\calB:=\Forb(\calF(K))$ implies there is a collection $\calC$ of $\calL_{\calB}$-templates with domain $[n]$ such that the following hold.
\begin{enumerate}[(i)]
\item For all $\mathcal{F}(K)$-free $\mathcal{L}$-structures $M$ with domain $[n]$, there is $C\in \mathcal{C}$ such that $M\unlhd_pC$,
\item For all $C\in \mathcal{C}$, $\prob(\widetilde{\mathcal{F}(K)},C)\leq \epsilon$ and $\prob(\calE,C)\leq \epsilon$.
\item $\log |\mathcal{C}|\leq cn^{r-\frac{1}{m}}\log n$.
\end{enumerate}
Note that because $K\geq r$, $\calH=\Forb(\calF)$ and $\calB=\Forb(\calF(K))$ imply we must have $S_r(\calH)=S_r(\calB)$. Consequently all $\calL_{\calB}$-templates are also $\calL_{\calH}$-templates.  In particular the elements in $\calC$ are all $\calL_{\calH}$-templates.  Therefore, (ii) and Theorem \ref{GENSUPERSAT} imply that for all $C\in \calC$, either $sub(C)\leq \ex(n,\calH)^{1+\eta/4}$ (case $\pi(\calH)>1$) or $sub(C)\leq 2^{\eta{n\choose r}/4}$ (case $\pi(\calH)=1$).   Note every element in $\calH_n$ is $\calF$-free, so is also $\calF(K)$-free.  This implies by (i) that every element of $\calH_n$ is a full subpattern of some $C\in \calC$.  Therefore we can construct every element in $\calH_n$ as follows.
\begin{enumerate}[$\bullet$]
\item Choose a $C\in \mathcal{C}$.  There are at most $|\calC|\leq 2^{cn^{r-\frac{1}{m}} \log n}$ choices.
\item Choose a full subpattern of $C$.  There are at most $sub(C)\leq \ex(n,\calH)^{1+\eta/4}$ choices if $\pi(\calH)>1$ and at most $sub(C)\leq 2^{\eta{n\choose r}/4}$ choices if $\pi(\calH)=1$.
\end{enumerate}
This implies
\begin{eqnarray}\label{AP}
|\calH_n|\leq \begin{cases} 2^{cn^{r-\frac{1}{m}}\log n}\ex(n,\calH)^{1+\eta/4} & \text{ if }\pi(\calH)>1\\
2^{cn^{r-\frac{1}{m}}\log n}2^{\eta{n\choose r}/4} & \text{ if }\pi(\calH)=1.
\end{cases}\end{eqnarray}
If $\pi(\calH)>1$, then we may assume $n$ is sufficiently large so that $\ex(n,\calH)\leq \pi(\calH)^{(1+\eta/4){n\choose r}}$ (see Observation \ref{ob5}(a)).  Combining this with (\ref{AP}), we have that when $\pi(\calH)>1$, 
$$
|\calH_n|\leq 2^{cn^{r-\frac{1}{m}}\log n}\pi(\mathcal{H})^{(1+\eta/4)^2{n\choose r}}\leq \pi(\mathcal{H})^{{n\choose r}+\eta {n\choose r}},
$$
where the last inequality is because $\pi(\calH)>1$, $(1+\eta/4)^2<1+\eta$, and $n$ is sufficiently large.  If $\pi(\calH)=1$, then (\ref{AP}) implies
$$
|\calH_n|\leq 2^{cn^{r-\frac{1}{m}}\log n}2^{\eta{n\choose r}/2}\leq 2^{\eta{n\choose r}},
$$
where the last inequality is because $n$ is sufficiently large. Thus, we have shown $|\calH_n|\leq 2^{\eta n^r}$ when $\pi(\calH)=1$ and $|\calH_n|\leq  \pi(\mathcal{H})^{{n\choose r}+\eta n^r}$ when $\pi(\calH)>1$.  We just have left to show that when $\pi(\calH)>1$, then $|\calH_n|\geq \pi(\calH)^{n\choose r}$.  This holds because for any $M\in \calR_{ex}([n],\calH)$, all $\ex(n,\calH)$ many full subpatterns of $M$ are in $\calH_n$.  Thus $|\calH_n|\geq \ex(n,\calH)\geq \pi(\calH)^{n\choose r}$, where the second inequality is by Observation \ref{ob5}(a). This finishes the proof.
\qed

\vspace{3mm}

\noindent We now prove a few lemmas needed for Theorems \ref{b4stab} and \ref{stab}.
\begin{lemma}\label{templatelem}
Suppose $C$ and $C'$ are $\calL_{\calH}$-templates with the same domain $V$.  Then for all $A\in {V\choose r}$, $A\in \diff(C,C')$ if and only if $Ch_C(A)\neq Ch_{C'}(A)$.
\end{lemma}
\begin{proof}
Fix $A\in {V\choose r}$.  Suppose first $A\in \diff(C,C')$.  Then there is $p\in S_r(\calH)$ and an enumeration $\abar$ of $A$ such that $C\models R_p(\abar)$ and $C'\models \neg R_p(\abar)$.  This implies $p(c_{\abar})\in Ch_C(A)$.  Suppose by contradiction $p(c_{\abar})$ were in $Ch_{C'}(A)$.  Then there is $p'(\xbar)\in S_r(\calH)$ and $\mu\in Perm(r)$ such that $p'(\mu(\xbar))= p(\xbar)$ and $C'\models R_{p'}(\mu(\abar))$.  Because $C'$ is an $\calL_{\calH}$-template, this implies $C'\models R_p(\abar)$, a contradiction.  

Suppose now $Ch_C(A)\neq Ch_{C'}(A)$.  Then there is $p(\xbar)\in S_r(\calH)$ and an enumeration $\abar$ of $A$ such that $p(c_{\abar})\in Ch_C(A)$ and $p(c_{\abar})\notin Ch_{C'}(A)$.  Since $p(c_{\abar})\in Ch_C(A)$, by definition there is $p'(\xbar)$ and $\mu\in Perm(r)$ such that $p'(\mu(\xbar))= p(\xbar)$ and $C\models R_{p'}(\mu(\abar))$.  Since $p(c_{\abar})\notin Ch_{C'}(A)$ and $C'$ is an $\calL_{\calH}$-template, $C'\models \neg R_{p'}(\mu(\abar))$.  This shows $qftp^C(\abar) \neq qftp^{C'}(\abar)$, so $A\in \diff(C,C')$, as desired.
\end{proof}

\begin{lemma}\label{deltaclose1}
Suppose $\calH$ is a non-trivial hereditary $\calL$-property.  Then there is $\gamma=\gamma(\calH)>0$ such that for all $\delta>0$ and $n\geq r$, if $C$ and $C'$ are $\calL_{\calH}$-templates with domain $[n]$ such that $C'$ is error-free and $\dist(C,C')\leq \delta$, then the following holds.  
\begin{enumerate}
\item If $\pi(\calH)>1$, then $sub(C)\leq sub(C')\ex(n,\calH)^{\gamma \delta}$.
\item If $\pi(\calH)=1$, then $sub(C)\leq sub(C')2^{\gamma \delta {n\choose r}}$.
\end{enumerate}
\end{lemma}

\begin{proof}
Fix $n\geq r$ and assume $C$ and $C'$ are $\calL_{\calH}$-templates with domain $[n]$ such that $C'$ is error-free and $\dist(C,C')\leq \delta$.  Then by definition of $\dist(C,C')$, $|\diff(C,C')|\leq \delta {n\choose r}$. By Lemma \ref{templatelem}, 
\begin{align}\label{ch}
\diff(C,C')=\{A\in {V\choose r}: Ch_C(A)\neq Ch_{C'}(A)\}.
\end{align}
Note that for every $A\in {V\choose r}$, $|Ch_C(A)|\leq |S_r(\calH)|$ (by definition of $Ch_C(A)$) and $1\leq |Ch_{C'}(A)|$ (since $C'$ is complete).  Thus $\frac{|Ch_{C}(A)|}{|Ch_{C'}(A)|}\leq |S_r(\calH)|$.  By Observation \ref{ob0} and (\ref{ch}),
\begin{align*}
sub(C)\leq \prod_{A\in {V\choose r}}|Ch_C(A)|&=\Big(\prod_{A\notin \diff(C,C')}|Ch_{C'}(A)|\Big)\Big(\prod_{A\in \diff(C,C')}|Ch_{C}(A)|\Big)\\
&=\Big(\prod_{A\in {V\choose r}}|Ch_{C'}(A)|\Big)\Big(\prod_{A\in \diff(C,C')}\frac{|Ch_{C}(A)|}{|Ch_{C'}(A)|}\Big).
\end{align*}
Combining this with $\frac{|Ch_{C}(A)|}{|Ch_{C'}(A)|}\leq |S_r(\calH)|$ and $|\diff(C,C')|\leq \delta {n\choose r}$ yields
\begin{align}\label{ineq1}
sub(C)\leq \Big(\prod_{A\in {V\choose r}}|Ch_{C'}(A)|\Big) |S_r(\calH)|^{\delta{n\choose r}}=sub(C')|S_r(\calH)|^{\delta {n\choose r}},
\end{align}
where the equality is by Observation \ref{ob0} and because $C'$ is error-free.  If $\pi(\calH)>1$, choose $\gamma>0$ such that $|S_r(\calH)|=\pi(\calH)^{\gamma}$ (this is possible since $\pi(\calH)>1$ implies $S_r(\calH)>1$).  Recall from Observation \ref{ob5}(a) that for all $n$, $\ex(n,\calH)\geq \pi(\calH)^{n\choose r}$. Combining this with our choice of $\gamma$ and (\ref{ineq1}), we have 
$$
sub(C)\leq sub(C')|S_r(\calH)|^{\delta {n\choose r}}=sub(C')\pi(\calH)^{\gamma \delta {n\choose r}}\leq sub(C')\ex(n,\calH)^{\gamma \delta}.
$$
If $\pi(\calH)=1$, choose $\gamma>0$ such that $|S_r(\calH)|\leq 2^{\gamma}$ (this is possible since $\calH$ nontrivial implies $|S_r(\calH)|\geq 1)$.  Combining our choice of $\gamma$ with (\ref{ineq1}) implies 
$$
sub(C)\leq sub(C')|S_r(\calH)|^{\delta {n\choose r}}\leq sub(C')2^{\gamma \delta {n\choose r}}.
$$
\end{proof}

\begin{lemma}\label{deltaclose2}
Suppose $C$ is an $\calL_{\calH}$-template with domain $W$ of size $n\geq r$ and $G\unlhd_pC$.  If $D\in \calR(W,\calH)$ is such that $\dist(C,D)\leq \delta$, then there is $G'\in \calH$ such that $G'\unlhd_pD$ and $\dist(G,G')\leq \delta$.
\end{lemma}

\begin{proof}
Fix $C$ and $D$ satisfying the hypotheses  Because $\dist(C,D)\leq \delta$, we have $|\diff(C,D)|\leq \delta{n\choose r}$. By Lemma \ref{templatelem}, 
\begin{align}\label{ch1}
\diff(C,D)=\{A\in {W\choose r}: Ch_C(A)\neq Ch_D(A)\}.
\end{align}
 Define a function $\chi: {W\choose r}\rightarrow S_r(C_W)$ as follows.  For $A\in {W\choose r}\setminus \diff(C,D)$, set $\chi(A)=Diag^G(A)$.  For each $A\in \diff(C,D)$, choose $\chi(A)$ to be any element of $Ch_D(A)$ (which is nonempty because $D$ is an $\calL_{\calH}$-template).  Since $G\unlhd_pC$, for all $A\in {W\choose r}$, $Diag^G(A)\in Ch_C(A)$.  Thus, by definition of $\chi$ and (\ref{ch1}), for all $A\in {W\choose r}\setminus \diff(C,D)$, $\chi(A)=Diag^G(A)\in Ch_C(A)=Ch_D(A)$.  For $A\in \diff(C,D)$, $\chi(A)\in Ch_D(A)$ by assumption.  Thus $\chi\in Ch(D)$.  Because $D$ is $\calH$-random, there is $G'\in \calH$ such that $G'\unlhd_{\chi}D$.  We show $\dist(G,G')\leq \delta$.  By definition of $\chi$ and since $G'\unlhd_{\chi}D$, we have that for all $A\in {W\choose r}$, if $A \notin \diff(C,D)$, then $Diag^{G'}(A)=\chi(A)=Diag^G(A)$, which implies $A\notin \diff(G,G')$.  Thus $\diff(G,G')\subseteq \diff(C,D)$ so $|\diff(G,G')|\leq \delta {n\choose r}$ and $\dist(G,G')\leq \delta$ by definition.  
\end{proof}

\noindent{\bf Proof of Theorem \ref{b4stab}}.  Let $\calH$ be a fast-growing hereditary $\calL$-property.  Fix $\epsilon$ and $\delta>0$.  Given $n$, let $A(n, \epsilon, \delta)=\calH_n\setminus E^{\delta}(\epsilon, n,\calH)$.  Recall, we want to show there is $\beta>0$ such that for sufficiently large $n$,
\begin{align}\label{m}
\frac{|A(n,\epsilon, \delta)|}{|\calH_n|}\leq 2^{-\beta {n\choose r}}.
\end{align}
Let $\gamma>0$ be as in Lemma \ref{deltaclose1} for $\calH$.  Choose $K>2r$ sufficiently large so that $1-\epsilon +\gamma \delta /K<1-\epsilon/2$.  
Apply Theorem \ref{COROLLARY2} to $\frac{\delta}{K}$ to obtain constants $c$ and $m>1$.  Assume $n$ is sufficiently large.  Then Theorem \ref{COROLLARY2} implies there is a collection $\calC$ of $\calL_{\calH}$-templates with domain $[n]$ such that the following hold.
\begin{enumerate}[(i)]
\item For every $H\in \mathcal{H}_n$, there is $C\in \mathcal{C}$ such that $H\unlhd_pC$.  
\item For every $C\in \mathcal{C}$, there is $C'\in \calR([n],\calH)$ such that $dist(C,C')\leq \delta$.
\item $\log |\mathcal{C}|\leq cn^{r-\frac{1}{m}}\log n$.
\end{enumerate}
Suppose $G\in A(n,\epsilon,\delta)$.  By (i), there is $C\in \calC_n$ such that $G\unlhd_pC$.   By (ii), there is $M_C\in \calR([n],\calH)$ such that $\dist(C,M_C)\leq \frac{\delta}{K}$.   By Lemma \ref{deltaclose2}, there is $G'\unlhd_p M_C$ with $\dist(G,G')\leq \frac{\delta}{K}\leq \delta$.  If $sub(M_C)\geq \ex(n,\calH)^{1-\epsilon}$, then by definition of $E^{\delta}(\epsilon, n,\calH)$, $\dist(G,G')\leq \delta$ and $G'\unlhd_pM_C$ would imply $G\in E^{\delta}(\epsilon, n,\calH)$, contradicting our assumption that $G\in A(n,\epsilon, \delta)=\calH_n\setminus E^{\delta}(\epsilon, n,\calH)$.  Therefore, we must have $sub(M_C)<\ex(n,\calH)^{1-\epsilon}$.  Note $M_C\in \calR([n],\calH)$ implies $M_C$ is error-free, so Lemma \ref{deltaclose1} and the fact that $\dist(C,M_C)\leq \delta/K$ imply $sub(C)\leq sub(M_C)\ex(n,\calH)^{\gamma\delta/K}$.  Combining this with the fact that $sub(M_C)<\ex(n,\calH)^{1-\epsilon}$ we have that 
$$
sub(C)<\ex(n,\calH)^{1-\epsilon}\ex(n,\calH)^{\gamma\delta/K}=\ex(n,\calH)^{1-\epsilon +\gamma \delta/K}\leq\ex(n,\calH)^{1-\epsilon/2},
$$
where the second inequality is by assumption on $K$.  Therefore every $G\in A(n,\epsilon, \delta)$ can be constructed as follows.
\begin{itemize}
\item Choose $C\in \calC_n$ with $sub(C)<\ex(n,\calH)^{1-\epsilon/2}$.  There are at most $|\calC_n|\leq 2^{cn^{r-\frac{1}{m}}\log n}$ ways to do this, where the bound is by (iii).  Since $n$ is large and $\pi(\calH)>1$, we may assume $2^{cn^{r-\frac{1}{m}}\log n}\leq \pi(\calH)^{\epsilon{n\choose r}/4}$.
\item Choose a full subpattern of $C$.  There are at most $sub(C)<\ex(n,\calH)^{1-\epsilon/2}$ ways to do this.
\end{itemize}
Combining these bounds yields $|A(n,\epsilon,\delta)|\leq \pi(\calH)^{\epsilon{n\choose r}/4}\ex(n,\calH)^{1-\epsilon/2}$.  Recall that $|\calH_n|\geq\ex(n,\calH)$ holds, since for any $M\in \calR_{ex}([n],\calH)$, all $\ex(n,\calH)$-many full subpatterns of $M$ are all in $\calH_n$.  Therefore
\begin{align}\label{l}
\frac{|A(n,\epsilon, \delta)|}{|\calH_n|}\leq \frac{\pi(\calH)^{\epsilon{n\choose r}/4}\ex(n,\calH)^{1-\epsilon/2}}{\ex(n,\calH)}=\pi(\calH)^{\epsilon{n\choose r}/4}\ex(n,\calH)^{-\epsilon/2}\leq \pi(\calH)^{-\epsilon {n\choose r}/4},
\end{align}
where the last inequality is because $\pi(\calH)^{n\choose r}\leq\ex(n,\calH)$.  Therefore we have $\frac{|A(n,\epsilon, \delta)|}{|\calH_n|}\leq 2^{-\beta{n\choose r}}$, where  $\beta = \frac{\epsilon \log \pi(\calH)}{4\log 2}$.  Note $\beta>0$ since $\pi(\calH)>1$.
\qed

\vspace{3mm}

\noindent {\bf Proof of Theorem \ref{stab}.}
Suppose $\calH$ is a fast growing hereditary $\calL$-property with a stability theorem.  Fix $\delta>0$.  Recall we want to show there is $\beta>0$ such that for sufficiently large $n$,
\begin{align*}
\frac{|\calH_n\setminus E^{\delta}(n,\calH)|}{|\calH_n|}\leq 2^{-\beta {n\choose r}}
\end{align*}
By Theorem \ref{b4stab}, it suffices to show that there are $\epsilon_1, \delta_1>0$ such that for all sufficiently large $n$, $E^{\delta_1}(\epsilon_1, n,\calH)\subseteq E^{\delta}(n,\calH)$.  

Because $\calH$ has a stability theorem, there is $\epsilon$ such that for all sufficiently large $n$, if $H\in \calR([n],\calH)$ satisfies $sub(H)\geq\ex(n,\calH)^{1-\epsilon}$, then there is $H'\in \calR_{ex}([n],\calH)$ with $\dist(H,H')\leq \frac{\delta}{2}$.  Fix $n$ sufficiently large.  We claim $E^{\delta/2}(\epsilon, n,\calH)\subseteq E^{\delta}(n,\calH)$.  Suppose $G\in E^{\delta/2}(\epsilon, n,\calH)$.  Then by definition, $G$ is $\delta/2$-close to some $G'$ such that $G'\unlhd_pH$, for some $H\in \calR([n],\calH)$ satisfying $sub(H)\geq\ex(n,\calH)^{1-\epsilon}$.  By choice of $\epsilon$ and because $n$ is sufficiently large, there is $H'\in \calR_{ex}([n],\calH)$ such that $\dist(H,H')\leq \frac{\delta}{2}$.  Lemma \ref{deltaclose2} implies there is some $G''\unlhd_pH'$ such that $\dist(G',G'')\leq \frac{\delta}{2}$.  Observe that $G''\in E(n,\calH)$ and  
$$
\dist(G, G'')\leq \dist(G,G')+\dist(G', G'') \leq \frac{\delta}{2}+\frac{\delta}{2}=\delta.
$$
This implies that $G\in E^{\delta}(n,\calH)$, as desired.
\qed

%************************************************************************
\section{Characterization of $\calH$-random $\calL_{\calH}$-templates}\label{Hrandom}
%************************************************************************
In this section we give an equivalent characterization for when an $\calL_{\calH}$-structure is an $\calH$-random $\calL_{\calH}$-template, where $\calH$ is a hereditary $\calL$-property.  The results in this section will be used in the proofs of our remaining results, Theorems \ref{version1}, \ref{GENSUPERSAT}, and \ref{COROLLARY2}.  For the rest of this section, $\calH$ is a fixed nonempty collection of finite $\calL$-structures.

\begin{definition}
Define $\FLAW$ to be the class of all $\calL_{\calH}$-structures of size $r$ which are not $\calL_{\calH}$-templates.  Elements of $\FLAW$ are called flaws.
\end{definition}

\begin{lemma}\label{flaw}
An $\calL_{\calH}$-structure $M$ is an $\calL_{\calH}$-template if and only if it is $\FLAW$-free.
\end{lemma}
\begin{proof}
Let $dom(M)=V$.  It is straightforward from Definition \ref{templatedef} to check that $M$ is an $\calL_{\calH}$-template if and only if for all $A\in {V\choose r}$, $M[A]$ is an $\calL_{\calH}$-template.  By definition of $\FLAW$, $M$ is $\FLAW$-free if and only if for all $A\in {V\choose r}$, $M[A]$ is an $\calL_{\calH}$-template.  This finishes the proof.
\end{proof}

\noindent We are now ready to prove the main result of this section.

\begin{proposition}\label{random}
Suppose $\calH$ is a hereditary $\calL$-property, and $\mathcal{F}$ is the class of finite $\mathcal{L}$-structures from Observation \ref{HP} such that $\Forb(\calF)=\calH$. Then a complete $\calL_{\calH}$-structure $M$ is $\calH$-random if and only if $M$ is $\tilde{\calF}$-free and error-free.
\end{proposition}

\begin{proof}
By Observation \ref{HP}, $\calF$ is closed under isomorphism.  Fix a complete $\calL_{\calH}$-structure $M$ and let $V=dom(M)$.  Suppose first that $M$ is $\calH$-random.  Then $M$ is complete and for every $\chi\in Ch(M)$, there is $N\in \calH$ such that $N\unlhd_{\chi}M$.  This implies by Proposition \ref{Lrandom} that $M$ is error-free.  Suppose by contradiction $M$ is not $\tilde{\calF}$-free.  Combining the assumption that $\calF$ is closed under isomorphism and the definition of $\tilde{\calF}$, this implies there is $B\subseteq V$ and $F\in \calF$ such that $M[B]\in \tilde{F}$.  By definition of $\tilde{F}$, there is $\chi_B\in Ch(M[B])$ such that $F\unlhd_{\chi_B}M[B]$.  Define a function $\chi:{V\choose r}\rightarrow S_r(C_V,\calH)$ as follows.  For each $A\in {B\choose r}$, set $\chi(A)=\chi_B(A)$.  Clearly, $\chi_B\in Ch(M[B])$ implies that for all $A\in {B\choose r}$, $\chi_B(A)\in Ch_M(A)$.  For each $A\in {V\choose r}\setminus {B\choose r}$, define $\chi(A)$ to be any element of $Ch_M(A)$ (this is possible since $M$ is complete by assumption).   By construction, $\chi\in Ch(M)$.  Because $M$ is $\calH$-random, there is $D\in \calH$ such that $D\unlhd_{\chi}M$.  By choice of $\calF$, since $D\in \calH$, we have that $D$ is $\calF$-free, which implies $D$ is $F$-free since $F\in \calF$.  We claim $D[B]\cong_{\calL}F$, a contradiction.  For each $A\in {B\choose r}$, $D\unlhd_{\chi}M$, $F\unlhd_{\chi_B}M[B]$, and the definition of $\chi$ imply
$$
Diag^D(A)=\chi(A)=\chi_B(A)=Diag^F(A).
$$
Thus $Diag(D[B])=\bigcup_{A\in {B\choose r}}Diag^D(A)=\bigcup_{A\in {B\choose r}}Diag^F(A)=Diag(F)$ implies $D[B]\cong_{\calL}F$.

For the converse, suppose $M$ is a complete $\calL_{\calH}$-structure which is $\tilde{\calF}$-free and error-free.  Suppose by contradiction $M$ is not $\calH$-random.  Then there is $\chi\in Ch(M)$ such that there is no $N\in \calH$ with $N\unlhd_{\chi}M$.  Since $M$ is error-free, Proposition \ref{Lrandom} implies there is some $\calL$-structure $N$ such that $N\unlhd_{\chi}M$.  Thus we must have $N\notin \calH$.  By choice of $\calF$ from Observation \ref{HP}, $N$ is not $\calF$-free.  This along with the fact that $\calF$ is closed under isomorphism implies there is $B\subseteq V$ such that $N[B]\in \calF$.  But $N\unlhd_{p}M$ implies $N[B]\unlhd_{p}M[B]$ (this is straightforward to check).  Since $N[B]\in \calF$, this implies $M[B]\in \tilde{\calF}$ by definition of $\tilde{\calF}$, contradicting that $M$ is $\tilde{\calF}$-free.
\end{proof}

\begin{corollary}\label{charR}
Suppose $\calH$ is a hereditary $\calL$-property, and $\mathcal{F}$ is the class of finite $\mathcal{L}$-structures from Observation \ref{HP} such that $\Forb(\calF)=\calH$. Let $M$ be an $\calL_{\calH}$-structure.  Then $M\in \calR(dom(M),\calH)$ if and only if $M$ is $\tilde{\calF}$-free, error-free, and $\FLAW$-free.
\end{corollary}
\begin{proof}
By definition, $M\in \calR(dom(M),\calH)$ if and only if $M$ is an $\calH$-random $\calL_{\calH}$-template.  By Lemma \ref{flaw} and Proposition \ref{random}, this holds if and only if $M$ is $\tilde{\calF}$-free, error-free, and $FLAW$-free.
\end{proof}

%************************************************************************
\section{Graph Removal and Proofs of Theorems \ref{GENSUPERSAT} and \ref{COROLLARY2}.}\label{rphrem}

%************************************************************************

In this section we will use a version of the graph removal lemma from \cite{AroskarCummings} to prove Theorem \ref{GENSUPERSAT} and to prove Theorem \ref{COROLLARY2} from Theorem \ref{version1}.  We now state definitions required to quote the graph removal lemma from \cite{AroskarCummings}. Throughout the rest of this section, $\mathcal{L}_0$ is a fixed finite relational language with $r_{\calL_0}=r$.  Note $\calL_0$ is not necessarily the same as $\calL$, although we are assuming $r_{\calL_0}=r_{\calL}=r$.  Given a partition $p$ of a finite set $X$, let $||p||$ denote the number of parts in  $p$.

\begin{definition} Let $Index=\{(R,p): R\in \mathcal{L}_0$ and $p$ is a partition of $[\ell]$ where $\ell$ is the arity of $R\}$.  Suppose $(R,p)\in Index$ and $R$ has arity $\ell$. 
\begin{enumerate}
\item $C_p(x_1,\ldots, x_{\ell})$ is the subtuple of $(x_1,\ldots, x_{\ell})$ obtained by replacing each $x_i$ with $x_{p(i)}$ where $p(i)=\min \{j: x_j$ is in the same part of $p$ as $i\}$, then deleting all but the first occurance of each variable in the tuple $(x_{p(1)},\ldots, x_{p(\ell)})$.
\item $R_p(C(\xbar))$ is the $||p||$-ary relation obtained from $R(x_1, \ldots, x_{\ell})$ by replacing each $x_i$ with $x_{p(i)}$ where $p(i)=\min \{j: x_j$ is in the same part of $p$ as $i\}$.
\item If $N$ is an $\mathcal{L}_0$-structure, define $DH^{R}_p(N)=\{\abar\in dom(N)^{\underline{||p||}}: N\models R_p(\abar)\}$.
\end{enumerate}
\end{definition}
Now we can define the notion of distance between two $\mathcal{L}_0$-structures from \cite{AroskarCummings}.    
\begin{definition}
Given $(R,p)\in Index$ and $M$, $N$ two finite $\mathcal{L}_0$-structures with the same universe $W$, set
$$
d_p^R(M,N)=\frac{|DH_p^R(M)\Delta DH_p^R(N)|}{|W|^{||p||}}\qquad \hbox{ and set }\qquad d(M,N)=\sum_{(R,p)\in Index}d^R_p(M,N).
$$
\end{definition}

\noindent We will see below in Lemma \ref{distlem} that this notion of distance, $d(M,N)$, is related to our notion of distance, $\dist(M,N)$. We first state the graph removal lemma of Aroskar and Cummings as it appears in their paper (Theorem 2 from \cite{AroskarCummings}).

\begin{theorem}[{\bf Aroskar-Cummings \cite{AroskarCummings}}]\label{triangleremoval}
Suppose $\mathcal{A}$ is a collection of finite $\mathcal{L}_0$-structures.  For every $\delta>0$ there exists $\epsilon>0$ and $K$ such that the following holds.  For all sufficiently large finite $\mathcal{L}_0$-structures $M$, if $\prob(\calA(K),M)<\epsilon$, then there is an $\calL_0$-structure $M'$ with  $dom(M')=dom(M)$ such that $d(M',M)<\delta$ and $\prob(\calA,M')=0$.
\end{theorem}

The following relationship between $d(M,N)$ and $\dist(M,N)$ will allow us to restate this graph removal lemma.  Given a tuple $\xbar=(x_1,\ldots, x_{\ell})$, a \emph{subtuple} of $\xbar$ is any tuple $\xbar'=(x_{i_1},\ldots, x_{i_{\ell'}})$ where $1\leq i_1<\ldots<i_{\ell'}\leq \ell$.  If $\ell'<\ell$, we say $\xbar'$ is a \emph{proper subtuple} of $\xbar$, denoted $\xbar'\subsetneq \xbar$.

\begin{lemma}\label{distlem}
If $M$ and $N$ are $\calL_0$-structures with the same finite domain $W$ of size at least $2r$, then 
$$
\dist(M,N)\leq (r!)^22^r d(M,N).
$$
\end{lemma}
\begin{proof}
Let $n=|W|$.  Note that $n\geq 2r$ implies for all $1\leq \ell \leq r$, 
\begin{align}\label{d}
\frac{n!}{(n-\ell)!}=n\cdot (n-1)\cdots (n-\ell+1)\geq (n-\ell +1)^{\ell}\geq (n/2)^{\ell}=n^{\ell}/2^{\ell}.
\end{align}
Given $1\leq \ell\leq r$, define 
$$
\diff^{\ell}(M,N)=\{\abar\in W^{\underline{\ell}}: qftp^M(\abar)\neq qftp^N(\abar)\text{ and for all $\abar'\subsetneq \abar$, $qftp^M(\abar')= qftp^N(\abar')$}\}.
$$
Observe that elements in $\diff(M,N)$ are \emph{sets} of elements from $W$, while elements in $\diff^{\ell}(M,N)$ are \emph{tuples} of elements of $W$.  Clearly if $A\in \diff(M,N)$, there is some $\ell\in [r]$ and a tuple $\abar\in A^{\underline{\ell}}$ such that $\abar \in \diff^{\ell}(M,N)$.  Define $\Psi:\diff(M,N)\rightarrow \bigcup_{\ell\in [r]}\diff^{\ell}(M,N)$ to be any map which sends each $A\in \diff(M,N)$ to some such tuple.  Given $\ell\in [r]$ and $\abar=(a_1,\ldots, a_{\ell})\in \diff^{\ell}(M,N)$, note that 
$$
\Psi^{-1}(\abar)\subseteq \{A\in {W\choose r}: \cup \abar \subseteq A\}.
$$
Since the right hand side has size ${n-\ell \choose r-\ell}$, we have that for all $\abar \in \diff^{\ell}(M,N)$, $|\Psi^{-1}(\abar)|\leq {n-\ell \choose r-\ell}$.

For each $\ell\in [r]$, we now define a map $f_{\ell}:\diff^{\ell}(M,N)\rightarrow \bigcup_{(R,p)\in Index, ||p||=\ell}DH_p^R(M)\Delta DH_p^R(N)$.  Let $\abar\in \diff^{\ell}(M,N)$.  Since $\abar\in \diff^{\ell}(M,N)$, there is a relation $R(x_1,\ldots, x_t)\in \calL_0$ and a map $h:[\ell]\rightarrow [t]$ such that $M\models R(a_{h(1)},\ldots, a_{h(t)})$ and $N\models \neg R(a_{h(1)},\ldots, a_{h(t)})$ or vice versa.  If $h$ is not surjective, then some permutation of $C_p(a_{h(1)},\ldots, a_{h(t)})$ is a proper subtuple $\abar'$ of $\abar$ such that $qftp^M(\abar')\neq qftp^N(\abar')$.  But this contradicts that $\abar \in \diff^{\ell}(M,N)$.  Thus $h$ is surjective.  Let $p$ be the partition of $[t]$ with parts $h^{-1}(\{1\}), \ldots, h^{-1}(\{\ell\})$.  Since $h$ is surjective, the parts are all nonempty, so $||p||=\ell$.  Then by definition, $C_p(a_{h(1)},\ldots, a_{h(t)}) \in DH^R_p(M)\Delta DH^R_p(N)$.  Define $f_{\ell}(\abar)=C_p(a_{h(1)},\ldots, a_{h(t)})$.  Observe that $\cup C_p(a_{h(1)},\ldots, a_{h(t)})=\cup \abar$ implies 
$$
f^{-1}_{\ell}(f_{\ell}(\abar))\subseteq \{\bbar\in W^{\ell}: \cup \bbar = \cup \abar\},
$$
so $|f^{-1}_{\ell}(f_{\ell}(\abar))|\leq \ell !$.  Thus $f_{\ell}: \diff^{\ell}(M,N)\rightarrow \bigcup_{(R,p)\in Index, ||p||=\ell}DH_p^R(M)\Delta DH_p^R(N)$ and 
\begin{align}\label{t}
\text{ for all }\cbar \in \bigcup_{(R,p)\in Index, ||p||=\ell}DH_p^R(M)\Delta DH_p^R(N),\quad |f_{\ell}^{-1}(\cbar)|\leq \ell!.
\end{align}  
Define a map $\beta: \diff(M,N)\rightarrow \bigcup_{(R,p)\in Index}DH_p^R(M)\Delta DH_p^R(N)$ as follows.  Given $A\in \diff(M,N)$, apply $\Psi$ to obtain $\Psi(A)\in \diff^{\ell}(M,N)$ for some $\ell\in [r]$.  Then define 
$$
\beta(\abar):=f_{\ell}(\Psi(\abar))\in \bigcup_{(R,p)\in Index, ||p||=\ell}DH_p^R(M)\Delta DH_p^R(N).
$$
Suppose $\cbar \in \bigcup_{(R,p)\in Index}DH_p^R(M)\Delta DH_p^R(N)$ and $\ell:=|\cbar|$.  Then $\cbar\in DH_p^R(M)\Delta DH_p^R(N)$ for some $(R,p)\in Index$ with $||p||=\ell$.  By definition of $\beta$, $\beta^{-1}(\cbar)=\Psi^{-1}(f^{-1}_{\ell}(\cbar))$.  Combining (\ref{t}) and the fact that $|\Psi^{-1}(\abar)|\leq {n-\ell\choose r-\ell}$ for all $\abar \in \diff^{\ell}(M,N)$, we have that 
$$
|\beta^{-1}(\cbar)|=|\Psi^{-1}(f^{-1}_{\ell}(\cbar))|\leq {n-\ell\choose r-\ell}\ell!.
$$
This shows that $|\diff(M,N)|\leq \sum_{\ell\in[r]} \sum_{(R,p)\in Index,||p||=\ell}{n-\ell\choose r-\ell}\ell!|DH_p^R(M)\Delta DH_p^R(N)|$.  Dividing both sides of this by ${n\choose r}$, we obtain the following.
\begin{align}\label{dd}
\dist(M,N)&\leq \sum_{\ell\in[r]} \sum_{(R,p)\in Index,||p||=\ell}\frac{{n-\ell\choose r-\ell}\ell!}{{n\choose r}}|DH_p^R(M)\Delta DH_p^R(N)|.
\end{align}
Note that for all $1\leq \ell <r$, 
$$
\frac{{n-\ell\choose r-\ell}\ell!}{{n\choose r}} = \frac{(n-\ell)!}{n!}\frac{\ell! r!}{(r-\ell)!}\leq \frac{2^{\ell}}{n^{\ell}}\frac{\ell! r!}{(r-\ell)!}<  \frac{(r!)^22^r}{n^{\ell}},
$$
where the first inequality is by (\ref{d}) and the last is because $\ell<r$.  If $\ell=r$, then 
$$
\frac{{n-\ell\choose r-\ell}\ell!}{{n\choose r}} =\frac{r!}{{n\choose r}}=\frac{(r!)^2(n-r)!}{n!} \leq \frac{(r!)^22^r}{n^r},
$$
where the inequality is by (\ref{d}).  Thus for all $\ell\in [r]$, $\frac{{n-\ell\choose r-\ell}\ell!}{{n\choose r}}\leq \frac{(r!)^22^r}{n^{\ell}}$.  Combining this with (\ref{dd}) yields
$$
\dist(M,N)\leq (r!)^22^r \sum_{\ell\in[r]} \sum_{(R,p)\in Index,||p||=\ell}\frac{|DH_p^R(M)\Delta DH_p^R(N)|}{n^\ell}=(r!)^22^rd(M,N).
$$
\end{proof}

\noindent We will use the following version of Theorem \ref{triangleremoval}, now adapted to our notation.  

\begin{theorem}\label{triangleremoval2}
Suppose $\mathcal{A}$ is a collection of finite $\mathcal{L}_0$-structures.  For every $\delta>0$ there exists $\epsilon>0$ and $K$ such that the following holds.  For all sufficiently large finite $\mathcal{L}_0$-structures $M$, if $\prob(\calA(K),M)<\epsilon$, then there is an $\calL_0$-structure $M'$ with $dom(M')=dom(M)$ such that $\dist(M',M)<\delta$ and $\prob(\calA,M')=0$.
\end{theorem}

\begin{proof}
Fix $\delta>0$.  Let $\delta'=\frac{\delta}{(r!)^22^r }$ and choose $K=K(\delta')$ and $\epsilon=\epsilon(\delta')$ by applying Theorem \ref{triangleremoval} to $\delta'$ and $\calA$.  Suppose $n$ is sufficiently large so that Theorem \ref{triangleremoval} applies to structures of size $n$.  Suppose $M$ is an $\mathcal{L}_0$-structure of size $n$ such that $\prob(\calA(K),M)<\epsilon$.  Then Theorem \ref{triangleremoval} implies there is an $\mathcal{L}_0$-structure $M'$ with $dom(M')=dom(M)$ such that $d(M',M)<\delta'$ and $\prob(\calA,M')=0$.  Combining this with Lemma \ref{distlem}, we have $\dist(M',M)\leq (r!)^22^r d(M',M)<(r!)^22^r \delta' = \delta$.
\end{proof}

\noindent{\bf Proof of Theorem \ref{GENSUPERSAT}.}
Let $\calH$ be a nontrivial hereditary $\calL$-property and let $\calF$ be as in Observation \ref{HP} so that $\calH=\Forb(F)$.  Recall we want to show that for all $\delta>0$, there are $\epsilon>0$ and $K$ such that for sufficiently large $n$, for any $\calL_{\calH}$-template $M$ of size $n$, if $\prob(\tilde{\calF}(K) \cup \calE(K),M)\leq \epsilon$ then
\begin{enumerate}
\item If $\pi(\calH)>1$, then $sub(M)\leq \ex(n,\calH)^{1+\delta}$.
\item If $\pi(\calH)\leq 1$, then $sub(M)\leq 2^{\delta {n\choose r}}$.
\end{enumerate}
Fix $\delta>0$.  Apply Lemma \ref{deltaclose1} to $\calH$ to obtain $\gamma>0$.  Let $\mathcal{A}=\tilde{\calF}\cup \calE\cup \FLAW$.  Apply Theorem \ref{triangleremoval2} to obtain $K$ and $\epsilon$ for $\delta/2\gamma$ and $\calA$.  Suppose $n$ is sufficiently large and $M$ is an $\calL_{\calH}$-template of size $n$ satsifying $\prob(\tilde{\calF}(K)\cup \calE(K), M)<\epsilon$.  Because $M$ is an $\calL_{\calH}$-template, Lemmas \ref{flaw} implies for all $B\in \FLAW$, $\prob(B,M)=0$. Therefore $\prob(\calA(K), M)<\epsilon$, so by Theorem \ref{triangleremoval2}, there is an $\calL_{\calH}$-structure $M'$ with $dom(M)=dom(M')$ such that $\prob(\calA,M')=0$ and $\dist(M,M')\leq \delta/2\gamma$.  Since $\prob(\calA,M')=0$, Corollary \ref{charR} implies $M'\in \calR(n,\calH)$.  Thus $sub(M')\leq \ex(n,\calH)$ holds by definition of $\ex(n,\calH)$.  Combining this with Lemma \ref{deltaclose1} (note $M'\in \calR(n,\calH)$ implies $M'$ is error-free), we have the following.
\begin{enumerate}
\item If $\pi(\calH)>1$, then $sub(M)\leq sub(M')\ex(n,\calH)^{\gamma (\delta/2\gamma)}=sub(M')\ex(n,\calH)^{\delta/2}\leq \ex(n,\calH)^{1+\delta/2}$.
\item If $\pi(\calH)=1$, then $sub(M)\leq sub(M')2^{\gamma (\delta/2\gamma) {n\choose r}}=sub(M')2^{\delta{n\choose r}/2}\leq \ex(n,\calH)2^{\delta/2  {n\choose r}}$.
\end{enumerate}
We are done in the case where $\pi(\calH)>1$.  If $\pi(\calH)=1$, assume $n$ is sufficiently large so that $\ex(n,\calH)\leq 2^{\delta/2{n\choose r}}$.  Then (2) implies $sub(M)\leq 2^{\delta{n\choose r}}$, as desired.
\qed

\vspace{3mm}

\noindent {\bf Proof of Theorem \ref{COROLLARY2} from Theorem \ref{version1}.} Suppose $\calH$ is a hereditary $\calL$-property.  Let $\calF$ be the class of finite $\calL$-structures from Observation \ref{HP} so that $\calH=\Forb(\calF)$.  Then for each $n$, $\mathcal{H}_n$ is the set of all $\calF$-free $\calL$-structures with domain $[n]$.  Let $\mathcal{A}=\tilde{\calF}\cup \calE\cup \FLAW$.  Fix $\delta>0$ and choose $K$ and $\epsilon$ as in Theorem \ref{triangleremoval2} for $\delta$ and the family $\mathcal{A}$. By replacing $K$ if necessary, assume $K\geq r$.  Apply Theorem \ref{version1} to $\calB:=\mathcal{F}(K)$ to obtain $c=c(K,r,\calL,\epsilon)$, $m=m(K,r)$.  Observe the choice of $K$ depended on $\calH$ and $r=r_{\calL}$, so $m=m(\calH, r_{\calL})$.  Since $r_{\calL}$ depends on $\calL$, $c=c(\calH,\calL,\epsilon)$.  Let $n$ be sufficiently large.  Then Theorem \ref{version1} applied to $W=[n]$ implies there is a collection $\calC$ of $\calL_{\calB}$-templates with domain $[n]$ such that the following hold.
\begin{enumerate}[(i)]
\item For all $\calF(K)$-free $\calL$-structures $M$ with domain $[n]$, there is $C\in \calC$ such that $M\unlhd_pC$.
\item For all $C\in \mathcal{C}$, $\prob(\widetilde{\mathcal{F(K)}},C)\leq \epsilon$ and $\prob(\calE,C)\leq \epsilon$.
\item $\log |\mathcal{C}|\leq cn^{r-\frac{1}{m}}\log n$.
\end{enumerate}
We show this $\calC$ satisfies the conclusions of Theorem \ref{COROLLARY2} with $c$, $m$ and $\delta$.  Note that because $K\geq r$, $S_r(\calH)=S_r(\calB)$, so all $\calL_{\calB}$-templates are also $\calL_{\calH}$-templates. In particular the elements in $\calC$ are all $\calL_{\calH}$-templates.  Clearly (iii) implies part (3) of Theorem \ref{COROLLARY2} holds.  For part (1), since any $H\in \calH_n$ is $\calF$-free, it is also $\calF(K)$-free, so (i) implies there is $C\in \calC$ such that $H\unlhd_pC$.  This shows part (1) of Theorem \ref{COROLLARY2} holds.  For part (2), fix $C\in \calC$.  Since $C$ is an $\calL_{\calH}$-template, Lemma \ref{flaw} implies $\prob(G,C)=0$ for all $G\in \FLAW$.  Then (ii) implies that for all $G\in \widetilde{\calF(K)}\cup \calE$, $\prob(G,C)\leq \epsilon$.  Since $\widetilde{\calF(K)}=\tilde{\calF}(K)$, these facts imply that for all $G\in \mathcal{A}(K)$, $\prob(G,C)\leq \epsilon$.   Thus Theorem \ref{triangleremoval2} implies there is an $\calL_{\calH}$-structure $C'$ with $dom(C)=dom(C')=[n]$ such that $\dist(C,C')\leq \delta$ and $\prob(\calA,C')=0$.  Since $\prob(\calA,C')=0$, $C'$ is a $\FLAW$-free, $\tilde{\calF}$-free, and error-free $\calL_{\calH}$-structure with domain $[n]$, so by Corollary \ref{charR}, $C'\in \calR([n],\calH)$. This finishes the proof.
\qed

%************************************************************************
\section{A Reduction}\label{section2thm1}
%************************************************************************
We have now proved all the results in this paper except Theorem \ref{version1}.  In this section we prove Theorem \ref{version1} by reducing it to another result, Theorem \ref{VERSION2} (which is proved in Section \ref{VERSION2pf}).  

%************************************************************************
\subsection{Preliminaries}\label{prelims} 
%************************************************************************
In this subsection we give preliminaries necessary for the statement of Theorem \ref{VERSION2}.  Many of these notions are similar to definitions from Section \ref{tildeLstructures}.  However, we will see that our proofs necessitate this more syntactic treatment.

\begin{definition}\label{chd2}
Suppose $C$ is a set of constants and $\sigma \subseteq S_r(C)$.
\begin{enumerate}[$\bullet$]
\item $V(\sigma)=\{c\in C: c$ appears in some $p(\cbar)\in \sigma\}$.
\item Given $A\in {V(\sigma)\choose r}$, let $Ch_{\sigma}(A)=\{p(\cbar) \in \sigma: \cup \cbar =A\}$.  Elements of $Ch_{\sigma}(A)$ are \emph{choices for $A$}.
\item We say $\sigma$ is \emph{complete} if $Ch_{\sigma}(A)\neq \emptyset$, for all $A\in {V(\sigma)\choose r}$.
\end{enumerate}
\end{definition}

\begin{example}\label{ex6}
Let $\calL$ and $\calP$ be as in Example \ref{ex1} (i.e. metric spaces with distances in $\{1,2,3\}$).  Let $W=\{u,v,w\}$ and $\sigma=\{p_1(c_u,c_v), p_2(c_u,c_v), p_2(c_u,c_w)\}\subseteq S_2(C_W,\calP)$.  Then $V(\sigma)=\{c_u,c_v,c_w\}$ and it is easy to check that $Ch_{\sigma}(c_uc_v)=\{p_1(c_u,c_v), p_2(c_u,c_v)\}$, $Ch_{\sigma}(c_uc_w)=\{p_2(c_u,c_w)\}$, and $Ch_{\sigma}(c_vc_w)=\emptyset$.  Observe, this $\sigma$ is not complete.
\end{example}

\begin{definition}\label{syndiag}
Suppose $C$ is a set of $n$ constants and $\sigma \subseteq S_r(C)$.  Given $m\leq n$, $\sigma$ is a \emph{syntactic $m$-diagram} if $|V(\sigma)|=m$ and for all $A\in {V(\sigma)\choose r}$, $|Ch_{\sigma}(A)|=1$.  
\end{definition}
We will say $\sigma\subseteq S_r(C)$ is a \emph{syntactic type diagram} if it is a syntactic $|V(\sigma)|$-diagram.

\begin{example}\label{ex7}
Let $\calL$ and $\calP$ be as in Example \ref{ex6}, and let $W=\{t,u,v,w\}$ be a set of size $4$.  Set $\sigma'=\{p_1(c_u,c_v), p_2(c_u,c_w), p_3(c_v,c_w)\}\subseteq S_2(C_W,\calP)$.  Then $V(\sigma')=\{c_u,c_v,c_w\}$ and we have that $Ch_{\sigma}(c_uc_v)=\{p_1(c_u,c_v)\}$, $Ch_{\sigma}(c_uc_w)=\{p_2(c_u,c_w)\}$, and $Ch_{\sigma}(c_vc_w)=\{ p_3(c_v,c_w)\}$.  This shows $\sigma'$ is a syntactic $3$-diagram. 
\end{example}

Observe that if $\sigma$ is a syntactic $m$-diagram, then by definition, $|V(\sigma)|=m$ and $|\sigma|={m\choose r}$.  Given a tuple of constants $\cbar=(c_1,\ldots, c_k)$, a first-order language $\calL_0$ containing $\{c_1,\ldots, c_k\}$, and an $\calL_0$-structure $M$, let $\cbar^M$ denote the tuple $(c_1^M,\ldots, c_k^M)\in dom(M)^k$.  

\begin{definition}\label{tpd}
Suppose $C$ is a set of constants and $\sigma \subseteq S_r(C)$.
\begin{enumerate}
\item If $M$ is an $\calL\cup V(\sigma)$-structure, write $M\models \sigma^M$ if $M\models p(\cbar^M)$ for all $p(\cbar)\in \sigma$.   Call $\sigma$ \emph{satisfiable} if there exists an $\calL\cup V(\sigma)$-structure $M$ such that $M\models \sigma^M$.
\item If $M$ is an $\calL\cup C$-structure, the \emph{type-diagram} of $M$ is the set 
$$
Diag^{tp}(M, C)=\{p(\cbar)\in S_r(C): M\models p(\cbar^M)\}.
$$
\end{enumerate}
\end{definition}
Suppose that $M$ is an $\calL$-structure with $dom(M)=W$.  The \emph{canonical type-diagram of $M$} is 
$$
Diag^{tp}(M)=\{p(c_{\abar})\in S_r(C_W): M\models p(\abar)\}.
$$
In other words, $Diag^{tp}(M)=Diag^{tp}(M,C_W)$ where $M$ is considered with its natural $\calL\cup C_W$-structure.  Observe that $Diatg^{tp}(M)$ is always a syntactic $|dom(M)|$-diagram.  The difference between $Diag^{tp}(M)$ and $Diag(M)$ is that elements of $Diag^{tp}(M)$ are types (with constants plugged in for the variables) while the elements of $Diag(M)$ are formulas (with constants plugged in for the variables).  Clearly $Diag(M)$ and $Diag^{tp}(M)$ contain the same information.

\begin{example}
Let $\calL$ and $\calP$ be as in Example \ref{ex7}.  Let $W=\{u,v,w\}$ and let $M$ be the $\calL$-structure with domain $W$ satisfying  $M\models p_1(u,v)\cup p_2(u,w)\cup p_3(v,w)$.  Then $Diag^{tp}(M)$ is the set $\{p_1(c_u,c_v), p_2(c_u,c_w), p_3(c_v,c_w)\}$, while $Diag(M)$ is the set of all $\calL\cup C_W$-sentences implied by $p_1(c_u,c_v)\cup p_2(c_v,c_w)\cup p_3(c_u,c_w)$.
\end{example}

We now make a few observations which will be used in the remainder of the chapter.

\begin{observation}\label{ob1}
Suppose $M$ is an $\calL$-structure with domain $W$ of size $n$.  Then the following hold.
\begin{enumerate}
\item Suppose $m\leq n$, $\sigma\subseteq S_r(C_W)$ is a syntactic $m$-diagram, and $N$ is an $\calL\cup V(\sigma)$-structure of size $m$.  Then $N\models \sigma^N$ if and only if $\sigma=Diag^{tp}(N, V(\sigma))$.    
\item Suppose $N$ is an $\mathcal{L}\cup C_W$-structure of size $n$ and $N\models Diag^{tp}(M)$.  Then $M\cong_{\calL}N$.
\item If $\sigma \subseteq S_r(C_W)$ and $Diag^{tp}(M)\subseteq \sigma$, then $\sigma$ is complete.
\end{enumerate}
\end{observation}
\begin{proof}
(1): Suppose first $\sigma=Diag^{tp}(N,V(\sigma))$.  Then by Definition \ref{tpd}, $N\models \sigma^{N}$.  Converesly, suppose $N\models \sigma^N$. By Definition \ref{tpd}, this implies $\sigma\subseteq Diag^{tp}(N,V(\sigma))$.  To show the reverse inclusion, suppose $p(c_{\abar})\in Diag^{tp}(N,V(\sigma))$.  By Definition \ref{tpd}, $N\models p(\cbar^N)$.  Let $A=\cup \cbar^N\in {dom(N)\choose r}$ (since $p\in S_r(\calL)$ is proper and $N\models p(\cbar^N)$, the coordinates of $\cbar^N$ must all be distinct). Since $\sigma$ is a syntactic $m$-diagram, $|Ch_{\sigma}(A)|=1$, so there is $p'(\xbar)\in S_r(\calL)$ and $\mu\in Perm(r)$ such that $p'(\mu(\cbar))\in \sigma$.  Since $N\models \sigma^N$, this implies $N\models p'(\mu(\cbar^N))$.  Clearly $N\models p(\cbar^N)$ and $N\models p'(\mu(\cbar^N))$ implies $p(\cbar^N)= p'(\mu(\cbar^N))$.  So we have $p(\cbar)=p'(\mu(\cbar))\in \sigma$ as desired.

(2): Clearly the map $f:W\rightarrow dom(N)$ sending $a\mapsto c_a^N$ is an $\calL$-homomorphism of $M$ into $N$.  Since by assumption, $M$ and $N$ both have size $n$, it must be a bijection, and thus an $\calL$-isomorphism.  

(3): For each $A\in {C_W\choose r}$, $Diag^M(A)\in Ch_{\sigma}(A)$ implies $Ch_{\sigma}(A)\neq \emptyset$.
\end{proof}

\begin{definition}
Suppose $\mathcal{A}$ is a collection of finite $\mathcal{L}$-structures and $C$ is a set of constant symbols.
\begin{enumerate}
\item We say $\sigma\subseteq S_r(C)$ is $\mathcal{A}$-satisfiable if $M\models \sigma^M$ and there is an $\mathcal{L}\cup V(\sigma)$-structure $M$ such that $M\upharpoonright_{\mathcal{L}}\in \mathcal{A}$.
\item Define $Diag^{tp}(\mathcal{A},C)= \{\sigma \subseteq S_r(C): \sigma $ is a syntactic type diagram which is $\mathcal{A}$-satisfiable$\}$.
\item Given $\sigma\subseteq S_r(C)$, set $Span(\sigma)=\{\sigma'\subseteq \sigma: \sigma'\text{ is a syntactic type diagram}\}$.
\end{enumerate}
\end{definition}

\begin{example}\label{ex8}
Let $\calL$ and $\calP$ be as in Example \ref{ex7}.  Let $C=\{c_1, c_2, c_3\}$ be a set of three constant symbols.  Then $\sigma\subseteq S_2(C, \calP)$ is a syntactic $3$-diagram if and only if $\sigma=\{p_i(c_1,c_2), p_j(c_1,c_3), p_k(c_2, c_3)\}$ for some $i,j,k\in [3]$.  Clearly such a $\sigma$ is $\calP$-satisfiable if and only if $|i-j|\leq k\leq i+j$, that is, if and only if the numbers $i,j,k$ do not violate the triangle inequality.  Thus $Diag^{tp}(\calP, C)$ consists of sets of the form $\sigma=\{p_i(c_1,c_2), p_j(c_1,c_3), p_k(c_2, c_3)\}$ where $i,j,k\in [3]$ satisfy $|i-j|\leq k\leq i+j$.

Suppose now that $\sigma=\{p_1(c_1,c_2), p_2(c_1,c_2), p_3(c_1,c_2), p_1(c_2,c_3), p_1(c_1,c_3)\}$.  Then $Span(\sigma)$ consists of the following syntactic $3$-diagrams.
\begin{enumerate}
\item $\{p_1(c_1,c_2), p_1(c_2,c_3), p_1(c_1,c_3)\}$.
\item $\{p_2(c_1,c_2), p_1(c_2,c_3), p_1(c_1,c_3)\}$.
\item $\{p_3(c_1,c_2), p_1(c_2,c_3), p_1(c_1,c_3)\}$.
\end{enumerate}
Observe that (1) and (2) are $\calP$-satisfiable, while (3) is not.
\end{example}

\noindent For the rest of this subsection, $\calH$ is a fixed collection of finite $\calL$-structures.

\begin{lemma}\label{lem0}
Suppose $X\subseteq W$ are finite sets, $M$ is a complete $\calL_{\calH}$-structure with domain $X$, and $\chi\in Ch(M)$.  Set $\sigma :=\{\chi(A): A\in {X\choose r}\}\subseteq S_r(C_W, \calH)$.  Then 
\begin{enumerate}
\item $\sigma$ is a syntactic $|X|$-diagram.
\item If $F\unlhd_{\chi}M$ then $\sigma=Diag^{tp}(F)$. 
\end{enumerate}
\end{lemma}
\begin{proof}
Clearly $V(\sigma)=C_X$.  Let $m=|C_X|$.  Note ${C_X\choose r}=\{C_A: A\in {X\choose r}\}$.  By definition of $\sigma$, for each $A\in {X\choose r}$, $\{\chi(A)\}=Ch_{\sigma}(C_A)$.  Thus $|Ch_{\sigma}(C_A)|=1$ for all $A\in {X\choose r}$ and $\sigma$ is a syntactic $m$-diagram.  This shows 1 holds.  For 2, suppose $F\unlhd_{\chi}M$. This means $dom(F)=X$ and for all $A\in {X\choose r}$, $Diag^F(A)=\chi(A)$.  Clearly this implies $F\models \sigma^F$, where $F$ is considered with its natural $C_{X}$-structure.  Part 1 of Observation \ref{ob1} then implies $\sigma=Diag^{tp}(F)$.
\end{proof}

\begin{definition}
Given an integer $\ell$ and a set of constants $C$, set 
$$
Err_{\ell}(C)=\{\sigma\subseteq S_r(C): \text{$\sigma$ is an unsatisfiable syntactic $\ell$-diagram}\}.
$$
We call the elements of $Err_{\ell}(C)$ \emph{syntactic $C$-errors of size $\ell$}.
\end{definition}

\begin{example}\label{ex9}
Let $\calL=\{R_1,R_2,R_3,E\}$ and $\calP$ be as in Example \ref{errorex}.  Let $C=\{c_1,c_2,c_3,c_4\}$ be a set of constants.  Recall from Example \ref{ex9}, that $q_1(c_1,c_2,c_3)\cup q_2(c_1,c_2,c_4)$ is unsatisfiable.  Therefore an example of a syntactic $C$-error of size $4$ is the set $\{q_1(c_2,c_3,c_4), q_2(c_1,c_2,c_3), q_1(c_1,c_3,c_4), q_1(c_1,c_2,c_4) \}$.
\end{example}

\begin{lemma}\label{lem1}
Suppose $W$ is finite a set, $r+1\leq \ell<2r$, and $M$ is a complete $\calL_{\calH}$-structure which is an error of size $\ell$ and with domain $X\subseteq W$.  Then there is a choice function $\chi\in Ch(M)$ such that $\{\chi(A): A\in {X\choose r}\}$ is a syntactic $C_W$-error of size $\ell$.
\end{lemma}
\begin{proof}
Since $M$ is an error of size $\ell$ then there are $\abar_1$, $\abar_2 \in X^{\underline{r}}$ such that $\cup \abar_1\bigcup \cup \abar_2=X$ and $p_1(\xbar), p_2(\xbar)\in S_r(\calH)$ such that $M\models R_{p_1}(\abar_1) \wedge R_{p_2}(\abar_2)$ but $p_1(c_{\abar_1})\cup p_2(c_{\abar_2})$ is unsatisfiable.  Define a function $\chi:{X\choose r}\rightarrow S_r(C_W,\calH)$ as follows.  Set $\chi(\cup \abar_1)=p(c_{\abar_1})$ and $\chi(\cup \abar_2)=p(c_{\abar_2})$.  For all other $A\in {X\choose r}$ choose any $\chi(A)\in Ch_M(A)$ (this is possible because $M$ is a complete).  By construction, $\chi\in Ch(M)$.  By part 1 of Lemma \ref{lem0}, $\sigma:=\{\chi(A): A\in {X\choose r}\}$ is a syntactic $\ell$-diagram.  Because $\sigma$ contains $p_1(c_{\abar_1})\cup p_2(c_{\abar_2})$, it is unsatisfiable.  By definition, $\sigma$ is a syntactic $C_W$-error of size $\ell$.
\end{proof}

%************************************************************************
\subsection{Proof of Theorem \ref{version1}}\label{section3thm1}
%************************************************************************

In this section we state Theorem \ref{VERSION2} and use it to prove Theorem \ref{version1}. 

\begin{theorem}\label{VERSION2}
Let $0<\epsilon<1$.  For all $k\geq r$, there is a positive constant $c=c(k,r,\calL,\epsilon)$ and $m=m(k,r)>1$ such that for all sufficiently large $n$ the following holds.  Suppose $\mathcal{F}$ is a collection of finite $\mathcal{L}$-structures, each of size at most $k$, and $\calH:=\Forb(\calF)\neq\emptyset$.  For any set $W$ of size $n$, there is a set $\Sigma \subseteq \mathcal{P}(S_r(C_W,\calH))$ such that the following hold.
\begin{enumerate}
\item For all $\mathcal{F}$-free $\mathcal{L}$-structures $M$ with domain $W$, there is $\sigma\in \Sigma$ such that $Diag^{tp}(M)\subseteq \sigma$.
\item For all $\sigma\in \Sigma$ the following hold.  For each $1\leq \ell\leq k$,  $|Diag^{tp}(\calF(\ell), C_W)\cap Span(\sigma)|\leq \epsilon{n\choose \ell}$, 
and for each $r+1\leq \ell \leq 2r$, $|Err_{\ell}(C_W)\cap Span(\sigma)|\leq \epsilon{n\choose \ell}$.
\item $\log |\Sigma|\leq cn^{r-\frac{1}{m}}\log n$.
\end{enumerate}
\end{theorem}

Given a collection $\calH$ of finite $\calL$-structures, we now define a way of building an $\calL_{\calH}$-template from a complete subset of $S_r(C_W,\calH)$.

\begin{definition}\label{D_Cdef}
Suppose $\calH$ is a nonempty collection of $\calL$-structures, $W$ is a set, and $\sigma\subseteq S_r(C_W,\calH)$ is such that $V(\sigma)=C_W$.  Define an $\calL_{\calH}$-structure $D_{\sigma}$ as follows.  Set $dom(D_\sigma)=W$ and for each $\abar \in W^{r}$, define $D_{\sigma}\models R_p(\abar)$ if and only if $p(c_{\abar})\in Ch_{\sigma}(\cup \abar)$.
\end{definition}

In the notation of Definition \ref{D_Cdef}, note that for all $A\in {W\choose r}$, $Ch_{D_{\sigma}}(A)=Ch_{\sigma}(A)$ (here $Ch_{D_{\sigma}}(A)$ is in the sense of Definition \ref{chdef} and $Ch_{\sigma}(A)$ is in the sense of Definition \ref{chd2}).  We now prove two lemmas.

\begin{lemma}\label{templateclaim}
Suppose $\calF$ is collection of finite $\calL$-structures and $\calH=\Forb(\calF)\neq \emptyset$.  For any set $W$ and complete $\sigma\subseteq S_r(C_W,\calH)$, $D_{\sigma}$ is an $\calL_{\calH}$-template.
\end{lemma}
\begin{proof} First, observe that $D_{\sigma}$ is a complete $\calL_{\calH}$-structure since for each $A\in {W\choose r}$, $Ch_{D_{\sigma}}(A)=Ch_{\sigma}(C_A)$, and $Ch_{\sigma}(C_A)\neq \emptyset$ because $\sigma$ is complete by assumption (in the sense of Definition \ref{chd2}).  Suppose now $\abar\in W^r\setminus W^{\underline{r}}$.  Then because $S_r(\calH)$ contains only proper types, there is no $p(\xbar)\in S_r(\calH)$ such that $p(c_{\abar})\in S_r(C_W,\calH)$.  Thus $D_{\sigma}\models \neg R_p(\abar)$ for all $p(\xbar)\in S_r(\calH)$, so $D_{\sigma}$ satisfies part (1) of Definition \ref{templatedef}.    Suppose $p(\xbar), p'(\xbar)\in S_r(\calH)$ and $\mu\in Perm(r)$ are such that $p(\xbar)=p'(\mu(\xbar))$.  Suppose $a\in W^{\underline{r}}$.  Then by definition of $D_{\sigma}$, $D_{\sigma}\models R_p(\abar)$ if and only if $p(c_{\abar})\in \sigma$.  Since $p(c_{\abar})=p'(c_{\mu(\abar)})$, $p(c_{\abar})\in \sigma$ if and only if $p'(c_{\mu(\abar)})\in \sigma$. By definition of $D_{\sigma}$, $p'(c_{\mu(\abar)})\in \sigma$ if and only if $D_{\sigma}\models R_{p'}(\mu(\abar))$.  Thus we've shown $D_{\sigma}\models R_p(\abar)$ if and only if $D_{\sigma}\models R_{p'}(\mu(\abar))$, so $D_{\sigma}$ satisfies part (2) of Definition \ref{templatedef}.  This finishes the verification that $D_{\sigma}$ is an $\calL_{\calH}$-template.
\end{proof}

\begin{lemma}\label{boundinglem}
Suppose $k\geq r$, $W$ is a finite set, $\calH$ is a nonempty collection of finite $\calL$-structures, and $\sigma \subseteq S_r(C_W,\calH)$ is complete.  Suppose $\mathcal{F}$ is a collection of finite $\calL$-structures, each of size at most $k$.  Then for each $1\leq \ell \leq k$, there is an injection 
$$
\Phi: \cop(\tilde{\mathcal{F}}(\ell), D_{\sigma})\rightarrow Diag^{tp}(\mathcal{F}(\ell), C_W)\cap Span (\sigma).
$$
and for each $r+1\leq \ell\leq 2r$, there is an injection
$$
\Theta: \cop(\calE(\ell), D_{\sigma})\rightarrow Err_{\ell}(C_W)\cap Span (\sigma).
$$
\end{lemma}
\begin{proof}
Without loss of generality, assume $\mathcal{F}$ is closed under isomorphism (we can do this because it does not change either of the sets $\cop(\tilde{\mathcal{F}}(\ell), D_{\sigma})$ or $Diag^{tp}(\mathcal{F}(\ell), C_W)\cap Span (\sigma)$).  Suppose $1\leq \ell \leq k$ and $G\in \cop(\tilde{\calF}(\ell), D_{\sigma})$.  Then $G\subseteq_{\calL_{\calH}}D_{\sigma}$ and $G\cong_{\calL_{\calH}}B$, for some $B\in \tilde{\calF}(\ell)$. It is straightforward to check that since $\mathcal{F}$ is closed under isomorphism, this implies $G\in \tilde{\calF}(\ell)$.  So without loss of generality we may assume that $B=G$.  Then there is some $F\in \mathcal{F}(\ell)$ such that $F\unlhd_pG$.  Choose any such $F$ and let $\chi\in Ch(G)$ be such that $F\unlhd_{\chi}G$.  Define $\Phi(G)=\{\chi(A):  A\in {dom(G)\choose r}\}$.  By part 2 of Lemma \ref{lem0}, $\Phi(G)=Diag^{tp}(F)$.  Thus by definition, $\Phi(G)\in Diag^{tp}(\calF(\ell), C_W)$. By definition of $D_{\sigma}$ and because $\chi\in Ch(G)$, $G\subseteq_{\calL_{\calH}} D_{\sigma}$ implies $\Phi(G)\subseteq \sigma$, so $\Phi(G)\in Diag^{tp}(\calF(\ell),C_W)\cap Span (\sigma)$, as desired.  To see that $\Phi$ is injective, note that for all $G\in \cop(\tilde{\mathcal{F}}(\ell), D_{\sigma})$, $V(\Phi(G))=dom(G)$.  Therefore if $G_1\neq G_2\in \cop(\tilde{\mathcal{F}}(\ell), D_{\sigma})$, $dom(G_1)\neq dom(G_2)$ implies $V(\Phi(G_1))\neq V(\Phi(G_2))$, so $\Phi(G_1)\neq \Phi(G_2)$.

Suppose now $r+1\leq \ell\leq 2r$ and $G\in \cop(\calE(\ell), D_{\sigma})$.  Then $G$ is a complete $\calL_{\calH}$-structure which is an error of size $\ell$.  Lemma \ref{lem1} implies there is $\chi\in Ch(G)$ such that $\{\chi(A): A\in {dom(G)\choose r}\}$ is a syntactic $C_W$-error of size $\ell$.  Set $\Theta(G)=\{\chi(A): A\in {dom(G)\choose r}\}$.  Then this shows $\Theta(G)\in Err_{\ell}(C_W)$.  By definition of $D_{\sigma}$ and because $\chi\in Ch(G)$, $G\subseteq_{\calL_{\calH}}D_{\sigma}$ implies $\Theta(G)\subseteq \sigma$, so $\Theta(G)\in Err_{\ell}(C_W)\cap Span(\sigma)$, as desired.  To see that $\Theta$ is injective, note that for all $G\in \cop(\calE(\ell), D_{\sigma})$, $V(\Theta(G))=dom(G)$.  Therefore if $G_1\neq G_2\in \cop(\mathcal{E}(\ell), D_{\sigma})$, $dom(G_1)\neq dom(G_2)$ implies $V(\Theta(G_1))\neq V(\Theta(G_2))$, so $\Theta(G_1)\neq \Theta(G_2)$.\end{proof}

\noindent {\bf Proof of Theorem \ref{version1} from Theorem \ref{VERSION2}.}
Let $0<\epsilon<1$ and let $k\geq r$ be an integer.  Choose the constants $c=c(k,r,\calL,\epsilon)$ and $m=m(k,r)$ to be the ones given by Theorem \ref{VERSION2}.  Suppose $\calF$ is a collection of finite $\calL$-structures, each of size at most $k$, and $\calB:=\Forb(\calF)\neq \emptyset$.  Suppose $n$ is sufficiently large and $W$ is a set of size $n$.  Theorem \ref{VERSION2} applied to $\calB$ implies there exists a set $\Sigma\subseteq \mathcal{P}(S_r(C_W,\calB))$ such that the following hold.
\begin{enumerate}[(i)]
\item For all $\mathcal{F}$-free $\mathcal{L}$-structures $M$ with domain $W$, there is $\sigma\in \Sigma$ such that $Diag^{tp}(M)\subseteq \sigma$.
\item For all $\sigma\in \Sigma$ the following hold.  For each $1\leq \ell\leq k$,  $|Diag^{tp}(\calF(\ell), C_W)\cap Span(\sigma)|\leq \epsilon{n\choose \ell}$, 
and for each $r+1\leq \ell \leq 2r$, $|Err_{\ell}(C_W)\cap Span(\sigma)|\leq \epsilon{n\choose \ell}$.
\item $\log |\Sigma|\leq cn^{r-\frac{1}{m}}\log n$.
\end{enumerate}
Set $\mathcal{D}=\{D_{\sigma}: \sigma \in \Sigma\}$, where for each $\sigma\in \Sigma$, $D_{\sigma}$ is the $\calL_{\calB}$-structure from Definition \ref{D_Cdef}.  We claim this $\mathcal{D}$ satisfies conclusions of Theorem \ref{version1}.  First note (i) and part 3 of Observation \ref{ob1} imply that every $\sigma\in \Sigma$ is complete in the sense of Definition \ref{chd2}.  Therefore Lemma \ref{templateclaim} implies each $D_{\sigma}\in \calD$ is an $\calL_{\calB}$-template.  We now verify parts (1)-(3) of Theorem \ref{version1} hold for this $\calD$.

Clearly $|\mathcal{D}|\leq |\Sigma|$, so (iii) implies part (3) of Theorem \ref{version1} is satisfied.  Suppose now $M$ is an $\mathcal{F}$-free $\mathcal{L}$-structure with $dom(M)=W$. By (i), there is $\sigma\in \Sigma$ such that $Diag^{tp}(M)\subseteq \sigma$.  We claim that $M\unlhd_pD_{\sigma}$.  Let $A \in {W\choose r}$ and suppose $p(\xbar)\in S_r(\calH)$ is such that $M\models p(\abar)$ for some enumeration $\abar$ of $A$.  Then $Diag^M(A)=p(c_{\abar})\in Diag^{tp}(M)\subseteq \sigma$ implies by definition of $D_{\sigma}$, $D_{\sigma}\models R_p(\abar)$, so $p(c_{\abar})\in Ch_{D_{\sigma}}(A)$.  This shows $M\leq_p D_{\sigma}$.  Then $M\unlhd_pD_{\sigma}$ holds because by assumption $dom(M)=dom(D_{\sigma})=W$.  Thus part (1) of Theorem \ref{version1} is satisfied.  

We now verify part (2) of Theorem \ref{version1}.  Let $D_{\sigma}\in \mathcal{D}$.  We need to show $\prob(\tilde{\mathcal{F}},D_{\sigma})\leq \epsilon$ and $\prob(\calE, D_{\sigma})\leq \epsilon$.  For each $1\leq \ell \leq k$, we have
$$
|\cop(\tilde{\mathcal{F}}(\ell), D_{\sigma})|\leq |Diag^{tp}(\mathcal{F}(\ell),C_W)\cap Span(\sigma)|\leq \epsilon{n\choose \ell},
$$
where the first inequality is because of Lemma \ref{boundinglem} and the second inequality is by (ii).  This implies that for all $1\leq \ell \leq k$, $|\cop(\tilde{\mathcal{F}}(\ell), D_{\sigma})|\leq \epsilon {n\choose \ell}$, so $\prob(\tilde{\calF}(\ell),D_{\sigma})\leq \epsilon$.  Since every element in $\tilde{\calF}$ has size at most $k$, we have $\prob(\tilde{\calF}, D_{\sigma})\leq \epsilon$.  Similarly, for each $r+1\leq \ell \leq 2r$, 
$$
|\cop(\calE(\ell), D_{\sigma})|\leq|Err_{\ell}(C_W)\cap Span (\sigma)|\leq \epsilon {n\choose \ell},
$$
where the first inequality is by Lemma \ref{boundinglem} and the second inequality is by (ii).  This implies for all $r+1\leq \ell \leq 2r$, $|\cop(\calE(\ell), D_{\sigma})|\leq \epsilon{n\choose \ell}$, so $\prob(\calE(\ell), D_{\sigma})\leq \epsilon$.  Since every element in $\calE$ has size at least $r+1$ and at most $2r$, we have $\prob(\calE, D_{\sigma})\leq \epsilon$. This finishes the proof.
\qed

%******************************************************************************************************************************************
\section{Applying Hypergraph Containers to Prove Theorem \ref{VERSION2}}\label{VERSION2pf}
%******************************************************************************************************************************************

In this section we prove Theorem \ref{VERSION2}.  We will use the hypergraph containers theorem.  We begin with a definition. 

\begin{definition}
Suppose $K$ is a positive integer and $\mathcal{A}$ is a collection of finite $\mathcal{L}$-structures each of size at most $K$.  Set 
$$
cl_K(\mathcal{A})=\{M:\text{$M$ is an $\calL$-structure of size $K$ such that $prob(\calA, M)>0\}$}.
$$
\end{definition}
Observe that in the above notation, an $\mathcal{L}$-structure of size at least $K$ is $\mathcal{A}$-free if and only if it is $cl_K(\mathcal{A})$-free.  We  now state a key lemma.

\begin{lemma}\label{lemma**}
Assume $n\geq k\geq r$ and $\calF$ is a nonempty collection of $\calL$-structures, each of size at most $k$.  Suppose $\calH:=\Forb(\calF)\neq \emptyset$ and $W$ is a set of size $n$.  Fix $0<\epsilon<1/2$.  Suppose $\sigma\subseteq S_r(C_W,\calH)$ is complete and satisfies $V(\sigma)=C_W$.  If 
$$
|(Diag^{tp}(cl_k(\calF),C_W)\cup Err_k(C_W))\cap Span(\sigma)|\leq \epsilon {n\choose k}
$$
holds, then for all $1\leq \ell\leq k$, $|(Diag^{tp}(\calF(\ell), C_W) \cup Err_{\ell}(C_W))\cap Span(\sigma)|\leq \epsilon {n\choose \ell}$.
\end{lemma}
\begin{proof}
For $1\leq \ell<k$, set $\Gamma(\ell)=(Diag^{tp}(\calF(\ell), C_W) \cup Err_{\ell}(C_W))\cap Span(\sigma)$ and let 
$$
\Gamma(k)=(Diag^{tp}(cl_k(\calF),C_W)\cup Err_k(C_W))\cap Span(\sigma).
$$
We want to show that $|\Gamma(k)|\leq \epsilon {n\choose k}$ implies that for all $\ell\in [k]$, $|\Gamma(\ell)|\leq \epsilon {n\choose \ell}$.  If $\ell=k$, this is immediate.  Fix $1\leq \ell<k$.  We claim the following holds.
\begin{align}\label{fact}
\text{For all $S_0\in \Gamma(\ell)$, $|\{S_1\in \Gamma(k): S_0\subseteq S_1\}|\geq {n-\ell \choose r-\ell}$.}
\end{align}
Suppose $S_0\in \Gamma(\ell)$.  Consider the following procedure for constructing a set $S_1\subseteq S_r(C_W,\calH)$.
\begin{itemize}
\item Choose $X\in {C_W\choose k}$ such that $V(S_0)\subseteq X$.  There are ${n-\ell\choose k-\ell}$ choices.
\item For each $A\in {X\choose r}\setminus {V(S_0)\choose r}$, choose some $p_A\in Ch_{\sigma}(A)$ (this is possible since $\sigma$ is complete). 
\item Set $S_1=S_0\cup \{p_A: A\in {X\choose r}\setminus {V(S_0)\choose r}\}$. 
\end{itemize}
Suppose $S_1$ is constructed from $S_0$ in this way.  We claim $S_1\in \Gamma(k)$.  By construction and because $S_0$ is a syntactic $\ell$-diagram, $S_1$ is a syntactic $k$-diagram.  Also by construction, $S_1\subseteq \sigma$, so by definition, $S_1 \in Span(\sigma)$.  If $S_1$ is unsatisfiable, then by definition $S_1\in Err_k(C_W)\cap Span (\sigma)\subseteq \Gamma(k)$, so we are done.  Suppose now $S_1$ is satisfiable and $M$ is an $\calL\cup V(S_1)$-structure such that $M\models S_1^M$.  Let $N=M[V(S_0)^M]$.  Then considered as an $\calL\cup V(S_0)$-structure, $N\models S_0^N$, so part 1 of Observation \ref{ob1} implies $Diag^{tp}(N)=S_0$.  On the other hand, $S_0\in Diag^{tp}(\calF(\ell), C_W)$ implies there is $F\in F(\ell)$ which can be made into an $\calL\cup V(S_0)$-structure such that $F\models S_0^F$.  Part 2 of Observation \ref{ob1} then implies $N\cong_{\calL}F$.  Since $\calF$ is closed under isomorphism, $N\in \calF$.  Since $N\subseteq_{\calL}M$ and $M$ has size $k$, this implies by definition that $M\in cl_k(\calF)$.  Since $S_1=Diag^{tp}(M,V(S_1))$, we have $S_1\in Diag^{tp}(cl_k(\calF),C_W)$ by definition.  Thus we have shown that $S_1$ is in $\Gamma(k)$.  Observe that every distinct choice of $X$ produces a distinct $S_1$, so this construction produces at least ${n-\ell\choose k-\ell}$ distinct $S_1\in \Gamma(k)$ such that $S_0\subseteq S_1$.  We we have proved (\ref{fact}) holds for all $1\leq \ell <k$.  Consider the following procedure for constructing element in $S_0\in \Gamma(\ell)$:
\begin{itemize}
\item Choose $S_1\in \Gamma(k)$.  There are $|\Gamma(k)|$ choices.
\item Choose $S_0\subseteq S_1$ such that $S_0\in \Gamma(\ell)$ (if one exists).  There are at most ${V(S_1)\choose \ell}={k\choose \ell}$ choices.
\end{itemize}
By (\ref{fact}), this construction produces every element $S_0\in \Gamma(\ell)$ at least ${n-\ell\choose k-\ell}$ times.  This shows
$$
|\Gamma(\ell)|\leq |\Gamma(k)|{k\choose \ell}\Big/{n-\ell\choose k-\ell}\leq \epsilon {n\choose k}{k\choose \ell}\Big/{n-\ell \choose k-\ell}=\epsilon {n\choose \ell},
$$
where the second inequality is because $|\Gamma(k)|\leq \epsilon {n\choose k}$ by assumption.
\end{proof}

We now present a computational lemma which will be used in the proof of Theorem \ref{VERSION2}.
\begin{lemma}\label{m}
For all integers $2\leq x<y$, $m(y,x):=\max\Big\{ \frac{{\ell\choose x}-1}{\ell-x}: x< \ell \leq y \Big\}> 1$.
\end{lemma}
\begin{proof}
We show that for all $2\leq x<y$, ${y\choose x}>y-x+1$.  Since by definition, $m(y,x)\geq \frac{{y\choose x}-1}{y-x}$, this will imply $m(y,x)> 1$, as desired.  Fix $x\geq 2$.  We induct on $t$ where $y=x+t$.  Suppose first $y=x+1$.  Then ${y\choose x}=\frac{(x+1)!}{x!}=x+1$.  By assumption on $x$, $x+1\geq 3>2=y-x+1$.  Assume now that $y>x+1$ and suppose by induction the claim holds for $y-1$.  Then ${y\choose x}=\frac{(y-1)!y}{x!(y-x-1)!(y-x)}={y-1\choose x}\frac{y}{y-x}$.  By our induction hypothesis, 
$$
{y-1\choose x}\frac{y}{y-x}\geq ((y-1-x)+1)\Big(\frac{y}{y-x}\Big)=(y-x)\frac{y}{y-x}=y>y-x+1,
$$
where the last inequality is because $x\geq 2$.  Thus ${y\choose x}>y-x+1$, as desired.
\end{proof}

\vspace{2mm}
\noindent {\bf Defining the Hypergraph.}

We now give a procedure for defining a special hypergraph given a finite set and a collection of $\calL$-structures satisfying certain properties.  Assume we are given the following.
\begin{enumerate}[$\bullet$] 
\item A nonempty collection, $\calF$, of finite $\calL$-structures, each of size at most $k$, where $k\geq r$ is an integer.
\item A set $W$ of size $n$, where $n\geq k$ is an integer.
\end{enumerate}

Let $\mathcal{H}$ be the class of all finite $cl_k(\mathcal{F})$-free $\mathcal{L}$-structures.  Define the hypergraph $H=H(\calF, W)$ as follows.  
\begin{align*}
V(H)&=S_r(C_W,\calH)\text{ and }\\
E(H)&=Diag^{tp}(cl_k(\mathcal{F}), C_W)\cup Err_k(C_W).
\end{align*}
We now make a few observations about $H$.  First, note that the edges of $H$ are syntactic $k$-diagrams, so $H$ is a ${k\choose r}$-uniform hypergraph.  By definition $|V(H)|={n\choose r}|S_r(\mathcal{H})|$.  If $X$ and $X'$ are both in ${C_W\choose k}$, then since relabeling constants does not change the satisfiability properties of a collection of $\calL\cup C_W$-sentences, we must have $|Diag^{tp}(Cl_k(\calF), X)\cup Err_k(X)|=|Diag^{tp}(Cl_k(\calF), X')\cup Err_k(X')|$.  Therefore, the following is well defined.

\begin{definition}\label{alpha}
Let $\alpha=\alpha(\calF)$ be such that for all $X\in {C_W\choose k}$, $|Diag^{tp}(Cl_k(\calF), X)\cup Err_k(X)|=\alpha$.
\end{definition}

We claim that $|E(H)|=\alpha {n\choose k}$.  Indeed, any $\sigma\in E(H)$ can be constructed as follows. 
\begin{itemize}
\item Choose $X\in {C_W\choose k}$.  There are ${n\choose k}$ choices.
\item Choose an element $\sigma \in Diag^{tp}(Cl_k(\calF), C_W)\cup Err_k(C_W)$ such that $V(\sigma)=X$, i.e. choose an element $\sigma \in Diag^{tp}(Cl_k(\calF), X)\cup Err_k(X)$.  There are $\alpha$ choices. 
\end{itemize}
Each of these choices lead to distinct subsets $\sigma\in E(H)$, so this shows $|E(H)|=\alpha {n\choose k}$.  Note that because $\calF\neq \emptyset$, $\alpha\geq 1$.  On the other hand, there are at most $|S_r(\calH)|^{k\choose r}$ syntactic $k$-diagrams $\sigma$ with $V(\sigma)=X$, so $\alpha\leq |S_r(\calH)|^{k\choose r}\leq |S_r(\calL)|^{k\choose r}$.  We now make a key observation about this hypergraph.

\begin{proposition}\label{part1prop}
For any $\mathcal{F}$-free $\mathcal{L}$-structure $M$ with domain $W$, $Diag^{tp}(M)$ is an independent subset of $V(H)$.
\end{proposition}
\begin{proof}
Suppose towards a contradiction that $Diag^{tp}(M)$ contains an edge $\sigma\in E(H)$.  Then $\sigma$ is a syntactic $k$-diagram which is either in $Err_k(C_W)$ or $Diag^{tp}(cl_k(\calF), C_W)$. Clearly $\sigma\notin Err_k(C_W)$, since $M\models \sigma^M$ implies $\sigma$ is satisfiable.   Thus $\sigma\in Diag^{tp}(cl_k(\calF), C_W)$.  So there is an $\mathcal{L}\cup V(\sigma)$-structure $B$ such that $B\upharpoonright_{\mathcal{L}}\in cl_k(\mathcal{F})$ and $Diag^{tp}(B, V(\sigma))=\sigma$.  Let $A=\{a: c_a\in V(\sigma)\}\subseteq W$ and let $N=M[A]$.  Suppose $p(c_{\abar})\in \sigma$.  Since $\sigma\subseteq Diag^{tp}(M)$, $M\models p(\abar)$.  Since $N\subseteq_{\calL}M$ and $\cup\abar \subseteq A=dom(N)$, we have $N\models p(\abar)$.  This shows that with its canonical $\calL\cup V(\sigma)$-structure, $N\models \sigma^N$.  Since $\sigma$ is a syntactic $k$-diagram and $N$ has size $k$, part 1 of Observation \ref{ob1} implies $\sigma=Diag^{tp}(N)$. Now $Diag^{tp}(N)=\sigma=Diag^{tp}(B, V(\sigma))$ implies by part 2 of Observation \ref{ob1}, that $N\cong_{\calL}B$.  But now $N$ is an $\mathcal{L}$- substructure of $M$ isomorphic to an element of $cl_k(\mathcal{F})$, contradicting our assumption that $M$ is $\mathcal{F}$-free (recall $|dom(M)|=n\geq k$ implies $M$ is $cl_k(\calF)$-free if and only if $M$ if $\calF$-free).
\end{proof}

\noindent Observe that by definition of $H$, if $S\subseteq V(H)$, then
\begin{align}\label{ob2}
E(H[S])= \Big(Diag^{tp}(cl_k(\mathcal{F}), C_W)\cup Err_k(C_W)\Big)\cap Span(S).
\end{align}

We are now ready to prove Theorem \ref{VERSION2}.  At this point the reader may wish to review Theorem \ref{containers} and its corresponding notation in  Subsection \ref{containerssec}, as Theorem \ref{VERSION2} is the key tool used in this proof.

\vspace{3mm}

\noindent {\bf Proof of Theorem \ref{VERSION2}.}  Clearly it suffices to show the theorem holds for all $0<\epsilon<1/2$.  We claim that further, it suffices to show the theorem holds for any  $k\geq 2r$.  Indeed, suppose $k<2r$ and Theorem \ref{VERSION2} holds for all $k'\geq 2r$.  Suppose $\mathcal{F}$ is a nonempty collection of $\mathcal{L}$-structures, each of size at most $k$ and $\calH:=\Forb(\calF)\neq \emptyset$.  Then $\mathcal{F}$ is also a collection of $\mathcal{L}$-structures, each of size at most $k'=2r$.  Apply Theorem \ref{VERSION2} to $k'=2r$ to obtain constants $c=c(2r,r,\calL, \epsilon)$ and $m=m(2r, r)$.  Since $k<2r$, it is clear the conclusions of Theorem \ref{VERSION2} for $\calH$ and $2r$ imply the conclusions of Theorem \ref{VERSION2} for $\calH$ and $k$.  Thus we may take $c(k,r,\epsilon)=c(2r,r,\epsilon)$ and $m(k,r)=m(2r,r)$.  We now prove the theorem holds for all $0<\epsilon<1/2$ and $k\geq 2r$.

Fix $0<\epsilon<1/2$ and $k\geq 2r$.  Apply Theorem \ref{containers} to $s={k\choose r}$ to obtain the constant $c_0=c_0({k\choose r})$ and set 
$$
m=m(k,r)=\max\Big\{ \frac{{\ell\choose r}-1}{\ell-r}: r< \ell \leq k \Big\}.
$$
By Lemma \ref{m}, since $2\leq r<k$, $m> 1$.  Set $\epsilon'=\epsilon/|S_r(\calL)|^{k\choose r}$ and choose $0<\gamma<1$ sufficiently small so that 
\begin{align}\label{gamma}
2^{{{k\choose r}\choose 2}+1}|S_r(\calL)|r!(k-r)^{k-r} \gamma &\leq \frac{\epsilon'}{12{k\choose 2}!}.
\end{align}
Now set $c=c(k,r,\calL,\epsilon)=(c_0\log(\frac{1}{\epsilon'})|S_r(\calL)|)/\gamma m$.  Observe that $c$ actually depends on $\calL$, not just $r_{\calL}$.  Let $M\geq k$ be such that $n\geq M$ implies $(n-r)^{k-r}\geq n^{k-r}/2$, and $n^{-\frac{1}{m}}\gamma^{-1}<1/2$.  We show Theorem \ref{VERSION2} holds for this $c$ and $m$ for all $n\geq M$.  

Fix $n\geq M$.  Suppose $\calF$ a nonempty collection of finite $\calL$-structures, each of size at most $k$, such that $\calH:=\Forb(\calF)\neq \emptyset$.  Let $W$ be a set of size $n$ and let $H=H(\calF, W)$ be the ${k\choose r}$-uniform hypergraph described above.  Set $\tau=n^{\frac{-1}{m}}\gamma^{-1}$.  By our assumptions we have that $0<\epsilon', \tau<\frac{1}{2}$.  We show that $\delta(H,\tau)\leq \frac{\epsilon'}{12{k\choose r}!}$ so that we may apply Theorem \ref{containers} to $H$.  Let $\alpha=\alpha(\calF)$ be as in Definition \ref{alpha} so that $E(H)=\alpha{n\choose k}$ and let $N=|V(H)|$.  Given $2\leq j\leq {k\choose r}$, set
\begin{align}\label{fdef}
f(j)=\min\{ \ell:  {\ell \choose r}\geq j\}.
\end{align}
Observe that for each $2\leq j\leq {k\choose r}$, $r<f(j)\leq k$.  Indeed, $r<f(j)$ holds since ${f(j)\choose r}\geq j\geq 2$, and $f(j)\leq k$ holds since $k\in \{\ell: {\ell \choose r}\geq j\}$.  Thus by definition of $m$, for each $2\leq j\leq {k\choose r}$, 
\begin{align}\label{minequality}
m\geq  \frac{{f(j)\choose r}-1}{f(j)-r} \geq \frac{j-1}{f(j)-r},
\end{align}
where the inequality is because by (\ref{fdef}), ${f(j)\choose r}\geq j$.  We now show that for all $\sigma \subseteq V(H)$ with $2\leq |\sigma|\leq {k\choose r}$, $d(\sigma)\leq \alpha n^{k-f(|\sigma|)}$.  Fix $\sigma\subseteq V(H)$ so that $2\leq |\sigma|\leq {k\choose r}$.  

Observe that if $|V(\sigma)|>k$, then $\{e\in E(H): \sigma\subseteq e\}=\emptyset$ since every $e\in E(H)$ is a syntactic $k$-diagram, so must satisfy $|V(e)|=k$. So in this case $d(\sigma)=0\leq \alpha n^{k-f(|\sigma|)}$.  Similarly, if there is $A\in {V(\sigma)\choose r}$ such that $|Ch_{\sigma}(A)|\geq 2$, then $\{e\in E(H): \sigma\subseteq e\}=\emptyset$, since every $e\in E(H)$ is a syntactic $k$-diagram, so must satisfy $|Ch_e(A)|=1$.  So in this case, $d(\sigma)=0\leq \alpha n^{k-f(|\sigma|)}$.  Suppose now $|V(\sigma)|\leq k$ and $|Ch_{\sigma}(A)|\leq 1$ for all $A\in {V(\sigma)\choose r}$.  This implies $|\sigma|\leq {|V(\sigma)|\choose r}$, so by (\ref{fdef}), $f(|\sigma|)\leq |V(\sigma)|$.  Every edge in $H$ containing $\sigma$ can be constructed as follows.
\begin{itemize}
\item Choose a set $X\in {C_W\choose k}$ such that $V(\sigma)\subseteq X$ (this makes sense since $|V(\sigma)|\leq k$).  There are ${n-|V(\sigma)|\choose k-|V(\sigma)|}$ ways to do this.
\item Choose an element of $Diag^{tp}(cl_k(\calF), X) \cup Err_k(X)$ containing $\sigma$.  By definition of $\alpha$, there are at most $\alpha$ choices for this.
\end{itemize}
Therefore, $d(\sigma)=|\{ e\in E(H): \sigma\subseteq e\}| \leq \alpha {n-|V(\sigma)|\choose k-|V(\sigma)|}\leq \alpha n^{k-|V(\sigma)|}\leq \alpha n^{k-f(|\sigma|)}$, where the last inequality is because $f(|\sigma|)\leq |V(\sigma)|$.  Thus we have shown that for any $2\leq j\leq {k\choose r}$ and $\sigma\subseteq V(H)$ with $|\sigma|=j$, $d(\sigma)\leq \alpha n^{k-f(j)}$.  Thus given $2\leq j\leq {k\choose r}$ and a vertex $v\in V(H)$, 
$$
d^{(j)}(v)=\max \{ d(\sigma): v\in \sigma, |\sigma|=j\}\leq \alpha n^{k-f(j)}.
$$
Note the average degree of $H$ is
$$
d={k\choose r}|E(H)|/|V(H)| = \frac{{k\choose r}\alpha {n\choose k}}{|S_r(\calH)|{n\choose r}}=\frac{\alpha}{|S_r(\calH)|}{n-r\choose k-r}\geq \frac{\alpha}{|S_r(\calH)|}\Big(\frac{n-r}{k-r}\Big)^{k-r}.
$$
Combining this with our assumption $n$, we obtain the following inequality.
\begin{align}\label{dbound}
d\geq \frac{\alpha}{|S_r(\calH)|}\Big(\frac{n-r}{k-r}\Big)^{k-r}=\frac{\alpha}{|S_r(\calH)|(k-r)^{k-r}} (n-r)^{k-r}\geq \frac{\alpha}{2|S_r(\calH)|(k-r)^{k-r}} n^{k-r}.
\end{align}
Combining all of these computations we have the following.
$$
\delta_j=\frac{\sum_{v\in V(H)} d^{(j)}(v)}{Nd\tau^{j-1}}\leq \frac{Nn^{k-f(j)}}{Nd\tau^{j-1}}= \frac{n^{k-f(j)+(j-1)\frac{1}{m}}\gamma^{j-1}}{d}.
$$
Using our lower bound for $d$ from (\ref{dbound}), this implies
\begin{align*}
\delta_j \leq 2|S_r(\calH)|(k-r)^{k-r} \gamma^{j-1}\alpha^{-1} n^{k-f(j)+\frac{j-1}{m}-k+r}=2|S_r(\calH)|(k-r)^{k-r} \gamma^{j-1}\alpha^{-1} n^{r-f(j)+\frac{j-1}{m}}.
\end{align*}
By (\ref{minequality}), $\frac{j-1}{m}\leq f(j)-r$, so this implies $\delta_j$ is at most
$$
2|S_r(\calH)|(k-r)^{k-r} \gamma^{j-1}\alpha^{-1} n^{r-f(j)+f(j)-r}= 2|S_r(\calH)|(k-r)^{k-r} \gamma^{j-1}\alpha^{-1}\leq 2|S_r(\calH)|(k-r)^{k-r} \gamma,
$$
where the last inequality is because $j\geq 2$ and $\gamma<1$ implies $\gamma^{j-1}<\gamma$ and $\calF\neq \emptyset$ implies $\alpha^{-1}\leq 1$.  Therefore
\begin{align}\label{delta1}
\delta(H,\tau)&=2^{{{k\choose r}\choose 2}-1}\sum_{j=2}^{{k\choose r}}2^{-{j-1\choose 2}} \delta_j\leq 2^{{{k\choose r}\choose 2}-1}2|S_r(\calH)|(k-r)^{k-r} \gamma \sum_{j=2}^{{k\choose r}}2^{-{j-1\choose 2}}.
\end{align}
If $t={{k\choose r}\choose 2}$, then $\sum_{j=2}^{{k\choose r}}2^{-{j-1\choose 2}}\leq\sum_{j=0}^t2^{-t}$.  Using the formula for summing finite geometric series, $\sum_{j=0}^t2^{-t}=\frac{1-2^{-t-1}}{1-2^{-1}}=2(1-2^{-t-1})<2$.  Plugging this into (\ref{delta1}) yields 
\begin{align*}
\delta(H,\tau)&\leq 2^{{{k\choose r}\choose 2}-1}2|S_r(\calH)|(k-r)^{k-r} \gamma 2=2^{{{k\choose r}\choose 2}+1}|S_r(\calH)|(k-r)^{k-r} \gamma\leq 2^{{{k\choose r}\choose 2}+1}|S_r(\calL)|(k-r)^{k-r} \gamma.
\end{align*}
By (\ref{gamma}), the right hand side above is at most $\frac{\epsilon'}{12{k\choose r}!}$, so we have shown that $\delta(H,\tau) \leq \frac{\epsilon'}{12{k\choose r}!}$.  Thus Theorem \ref{containers} implies there is $\Sigma_0\subseteq \mathcal{P}(V(H))$ with the following properties.
\begin{enumerate}[(i)]
\item For every independent set $I\subseteq V(H)$, there is $\sigma \in \Sigma_0$ such that $I \subseteq \sigma$.
\item For every $\sigma \in \Sigma_0$, $e(H[\sigma])\leq \epsilon' e(H)$.
\item $\log |\Sigma_0| \leq c_0\log(1/\epsilon') N\tau \log (1/\tau)$.
\end{enumerate}
Define $\Sigma=\{\sigma \in \Sigma_0: \exists \text{ an $\calF$-free $\mathcal{L}$-structure $M$ with domain $W$ such that $Diag^{tp}(M)\subseteq \sigma \}$}$.  Observe that every $\sigma \in \Sigma$ is complete by part 3 of Observation \ref{ob1}.  We show $\Sigma$ satisfies the conclusions (1)-(3) of Theorem \ref{VERSION2}. Suppose $M$ is an $\mathcal{F}$-free $\mathcal{L}$-structure with domain $W$.  Proposition \ref{part1prop} implies $Diag^{tp}(M)$ is an independent subset of $V(H)$, so by (i) and definition of $\Sigma$, there is $\sigma\in \Sigma$ such that $Diag^{tp}(M)\subseteq \sigma$.  Thus part (1) of Theorem \ref{VERSION2} holds.  We now show part (2) holds.  Fix $\sigma\in \Sigma$.  By (\ref{ob2}), $(Diag^{tp}(cl_k(\calF), C_W)\cup Err_k(C_W))\cap Span(\sigma)=E(H[\sigma])$.  So (ii) implies
$$
e(H[\sigma])=|(Diag^{tp}(cl_k(\calF), C_W)\cup Err_k(C_W))\cap Span(\sigma)|\leq \epsilon' e(H). 
$$
By definition of $\epsilon'$ and because $\alpha \leq |S_r(\calL)|^{k\choose r}$, 
$$
\epsilon'e(H)=\epsilon'\alpha{n\choose k}=\frac{\epsilon}{|S_r(\calL)|^{k\choose r}}\alpha {n\choose k}\leq \epsilon {n\choose k}.
$$
Thus $|(Diag^{tp}(cl_k(\calF), C_W)\cup Err_k(C_W))\cap Span(\sigma)|\leq \epsilon{n\choose k}$.  By Lemma \ref{lemma**}, for all $1\leq \ell \leq k$, 
$$
|(Diag^{tp}(\calF(\ell),C_W)\cup Err_{\ell}(C_W))\cap Span (\sigma)|\leq \epsilon {n\choose \ell}.
$$
Since $k\geq 2r$, this immediately implies part (2) of Theorem \ref{VERSION2} holds.  By (iii), definition of $c$, and because $\Sigma\subseteq \Sigma_0$ we have that
\begin{align*}
|\Sigma|\leq |\Sigma_0|\leq c_0\log(1/\epsilon') N\tau \log (1/\tau)&=c_0\log(1/\epsilon') |S_r(\calH)|{n\choose r}\tau \log (1/\tau) \\
&\leq c_0\log(1/\epsilon') |S_r(\calL)|{n\choose r}\tau \log (1/\tau)=c\gamma m{n\choose r}\tau \log\Big(\frac{1}{\tau}\Big).
\end{align*}
This shows $|\Sigma|\leq c\gamma mn^r\tau \log(\frac{1}{\tau})$.  By definition of $\tau$,
$$
c\gamma mn^r\tau \log\Big(\frac{1}{\tau}\Big)=c mn^{r-\frac{1}{m}}\Big(\frac{\log n}{m}+\log\gamma\Big)= c n^{r-\frac{1}{m}}\Big(\log n+m\log\gamma\Big)\leq cn^{r-\frac{1}{m}}\log n,
$$
where the last inequality is because $\gamma \leq 1\leq m$ implies $m\log \gamma \leq 0$.  This shows part (3) of Theorem \ref{VERSION2} holds, so we are done.
\qed

%**********************************************************************************
\section{Metric Spaces}\label{ms}
%**********************************************************************************

Given integers $r\geq 3$ and $n\geq 1$, let $M_r(n)$ be the set of metric spaces with distances all in $[r]$ and underlying set $[n]$.  In this section we will reprove certain structure and enumeration theorems about $M_r(n)$ from \cite{MT} using the machinery of this paper, along with combinatorial ingredients from \cite{MT}.  We include this example because it demonstrates interesting behavior with regards to the existence of a stability theorem.  In particular, we will prove that when $r$ is even, the hereditary property associated to $\bigcup_{n\in \mathbb{N}}M_r(n)$ has a stability theorem in the sense of Definition \ref{stabdef}, but when $r$ is odd, this is not the case.  

We would like to point out that in order to be consistent with \cite{MT}, $r$ will always denote the largest distance appearing in our metric spaces.  This $r$ has nothing to do with our use of the letter $r$ as shorthand for $r_{\calL}$ throughout the rest of the paper.  In this section all languages will binary (i.e. $r_{\calL}=2$), so we do not think any confusion will arise.

\subsection{Results from \cite{MT}.}
In this subsection we state results from \cite{MT}.  {\bf For the rest of this section, $r\geq 3$ is a fixed integer}. We require some notation and definitions in order to state the relevant theorems from \cite{MT}.   An \emph{$[r]$-graph} (respectively an $2^{[r]}$-graph) is a pair $(V,c)$ such that $V$ is a set of vertices and $c:{V\choose 2}\rightarrow [r]$ (respectively $c:{V\choose 2}\rightarrow 2^{[r]}$) is a function.

\begin{definition} Given an $[r]$-graph $G=(V,d)$ and a $2^{[r]}$-graph $R=(V',c)$, we say $G$ is a \emph{sub-$[r]$-graph} of $R$, denoted $G\subseteq_{[r]} R$, if $V'\subseteq V$ and for each $xy\in {V'\choose r}$, $d(xy)\in c(xy)$.  We say $G$ is a \emph{full sub-$[r]$-graph} of $R$, denote $G\subseteqq_{[r]}R$, if moreover, $V'=V$.  
\end{definition}

A $2^{[r]}$-graph $(V,c)$ is \emph{complete} if for all $xy\in {V\choose 2}$, $|c(xy)|\geq 1$.  Two $[r]$-graphs (respectively $2^{[r]}$-graphs) $(V,d)$ and $(V',d')$ are \emph{isomorphic} if there is a bijection $f:V\rightarrow V'$ such that for all $xy\in {V\choose 2}$, $d(xy)=d'(f(x)f(y))$.  A \emph{violating triple} is a tuple $(i,j,k)\in \mathbb{N}^3$ such that $|i-j|\leq k\leq i+j$ is false.  An $[r]$-graph $G=(V,d)$ is a \emph{metric $[r]$-graph} if for all $\{x,y,z\}\in {V\choose 3}$, $(d(xy), d(yz),d(xz))$ is not a violating triple.  A $2^{[r]}$-graph $G=(V,c)$ is a \emph{metric $2^{[r]}$-graph} for all $\{x,y,z\}\in {V\choose 3}$, no $(i,j,k)\in c(xy)\times c(yz)\times c(xz)$ is a violating triple.  Observe that if $R$ is a complete $2^{[r]}$-graph, then $R$ is a metric $2^{[r]}$-graph if and only if all its full sub-$[r]$-graphs are metric $[r]$-graphs.  Given integers $i<j$, set $[i,j]=\{i,i+1, \ldots, j\}$.  Set 
$$
m(r)=\Bigg\lceil \frac{r+1}{2}\Bigg\rceil.
$$
If $r$ is odd, let $L_r=[ \frac{r-1}{2}, r-1]$ and $U_r=[\frac{r+1}{2}, r]$.  Observe that if $r$ is odd, then $m(r)=|U_r|=|L_r|$ and if $r$ is even, then $m(r)=|[\frac{r}{2},r]|$.  We now define a special subfamily of $M_r(n)$.

\begin{definition}\label{Crdefeven}
Suppose $n\geq 3$ is an integer.
\begin{enumerate}
\item If $r\geq 4$ is even, define $C_r(n)$ be the set of $[r]$-graphs $G=([n],d)$ such that $d(e)\in [\frac{r}{2}, r ]$ for all $e\in {[n]\choose 2}$.
\item If $r\geq 3$ is odd, define $C_r(n)$ to be the set of $[r]$-graphs $G=([n],d)$ such that there is a partition $V_1\cup \ldots\cup  V_t$ of $[n]$ so that for every $xy\in {[n]\choose 2}$, 
\[
d(xy) \in \begin{cases}
L_r& \text{if } xy\in {V_i\choose 2}\textnormal{ for some }i\in [t] \\
U_r & \text{if } x\in V_i, y\in V_j \textnormal{ for some }i\neq j\in [t].
\end{cases}
\]
\end{enumerate}
\end{definition}

It is straightforward to check that in both even and odd cases, $C_r(n)\subseteq M_r(n)$. 

\begin{definition}\label{newdeltaclose} Let $G=(V,c)$ and $G=(V,c')$ be finite $C$-graphs where $C=[r]$ or $C=2^{[r]}$. Set
$$
\Delta(G,G')=\{xy\in {V\choose 2}: c(xy)\neq c'(xy)\}.
$$
Given $\delta>0$, we say $G$ and $G'$ are \emph{$\delta$-close}\footnote{In \cite{MT}, $\delta$-closeness is instead defines as $|\Delta(G,G')|\leq \delta n^2$.  For large $n$, this is basically the same as Definition \ref{newdeltaclose} up to a factor of $2$, and therefore doesn't change the statements of the results from \cite{MT}.  We have chosen to use Definition \ref{newdeltaclose} for convenience.} if $|\Delta(G,G')|\leq \delta{|V|\choose 2}$.
\end{definition}

Set $C_r^{\delta}(n)=\{G\in M_r(n): \textnormal{there is some }G' \in C_r(n)\textnormal{ such that }|\Delta(G,G')|\leq \delta{|V|\choose 2} \}$.  We can now state the approximate structure theorem from \cite{MT}. Informally, it states that for all $\delta>0$, most members of $M_r(n)$ are in $C_r^{\delta}(n)$ when $n$ is sufficiently large depending on $\delta$.

\begin{theorem}[{\bf Mubayi-Terry \cite{MT}}]\label{DELTACLOSETHM}
Let $r\geq 3$ be an integer.   Then for all $\delta>0$, there exists an $M$ and $\beta >0$ such that $n>M$ implies
$$
\frac{|M_r(n)\setminus C_r^{\delta}(n)|}{|M_r(n)|}\leq \frac{|M_r(n)\setminus C_r^{\delta}(n)|}{m(r)^{n\choose 2}}\leq 2^{-\beta {n\choose 2}}.
$$
\end{theorem}
\noindent We now state the approximate enumeration theorem from \cite{MT}.
\begin{corollary}[{\bf Mubayi-Terry \cite{MT}}]\label{oddcor}
Let $r\geq3$ be an integer.  Then $|M_r(n)|=m(r)^{{n\choose 2}+o(n^2)}$.
\end{corollary}

In fact, \cite{MT} contains much more precise structural and enumerative results for $M_r(n)$ in the case when \emph{$r$ is even}.  Finding similar refinements of Theorem \ref{DELTACLOSETHM} and Corollary \ref{oddcor} in the case when $r$ is odd is still open.  This suggests there is something ``nicer'' about the even case than the odd case.  We will show in this section that one candidate for what makes the even case ``nice'' is that when $r$ is even, the hereditary property corresponding to $M_r(n)$ has a stability theorem in the sense of Definition \ref{stabdef}, while in the odd case it does not.

\subsection{Translation to hereditary $\calL$-property.}
In this subsection we translate some of the combinatorial notions used in \cite{MT} to the setup used in this paper.  Recall $r\geq 3$ is a fixed integer.  Let $\mathcal{L}_r=\{R_1(x,y),\ldots, R_r(x,y)\}$ consist of $r$ binary relation symbols.  We consider elements $(V,d)$ of $M_r(n)$ as $\mathcal{L}_r$-structures by interpreting $R_i(x,y)$ if and only if $d(xy)=i$, for each $(x,y)\in V^2$.  Let $\mathcal{M}_r$ denote the class of $\calL_r$-structures obtained by closing $M_r=\bigcup_{n\in \mathbb{N}}M_r(n)$ under isomorphism.  Clearly $\mathcal{M}_r$ is a hereditary $\mathcal{L}_r$-property, and $(\mathcal{M}_r)_n=M_r(n)$.  For the rest of the section, $\calP=\calM_r$, and $\calL=\calL_r$.  Observe that since $r_{\calL}=2$, $\calL_{\calP}=\{R_p(x,y): p(x,y)\in S_2(\calP)\}$.  For each $i\in [r]$, set  
$$
q_i(x,y):=\{x\neq y\}\cup \{R_i(x,y), R_i(y,x)\}\cup \{\neg R_j(x,y), \neg R_j(y,x): j\in [r]\setminus \{i\}\},
$$
and let $p_i(x,y)$ be the unique quantifier-free $2$-type containing $q_i(x,y)$. In other words, $p_i(x,y)$ is the type saying the distance between $x$ and $y$ is $i$.  Clearly  $S_2(\calP)=\{p_i(x,y): i\in [r]\}$, so $\calL_{\calP}=\{R_{p_i}(x,y): i\in [r]\}$.   The following observation will be useful and is obvious from the definition of a choice function.

\begin{observation}\label{ob7}
If $G$ is an $\calL_{\cal}$-template with domain $V$ and $\chi\in Ch(G)$, then for all $uv\in {V\choose 2}$, $Ch_G(uv)=\{p_j(c_u,c_v): G\models R_{p_j}(u,v)\vee R_{p_j}(v,u)\}$.
\end{observation}

 \begin{definition}
 Suppose $G$ is an $\calL_{\calP}$-structure with underlying set $V$.  The \emph{$2^{[r]}$-graph associated to $G$} is $\Psi(G):=(V, c)$, where for each $xy\in {V\choose 2}$, $c(xy)=\{i: G\models R_{p_i}(x,y)\vee R_{p_i}(y,x)\}$ (in other words, $i\in c(x,y)$ if and only if $p_i(c_x,c_y)\in Ch_G(xy)$).  
 \end{definition}
 
 Given a $2^{[r]}$-graph $(V,c)$, define $\Psi^{-1}(V,c)$ to be the $\calL_{\calP}$-structure $G$ with domain $V$ such that for each $(x,y)\in V^2$ and $i\in [r]$, $G\models R_{p_i}(x,y)$ if and only if $x\neq y$ and $i\in c(xy)$.  We leave the following observations to the reader.
 
 \begin{observation}\label{obms}
Suppose $(V,c)$ is a complete $2^{[r]}$-graph and $G$ is an $\calL_{\calP}$-template. Then 
\begin{enumerate}[(a)]
\item $\Psi^{-1}(V,c)$ is an $\calL_{\calP}$-template and $\Psi(\Psi^{-1}(V,c))=(V,c)$.  
\item $\Psi(G)$ is a complete $2^{[r]}$-graph and $\Psi^{-1}(\Psi(G))=G$.
\end{enumerate}
\end{observation}

Suppose $G\in M_r(n)$ is the $[r]$-graph $(V,d)$ considered as an $\calL$-structure.  We will often abuse notation by conflating $G$ and $(V,d)$.  For instance if $R$ is a $2^{[r]}$-graph, we will write $G\subseteqq_{[r]}R$ to mean $(V,d)\subseteqq_{[r]}R$.  Similarly if $R$ is an $\calL_{\calP}$-template, we will write $(V,d)\unlhd_pR$ to mean $G\unlhd_pR$.
 
\begin{proposition}\label{translation}
Suppose $G$ is an $\calL_{\calP}$-template with domain $V$ and $\Psi(G)=(V,c)$.  Then $G'\unlhd_pG$ if and only if $G'\subseteqq_{[r]}\Psi(G)$.
\end{proposition}

\begin{proof}
Suppose first $G'\unlhd_pG$.  Let $\chi\in Ch(G)$ such that $G'\unlhd_{\chi}G$.  Let $(V,d)$ be the $[r]$-graph such that for each $uv\in {V\choose 2}$, $d(uv)=i$  is the unique element of $[r]$ such that $\chi(uv)= p_i(c_u,c_v)$.  Observe that $G'=(V,d)$, considered as an $\calL$-structure. We show $(V,d)\subseteqq_{[r]}\Psi(G)$.  Fix $uv\in {V\choose 2}$.  We want to show $d(uv)\in c(uv)$.  Since $\chi\in Ch(G)$, by definition of $(V,d)$, and by Observation \ref{ob7}, if $i=d(uv)$, then 
$$
p_i(c_u,c_v)=\chi(uv)\in Ch_G(uv)=\{p_j(c_u,c_v): G\models R_{p_j}(u,v)\vee R_{p_j}(v,u)\}.
$$
This implies $d(uv)=i\in \{j: G\models R_{p_j}(u,v)\vee R_{p_j}(v,u)\}=c(uv)$, where the last equality is by definition of $\Psi(G)=(V,c)$. Thus $d(uv)\in c(uv)$ for all $uv\in {V\choose 2}$, so $(V,d)\subseteqq_{[r]}(V,c)$, as desired.

Suppose on the other hand that $G'=(V,d)\subseteqq_{[r]}\Psi(G)$.  We want to show that considered as an $\calL$-structure, $G'\unlhd_pG$.  Define a function $\chi:{V\choose 2}\rightarrow S_2(C_V,\calP)$ as follows.  For each $uv\in {V\choose 2}$, if $d(uv)=i$, set $\chi(uv)=p_i(c_u,c_v)$.  Since $(V,d)\subseteqq_{[r]}(V,c)$ and by definition of $\Psi(G)=(V,c)$, we have 
$$
d(uv)\in c(u,v)=\{j: G\models R_{p_j}(x,y)\vee R_{p_j}(y,x)\}.
$$
Thus $\chi(uv)=p_i(c_u,c_v)\in \{p_j(c_u,c_v): G\models R_{p_j}(u,v)\vee R_{p_j}(v,u)\}=Ch_G(uv)$, where the last equality is by Observation \ref{ob7}. This verifies that $\chi\in Ch(G)$.  By definition of $\chi$, for each $uv\in {V\choose 2}$, $\chi(uv)=Diag^{G'}(uv)$, so $G'\unlhd_{\chi}G$.  
\end{proof}

Given $n$, let $\tilde{M}_r(n)$ be the set of complete metric $2^{[r]}$-graphs on $[n]$.  

\begin{corollary}\label{cor1}
If $([n],c)\in \tilde{M}_r(n)$, then $\Psi^{-1}(V,c)\in \calR([n],\calP)$.  If $G\in \calR([n],\calP)$, then $\Psi(G)\in \tilde{M}_r(n)$.
\end{corollary}
\begin{proof}
Suppose $([n],c)\in \tilde{M}_r(n)$.  By Observation \ref{obms}(a), $\Psi^{-1}([n],c)$ is an $\calL_{\calP}$-template with domain $[n]$.  Since $([n],c)$ is a complete metric $2^{[r]}$-graph, all full sub-$[r]$-graphs of $([n],c)$ are metric spaces.  By Proposition \ref{translation}, this implies all full subpatterns of $\Psi^{-1}([n],c)$ are metric spaces, which implies $\Psi^{-1}([n],c)$ is $\calP$-random by definition.  Suppose now $G\in \calR([n],\calP)$.  Let $\Psi(G)=([n],c)$.  By Observation \ref{obms}(b), $([n],c)$ is a complete $2^{[r]}$-graph.  To show $([n],c)$ is a metric $2^{[r]}$-graph, let $([n],d)\subseteqq_{[r]}([n],c)$.  By Proposition \ref{translation}, $([n],d)\unlhd_pG$, so since $G$ is $\calP$-random, $([n],d)$ is a metric space.  Thus all full sub-$[r]$-graphs of $([n],c)$ are metric $[r]$-graphs.  This implies $([n],c)$ is a metric $2^{[r]}$-graph.  Thus we have shown $([n],c)\in \tildeM_r(n)$.
\end{proof}

We now recall a definition from \cite{MT}.  If $G=(V,c)$ is a $2^{[r]}$-graph, then $W(R)=\prod_{xy\in {V\choose 2}}|c(xy)|$.

\begin{proposition}\label{0}
Suppose $G$ is a finite $\calL_{\calP}$-template. Then $sub(G)=W(\Psi(G))$.
\end{proposition}
\begin{proof}
Let $\Psi(G)=(V,c)$.  By Proposition \ref{translation}, the full subpatterns of $G$ are exactly the full sub-$[r]$-graphs of $(V,c)$, considered as $\calL$-structures.  Clearly the number of full sub-$[r]$-graphs of $(V,c)$ is $\prod_{uv\in {V\choose 2}}|c(uv)|$.  This shows $sub(G)=\prod_{uv\in {V\choose 2}}|c(uv)|=W(\Psi(G))$. 
\end{proof}

If $G\in \tildeM_r(n)$, say that $G$ is \emph{product-extremal} if 
$$
W(G)=\max\{W(G'): G'\in \tildeM_r(n)\}.
$$

\begin{proposition}\label{1}
Suppose $G$ is an $\calL_{\calP}$-template with domain $[n]$.  Then $G\in \calR_{ex}([n],\calP)$ if and only if $\Psi(G)$ is a product-extremal element of $\tildeM_r(n)$.
\end{proposition}
\begin{proof}
Suppose first $G\in \calR_{ex}([n],\calP)$.  By Corollary \ref{cor1}, and definition of $\tilde{M}_r(n)$, $\Psi(G)\in \tildeM_r(n)$.  Suppose towards a contradiction that $\Psi(G)$ is not product-extremal.  Then there is $H\in \tildeM_r(n)$ such that $W(H)>W(\Psi(G))$. Corollary \ref{cor1} implies $\Psi^{-1}(H)\in \calR([n],\calP)$ and Proposition \ref{0} and Observation \ref{obms}(a) implies $sub(\Psi^{-1}(H))=W(\Psi(\Psi^{-1}(H)))=W(H)>W(\Psi(G))=sub(G)$, contradicting that $G\in \calR_{ex}([n],\calP)$.  Conversely, suppose $\Psi(G)$ is a product-extremal element of $\tildeM_r(n)$.  By Corollary \ref{cor1} and Observation \ref{obms}(b), $\Psi^{-1}(\Psi(G))=G\in \calR([n],\calP)$.  Suppose towards a contradiction $G\notin \calR_{ex}([n],\calP)$.  Then there is $G'\in \calR([n],\calP)$ such that $sub(G')>sub(G)$.  Corollary \ref{cor1} implies $\Psi(G')\in \tilde{M}_r(n)$ and Proposition \ref{0} implies $W(\Psi(G'))=sub(G')>sub(G)=W(\Psi(G))$, contradicting that $\Psi(G)$ is product-extremal.
\end{proof}

\subsection{Characterizing extremal structures and computing $\pi(\calP)$.}

In this subsection, we characterize product-extremal elements of $\tilde{M}_r(n)$.  These results, Lemmas \ref{extlemeven} and \ref{extlemodd}, are new results.  We will then use the correspondence between product-extremal elements of $\tilde{M}_r(n)$ and elements of $\calR_{ex}([n],\calP)$ from the preceding section to compute $\pi(\calP)$.  We begin by defining a special subfamily of $\tilde{M}_r(n)$ corresponding to the special subfamily $C_r(n)$ of $M_r(n)$.

\begin{definition}\label{tildeCrdef} Let $\tilde{C}_r(n)$ be the the set of complete $2^{[r]}$-graphs $([n],c)$ such that
\begin{enumerate}[(i)]
\item if $r$ is even, then for all $xy\in {[n]\choose 2}$, $c(xy)=[\frac{r}{2}, r]$.
\item if $r$ is odd, then there is a partition $V_1\cup \ldots\cup  V_t$ of $[n]$ such that for all $xy\in {[n]\choose 2}$, 
\[
c(xy) = \begin{cases}
L_r & \text{if } xy\in {V_i\choose 2}\textnormal{ for some }i\in [t] \\
U_r & \text{if } xy\in E(V_i,V_j)\textnormal{ for some }i\neq j\in [t].
\end{cases}
\] 
\end{enumerate}
\end{definition}
Note that for all $n$, $\tilde{C}_r(n)\subseteq \tilde{M}_r(n)$ and when $r$ is even, $\tilde{C}_r(n)$ consists of a single $2^{[r]}$-graph.  We will need a few results about multigraphs.  A \emph{multigraph} is a pair $(V,w)$ where $V$ is a set of vertices and $w:{V\choose 2}\rightarrow \mathbb{N}$ is a function.  Given integers $s\geq 2$ and $q\geq 0$, an \emph{$(s,q)$-graph} is a multigraph $(V,w)$ such that for all $X\in {V\choose s}$, $\sum_{xy\in {X\choose 2}}w(xy)\leq q$. Given a multigraph $G=(V,w)$, set $P(G)=\prod_{xy\in {V\choose 2}}w(xy)$.  

\begin{definition}
Suppose $a\geq 1$ is an integer.  
\begin{enumerate}[(a)]
\item $U_{1,a}(n)$ is the set of multigraphs $([n],w)$ such that $w(xy)= a$ for all $xy\in {[n]\choose 2}$.
\item $U_{2,a}(n)$ is the set of multigraphs $([n],w)$ such that there is a set  $\{e_1,\ldots, e_{\lfloor n/2 \rfloor}\}$ of pairwise disjoint elements in ${[n]\choose 2}$ such that $w(e_i)=a+1$ for each $e_i$ and for all $xy\in {[n]\choose 2}\setminus \{e_1,\ldots, e_{\lfloor n/2 \rfloor}\}$, $w(xy)=a$.  
\end{enumerate}
\end{definition}

Observe that if $G\in U_{1,a}(n)$, then $P(G)=a^{n\choose 2}$ and if $G\in U_{2,a}(n)$, then $P(G)=a^{n\choose 2}(\frac{a+1}{a})^{\lfloor n/2\rfloor }$.

\begin{theorem}[{\bf Mubayi-Terry, Theorem 5.5.6 of \cite{mythesis}}]\label{mgthm}
Suppose $a\geq 1$, $n\geq 3$ and $G=([n],w)$ is a multigraph.
\begin{enumerate}[(a)]
\item If $G$ is a $(3,3a)$-graph, then $P(G)\leq a^{n\choose 2}$ with equality holding if and only if $G\in U_{1,a}(n)$.  
\item If $G$ is a $(3,3a+1)$-graph, then $P(G)\leq a^{n\choose 2}(\frac{a+1}{a})^{\lfloor \frac{n}{2}\rfloor}$ with equality holding if and only if $G\in U_{2,a}(n)$.
\end{enumerate}
\end{theorem}

In the following lemma we use a result of Balogh and Wagner from \cite{BW}.  In particular, we use Lemma 3.2 of \cite{BW}, which is a corollary of combinatorial results in \cite{MT}.

\begin{lemma}\label{extlemeven}
Suppose $r\geq 4$ is even, $n\geq 3$, and $G\in \tilde{M}_r(n)$ is product-extremal.  Then $W(G)=m(r)^{n\choose 2}$ and $G$ is the unique $2^{[r]}$-graph in $\tildeC_r(n)$.
\end{lemma}
\begin{proof}
Suppose $G=(V,c)\in \tildeM_r(n)$ is product-extremal.  Let $G_0$ be the unique element in $\tilde{C}_r(n)$.  Then by definition, $W(G_0)=|[\frac{r}{2},r]|^{n\choose 2}=m(r)^{n\choose 2}$.  Since $G$ is product-extremal, this implies $W(G)\geq W(G_0)=m(r)^{n\choose 2}$.  Now let $H=([n],w)$ be the multigraph defined by $w(xy)=|c(xy)|$ for all $xy\in {[n]\choose 2}$.  Observe $P(H)=W(G)$.  We claim that $H$ is a $(3, 3m(r))$-graph.  Suppose towards a contradiction there were three is $\{x,y,z\}\in {[n]\choose 3}$ such that $w(xy)+w(yz)+w(xz)>3m(r)$.  Then $c(xy), c(yz), c(xz)$ are nonempty subsets of $[r]$ satisfying
$$
w(xy)+w(yz)+w(xz)=|c(xy)|+|c(yz)|+|c(xz)|>3m(r).
$$
By Lemma 3.2 of \cite{BW}, this implies there is $(i,j,k)\in c(xy)\times c(yz)\times c(xz)$ which is a violating triple, contradicting that $G\in \tildeM_r(n)$.  Thus $H$ is a $(3,3m(r))$-graph.  Theorem \ref{mgthm} implies that $W(G)=P(H)\leq m(r)^{n\choose r}$ and equality holds if and only if $w(xy)=m(r)$ for all $xy\in {[n]\choose 2}$.  Since we already showed $W(G)\geq m(r)^{n\choose 2}$, equality must hold.  Thus $P(H)=W(G)=m(r)^{n\choose 2}$ and $w(xy)=m(r)=|c(xy)|$ for all $xy\in {[n]\choose 2}$.  Part (1) of Corollary 4.15 in \cite{MT} implies that the only metric $2^{[r]}$-graph $([n],c')$ satisfying $|c'(x,y)|=m(r)$ for all $xy\in {[n]\choose 2}$ is $G_0$.  Thus $G=G_0$.
\end{proof}

If $r\geq 3$ is odd and $n\geq 3$, define $\tildeE_r(n)$ to be the set of $2^{[r]}$-graphs $([n],c)$ such that there is a set $\{e_1,\ldots, e_{\lfloor n/2\rfloor}\}$ of pairwise disjoint elements in ${[n]\choose 2}$ such that $c(e_i)=U_r\cup L_r$ and if $xy\in {[n]\choose 2}\setminus \{e_1,\ldots, e_{\lfloor n/2\rfloor}\}$ then $c(xy)=U_r$.  Observe that $\tildeE_r(n)\subseteq \tildeC_r(n)$ and 
\begin{align}\label{Efact}
\text{ for any $G\in \tilde{E}_r(n)$, }W(G)=m(r)^{n\choose r}\Big(\frac{m(r)+1}{m(r)}\Big)^{\lfloor n/2 \rfloor}.
\end{align}

\begin{lemma}\label{extlemodd}
Suppose $r\geq 3$ is odd, $n\geq 3$, and $G$ is a product-extremal element of $\tilde{M}_r(n)$.  Then $P(G)=m(r)^{n\choose r}(\frac{m(r)+1}{m(r)})^{\lfloor n/2 \rfloor}$ and $G\in \tilde{E}_r(n)$.
\end{lemma}
\begin{proof}
Suppose $G=([n],c)\in \tildeM_r(n)$ is product-extremal.  Let $H$ be the multigraph $([n],w)$ where $w(xy)=|c(xy)|$ for all $xy\in {[n]\choose 2}$.  Observe $P(H)=W(G)$.  We claim $H$ is a $(3,3m(r)+1)$-graph.  Suppose by contradiction there is $\{x,y,z\} \in {[n]\choose 3}$ such that $w(xy)+w(yz)+w(xz)>3m(r)+1$.  Then $c(xy), c(yz), c(xz)$ are nonempty subsets of $[r]$ satisfying
$$
w(xy)+w(yz)+w(xz)=|c(xy)|+|c(yz)|+|c(xz)|>3m(r)+1.
$$
By Lemma 3.2 of \cite{BW}, this implies there is $(i,j,k)\in c(xy)\times c(yz)\times c(xz)$ which is a violating triple, contradicting that $G\in \tildeM_r(n)$. Thus $H$ is a $(3,3m(r)+1)$-graph, so by Theorem \ref{mgthm}, $W(G)=P(H)\leq m(r)^{n\choose r}(\frac{m(r)+1}{m(r)})^{\lfloor n/2 \rfloor}$ and equality holds if and only if $H\in U_{1,m(r)}(n)$.  Since $G$ is product-extremal, we must have $W(G)\geq W(G')$ for all $G'\in \tilde{E}_r(n)$.  Combining these facts with (\ref{Efact}), we must have $W(G)=P(H)=m(r)^{n\choose r}(\frac{m(r)+1}{m(r)})^{\lfloor n/2 \rfloor}$ and $H\in U_{1,m(r)}(n)$.  So there is a set  $\{e_1,\ldots, e_{\lfloor n/2 \rfloor}\}$ of pairwise disjoint elements in ${[n]\choose 2}$ such that $w(e_i)=|c(e_i)|=m(r)+1$ for each $e_i$ and for all $xy\in {[n]\choose 2}\setminus \{e_1,\ldots, e_{\lfloor n/2 \rfloor}\}$, $w(xy)=|c(xy)|=m(r)$.  These facts along with part (2) of Corollary 4.15 in \cite{MT}, imply that for each $e_i$, $c(e_i)=U_r\cup L_r$ and for all $xy\in {[n]\choose 2}\setminus \{e_1,\ldots, e_{\lfloor n/2 \rfloor}\}$, $c(xy)=U_r$.  In other words, $G$ is in $\tildeE_r(n)$.  
\end{proof}

\begin{corollary}\label{mspi}
Let $n\geq 2$.  If $r\geq 2$ is even, then $\ex(n,\calP)=m(r)^{n\choose 2}$.  If $r\geq 3$ is odd, then $\ex(n,\calP)=m(r)^{n\choose r}(\frac{m(r)+1}{m(r)})^{\lfloor n/2 \rfloor}$.  Consequently, $\pi(\calP)=m(r)$.
\end{corollary}
\begin{proof}
That $\ex(n,\calP)=\max\{W(G): G\in \tildeM_r(n)\}$ follows from Propositions \ref{0} and \ref{1}.  Thus if $r$ is even, Lemma \ref{extlemeven} implies $\ex(n,\calP)=m(r)^{n\choose 2}$.  Similarly, if $r$ is odd, Lemma \ref{extlemodd} implies $\ex(n,\calP)=m(r)^{n\choose r}(\frac{m(r)+1}{m(r)})^{\lfloor n/2 \rfloor}$. That $\pi(\calP)=m(r)$ holds in both cases now follows from the definition of $\pi(\calP)=\lim_{n\rightarrow \infty}\ex(n,\calP)^{1/{n\choose 2}}$.
\end{proof}

\subsection{Proofs of Results}
In this subsection, we prove Corollary \ref{oddcor} and Theorem \ref{DELTACLOSETHM}.  We also prove that if $r$ is even then $\calP$ has a stability theorem, and if $r$ is odd, then $\calP$ does \emph{not} have a stability theorem.  These results, Lemma \ref{msstabeven} and Corollary \ref{msstabodd}, are new.  We begin by proving Corollary \ref{oddcor}.
\vspace{3mm}

\noindent {\bf Proof of Corollary \ref{oddcor}.}
Theorem \ref{enumeration} implies $|\calP_n|=\pi(\calP)^{(1+o(1)){n\choose 2}}$.  By definition of $\calP$, $M_r(n)=\calP_n$, and Corollary \ref{mspi} implies $\pi(\calP)=m(r)$.  Thus $|\calP_n|=|M_r(n)|=m(r)^{(1+o(1)){n\choose 2}}$.
\qed

\vspace{3mm}

\noindent We now state a stability style result, Theorem 4.13 from \cite{MT}.

\begin{theorem}[{\bf Mubayi-Terry \cite{MT}}]\label{stabthm}
Fix an integer $r\geq 3$.  For all $\delta>0$ there is $0<\epsilon<1$ and $M$ such that for all $n>M$ the following holds. If $R\in \tilde{M}_r(n)$ and $W(R)>m(r)^{(1-\epsilon){n\choose 2}}$, then $R$ is $\delta$-close to an element in $\tilde{C}_r(n)$ (in the sense of Definition \ref{newdeltaclose}).
\end{theorem}

The following is straightforward from the definitions.

\begin{observation}\label{ob8}
If $G=(V,[n])$ and $G=(V,[n])$ are in $M_r(n)$, then $\dist(G,G')\leq \delta$ (in the sense of Definition \ref{deltaclosedef1}, considered as $\calL$-structures) if and only if $G$ and $G'$ are $\delta$-close in the sense of Definition \ref{newdeltaclose} (i.e. if and only if $|\Delta(G,G')|\leq \delta {n\choose 2}$).
\end{observation}

We can now prove Theorem \ref{DELTACLOSETHM}.
\vspace{3mm}

\noindent {\bf Proof of Theorem \ref{DELTACLOSETHM}.}
Fix $\delta>0$.  Choose $\epsilon>0$ and $M_1$ from Theorem \ref{stabthm} such that if $n>M_1$ and $H\in \tildeM_r(n)$ satisfies $W(H)>m(r)^{(1-\epsilon){n\choose 2}}$ then $H$ is $\delta/2$-close in the sense of Definition \ref{newdeltaclose} to an element of $\tildeC_r(n)$.  Now let $\beta>0$ and $M_2$ be as in Theorem \ref{b4stab} applied to $\delta/2$ and $\epsilon$.  Let $N=\max\{M_1,M_2\}$.  We show for all $n>N$, 
$$
\frac{|M_r(n)\setminus C^{\delta}_r(n)|}{|M_r(n)}|\leq 2^{-\beta n^2}.
$$
By Theorem \ref{b4stab}, it suffices to show that for all $n\geq N$, $E^{\delta/2}(\epsilon, n, \calP)\subseteq C^{\delta}_r(n)$.  Fix $n\geq N$ and suppose $H=(V,d)\in E^{\delta/2}(\epsilon, n, \calP)$.  We want to show $H\in C^{\delta}_r(n)$.  By definition of $E^{\delta/2}(\epsilon, n, \calP)$, there is some $H'=(V,d')\in E(\epsilon, n, \calP)$ such that $\dist(H,H')\leq \delta/2$.  By Observation \ref{ob8}, this implies $|\Delta(H,H')|\leq \delta {n\choose 2}$.  By definition of $E(\epsilon, n,\delta)$, there is $G'\in \calR([n],\calP)$ such that $H'\unlhd_pG'$ and $sub(G')>ex(n,\calP)^{1-\epsilon}$.  Recall that by Observation \ref{ob5}(a), for all $n$, $\ex(n,\calP)\geq \pi(\calP)^{{n\choose 2}}$, so this implies $sub(G')>\pi(\calP)^{(1-\epsilon){n\choose 2}}$.  Corollary \ref{cor1} implies $\Psi(G'):=(V,c')\in \tildeM_r(n)$, and Proposition \ref{0} implies $W(\Psi(G'))=sub(G')>\pi(\calP)^{(1-\epsilon){n\choose 2}}$.  Consequently, Theorem \ref{stabthm} implies $\Psi(G')$ is $\delta/2$-close in the sense of Definition \ref{newdeltaclose} to some $M=([n],c)\in \tildeC_r(n)$.  Define $H''=([n],d'')$ as follows.  
\begin{enumerate}[$\bullet$]
\item If $xy \notin \Delta(\Psi(G'), M) \cup \Delta (H,H')$, let $d''(xy)=d'(xy)=d(xy)$.
\item If $xy\in \Delta(\Psi(G'), M)\cup \Delta (H,H')$, let $d''(xy)$ be any element of $c(xy)$.  
\end{enumerate}
We claim $H''\subseteqq_{[r]}M$.  Fix $xy\in {[n]\choose 2}$.  We want to show $d''(xy)\in c(xy)$.  This holds by definition of $H''$ when $xy\in \Delta(\Psi(G'), G'')\cup \Delta (H,H')$.  If $xy\not\in \Delta(\Psi(G'), M)\cup \Delta (H,H')$, then $d''(xy)=d'(xy)$ and $c(xy)=c'(xy)$.  Since, $H'\unlhd_pG'$, Proposition \ref{translation} implies $H'\subseteqq_{[r]}\Psi(G')$, thus $d'(xy)\in c'(xy)$.  Since $d''(xy)=d'(xy)$ and $c(xy)=c'(xy)$, this implies $d''(xy)\in c(xy)$, as desired.  Therefore, $H''\subseteqq_{[r]}M$ and $M\in \tildeC_r(n)$ implies $H''\in C_r(n)$.  We claim $|\Delta(H,H'')|\leq \delta {n\choose 2}$.  By definition of $H''$,
$$
\Delta(H,H'')\subseteq \Delta(H, H')\cup \Delta(\Psi(G'), M),
$$
so $|\Delta(H,H'')|\leq |\Delta(H,H')|+|\Delta(\Psi(G'),M)|\leq \delta {n\choose 2}$, where the inequality is by our assumptions.  This shows $H''\in C_r(n)$ and $|\Delta(H,H'')|\leq \delta{n\choose 2}$, i.e. $H\in C_r^{\delta}(n)$.
\qed

\vspace{3mm}

We leave the following lemma to the reader.

\begin{lemma}\label{msuse}
If $M$ and $N$ are complete $2^{[r]}$-graphs with the same vertex set $V$, then $\Delta(M,N)=\diff(\Psi^{-1}(M),\Psi^{-1}(N))$.
\end{lemma}

We now show that when $r$ is even, $\calP$ has a stability theorem in the sense of Definition \ref{stabdef}, but when $r$ is odd, this is not the case.

\begin{theorem}\label{msstabeven}
If $r\geq 2$ is even, then $\calP$ has a stability theorem.
\end{theorem}
\begin{proof}
Fix $\delta>0$.  By Theorem \ref{stabthm}, there is $\epsilon>0$ and $M$ such that for all $n>M$ if $H\in \tildeM_r(n)$ satisfies $P(H)>m(r)^{(1-\epsilon){n\choose 2}}$, then $H$ is $\delta$-close in the sense of Definition \ref{newdeltaclose} to the unique $2^{[r]}$-graph $H_0\in \tildeC_r(n)$.  Suppose now that $n>M$ and $G\in \calR([n],\calP)$ satisfies $sub(G)\geq \ex(n,\calP)^{1-\epsilon}$.  We want to show there is $G'\in \calR_{ex}([n],\calP)$ such that $\dist(G,G')\leq \delta$.  Recall that by part (a) of Observation \ref{ob5}, for all $n$, $\pi(\calP)^{n\choose 2}\leq \ex(n,\calP)$.  Thus our assumptions imply $sub(G)\geq \ex(n,\calP)^{1-\epsilon}\geq \pi(\calP)^{(1-\epsilon){n\choose 2}}$.  Proposition \ref{mspi} implies $\pi(\calP)=m(r)$ and Corollary \ref{cor1} implies $\Psi(G)\in \tildeM_r(n)$, so by Proposition \ref{0}, 
$$
W(\Psi(G))=sub(G)\geq \pi(\calP)^{(1-\epsilon){n\choose 2}}=m(r)^{(1-\epsilon){n\choose 2}}.
$$
Thus Theorem \ref{stabthm} implies $\Psi(G)$ is $\delta$-close in the sense of Definition \ref{newdeltaclose} to the unique $2^{[r]}$-graph $H_0\in \tildeC_r(n)$.  By Lemma \ref{extlemeven}, $H_0$ is a product-extremal element of $\tildeM_r(n)$.  By Corollary \ref{cor1}, $\Psi^{-1}(H_0)\in \calR([n],\calP)$.  Since by Observation \ref{ob5}(a), $\Psi(\Psi^{-1}(H_0))=H_0$, Proposition \ref{1} and the fact that $H_0$ is product-extremal imply $\Psi^{-1}(H_0)\in \calR_{ex}([n],\calP)$.  By Lemma \ref{msuse} and Observation \ref{ob5}(b),
\begin{align}\label{labeldiff}
\Delta(H_0,\Psi(G))=\diff(\Psi^{-1}(H_0),\Psi^{-1}(\Psi(G)))=\diff(\Psi^{-1}(H_0), G).
\end{align}
Since $H_0$ and $\Psi(G)$ are $\delta$-close in the sense of Definition \ref{newdeltaclose}, $|\Delta(H_0, \Psi(G))|\leq \delta {n\choose 2}$.  Combining this with (\ref{labeldiff}), we have that $\dist(\Psi^{-1}(H_0),G)\leq \delta$.  Since $\Psi^{-1}(H_0)\in \calR_{ex}([n],\calP)$, we are done.
\end{proof}

\begin{corollary}\label{msstabodd}
When $r\geq 3$ is odd, $\calP$ does not have a stability theorem.
\end{corollary}
\begin{proof}
Let $A=([n],c)$ be such that for all $xy\in {[n]\choose 2}$, $c(xy)=L_r$.  Then by definition, $A\in \tildeM_r(n)$ and $W(A)=m(r)^{n\choose 2}$.  By Corollary \ref{cor1}, $\Psi^{-1}(A)\in \calR([n],\calP)$, and by Proposition \ref{0}, 
$$
sub(\Psi^{-1}(A))=W(A)=m(r)^{n\choose 2}.
$$
Let $B\in \calR_{ex}([n],\calP)$.   By Proposition \ref{extlemodd}, $\Psi(B)\in \tildeE_r(n)$. By definition of $A$ and $\tildeE_r(n)$, $\Delta(A,\Psi(B))={[n]\choose 2}$.  By Lemma \ref{msuse} and Observation \ref{ob5}(b),
$$
\Delta(A,\Psi(B))=\diff(\Psi^{-1}(A),\Psi^{-1}(\Psi(B))=\diff(\Psi^{-1}(A),B).
$$
Therefore, $\dist(\Psi^{-1}(A),B)=|\Delta(A,\Psi(B))|/{n\choose 2}=1$.  So for all $\delta<1$, $\Psi^{-1}(A)$ is not $\delta$-close to any element of $\calR_{ex}([n],\calP)$.  However, for all $\epsilon>0$, $\pi(\calP)=m(r)$ implies that if $n$ is sufficiently large, $sub(\Psi^{-1}(A))=m(r)^{n\choose 2}\geq ex(n,\calP)^{1-\epsilon}$, so $\Psi^{-1}(A)\in E(\epsilon, n, \calP)$.  This shows that $\calP$ does not have a stability theorem in the sense of Definition \ref{stabdef}.
\end{proof}

%*****************************************************************************
\section{Concluding Remarks}\label{end}
%*****************************************************************************

We end with some questions and conjectures.  Returning to the metric spaces of the previous section, it was shown in \cite{MT} the following is true.

\begin{theorem}[{\bf Mubayi-Terry \cite{MT}}] If $r\geq 2$ is even, then $M_r=\bigcup_{n\in \mathbb{N}}M_r(n)$ has a $0$-$1$ law in the language $\calL_r$.
\end{theorem}

It was then conjectured in \cite{MT} that this theorem is false in the case when $r$ is odd (to our knowledge, this is still open).  In the previous section, we showed that in the case when $r$ is even, $\calM_r$ has a stability theorem in the sense of Definition \ref{stabdef}, while when $r$ is odd, this is false.  These facts lead us to make the following conjecture.

\begin{conjecture}
Suppose $\calL$ is a finite relational language with $r_{\calL}\geq 2$, and $\calH$ is a fast-growing hereditary $\calL$-property such that $\bigcup_{n\in \mathbb{N}}\calH_n$ has a $0$-$1$ law.  Then $\calH$ has a stability theorem.
\end{conjecture}

The idea behind this conjecture is that if $\calH$ has nice asymptotic structure, it should reflect the structure of $\calR_{ex}([n],\calH)$.  Another phenomenon which can be observed from known examples is that the structures in $\calR_{ex}(n,\calH)$ are not very complicated.  The following questions are various ways of asking if this is always the case.  

\begin{question}
Suppose $\calL$ is a finite relational language with $r_{\calL}\geq 2$, and $\calH$ is a fast-growing hereditary $\calL$-property.  For each $n$, let $\calP_n=\calR_{ex}([n],\calH)$.  Does $\bigcup_{n\in \mathbb{N}} \calP_n$ always have a $0$-$1$ law?
\end{question}

We direct the reader to \cite{classification} for the definition of a formula having the $k$-order property.

\begin{question}
Suppose $\calL$ is a finite relational language with $r_{\calL}\geq 2$, and $\calH$ is a fast-growing hereditary $\calL$-property.  Is there a finite $k=k(\calH)$ such that for all $n$ and  $M\in \calR_{ex}(n,\calH)$, every atomic formula $\phi(x;y)\in \calL_{\calH}$ does not have the $k$-order property in $M$? 
\end{question}

A weaker version of this question is the following.  We direct the reader to \cite{simon2015guide} for the definition of the $VC$-dimension of a formula.

\begin{question}
Suppose $\calL$ is a finite relational language with $r_{\calL}\geq 2$, and $\calH$ is a fast-growing hereditary $\calL$-property.  Is there a finite $k=k(\calH)$ such that for all $n$ and  $M\in \calR_{ex}(n,\calH)$, every atomic formula $\phi(x;y)\in \calL_{\calH}$ has $VC$-dimension bounded by $k$ in $M$?
\end{question}

\appendix
%*******************************************************************
\section{Hereditary properties of colored hypergraphs}\label{coloredhg}
%*******************************************************************
In this section we show that Theorem \ref{enumeration} agrees with an existing enumeration theorem for hereditary properties of colored $k$-uniform hypergraphs which were proved by Ishigami in \cite{Ishigami}.  We include this example as it is the most general enumeration theorem of hereditary properties in the literature (to our knowledge).  As is pointed out in \cite{Ishigami}, these results extend those for hereditary properties of hypergraphs from \cite{DN} as well as enumeration results for hereditary graph properties in \cite{Alekseev2, BoTh2}. 

\subsection{Statements of Results from \cite{Ishigami}.}
 The definitions and results in this section are from \cite{Ishigami}.  Given an integer $k\geq 2$ and a set $C$, a \emph{$C$-colored $k$-uniform hypergraph}, also called a \emph{$(k,C)$-graph}, is a pair $G=(V,H)$, where $V$ is a vertex set and $H:{V\choose k}\rightarrow C$ is a function.  The set $C$ is called the set of colors.  Given two $(k,C)$-graphs $G=(V,H)$ and $G'=(V',H')$, $G$ is a \emph{subgraph} of $G'$ if $V\subseteq V'$ and for all $e\in {V\choose k}$, $H(e)=H'(e)$.  We say $G$ and $G'$ are \emph{isomorphic}, denoted $G\cong G'$, if there is a bijection $f:V\rightarrow V'$ such that for all $e\in {V\choose k}$, $H(e)=H'(f(e))$.  A hereditary property of $(k,C)$-graphs is a nonempty class of finite $(k,C)$-graphs which is closed under subgraphs and isomorphism.  Observe that if $C$ has only two elements, then $(k,C)$-graphs can be seen as  $k$-uniform hypergraphs.  

Assume $k\geq 2$, $C$ is finite, and $\calP$ is a hereditary property of $(k,C)$-graphs.  Let $2^C$ denote the powerset of $C$.  Then an $n$-vertex $(k, 2^{C}\setminus \emptyset)$-graph $G=(V,H)$ is call \emph{$\calP$-good} if and only if $\calP$ contains any $(k,C)$-graph $G'=(V,H')$ with the property that for all $e\in {V\choose k}$, $H'(e)\in H(e)$.  Define
$$
\max(n,\calP)=\max\Big\{ \frac{1}{{n\choose k}} \sum_{e\in {[n]\choose k}}\log_2|H(e)|: ([n],H) \text{ is a $\calP$-good $(k, 2^{C}\setminus \emptyset)$-graph}\Big\}.
$$
This notion comes from \cite{Ishigami} and is denoted there by ``$\ex(n,\calP)$.''  We have changed the notation to avoid confusion with Definition \ref{deltaclosedef1}.  The following is Theorem 1.1 from \cite{Ishigami}.

\begin{theorem}[{\bf Ishigami \cite{Ishigami}}]\label{ishigamisthm}
Suppose $k\geq 2$ is a fixed finite integer and $C$ is a fixed finite set.  If $\calP$ is a hereditary property of $(k,C)$-graphs then 
$$
|\calP_n|=2^{(\max(n,\calP)+o(1)){n\choose k}}.
$$
\end{theorem}

The goal of this section is to reprove Theorem \ref{ishigamisthm} using the machinery of this paper.  We begin by interpreting the basic definitions from this paper in the setting of colored hypergraphs.  Fix $k\geq 2$, a finite set $C$, and a hereditary property of $(k,C)$-graphs $\calP$.  Assume that for all $c\in C$, $\calP$ contains structures $G=(V,H)$ such that there is $e\in {V\choose k}$ with $H(e)=c$ (otherwise replace $C$ with a smaller set).  Let $\calL=\{E_c(x_1,\ldots, x_k): c\in C\}$ consist of a single $k$-ary relation symbol for every color $c\in C$.  We will consider $(k,C)$-graphs as $\calL$-structures in the obvious way, namely, given a $(k,C)$-graph $G=(V,H)$, $\abar \in V^k$, and $c\in C$, define $G\models E_c(\abar)$ if and only if $|\cup \abar|=k$ and $c\in H(\cup \abar)$.  Since $r_{\calL}=k$, $\calL_{\calP}=\{R_p(x_1,\ldots, x_k): p\in S_k(\calP)\}$.  For each $c\in C$, define $q_c(x_1,\ldots, x_k)$ to be the following set of formulas, where $\xbar=(x_1,\ldots, x_k)$:
\begin{align*}
&\{x_i\neq x_j: i\neq j\}\cup \{E_c(\mu(\xbar)): \mu \in Perm(k)\}\cup \\
&\{\neg R_{c'}(\mu(\xbar)): c'\neq c\in C, \mu \in Perm(k)\}\cup \{\neg R(\xbar'): \cup \xbar'\subsetneq \xbar\}.
\end{align*}
Let $p_c(\xbar)$ be the unique quantifier-free $k$-type containing $q_c(\xbar)$.  We leave it to the reader to verify that $S_k(\calP)=\{p_c(\xbar): c\in C\}$ (note this uses the assumption that $\calP$ contains structures with edges colored by $c$ for each $c\in C$).  The following is an important definition which allows us to translate Ishigami's result to our setting.

\begin{definition} Suppose $G$ is a complete $\calL_{\calP}$-structure with domain $V$.  The \emph{$(k, 2^{C}\setminus \{\emptyset\})$-graph associated to $G$} is $\Psi(G):=(V, H)$, where for each $e\in {V\choose k}$, 
$$
H(e)=\{c\in C: \text{for some enumeration $\bar{e}$ of $e$, $G\models R_{p_c}(\bar{e})$}\}.
$$
\end{definition}

Given a $(k, 2^{C}\setminus \{\emptyset\})$-graph $(V,H)$, let $\Psi^{-1}(V,H)$ be the $\calL_{\calP}$-structure with domain $V$ such that for each $\abar\in V^k$ and $c\in C$, $G\models R_{p_c}(\abar)$ if and only if $|\cup \abar|=k$ and $c\in H(\cup \abar)$.  We leave the following observations to the reader.

\begin{observation}\label{obchg}
Suppose $(V,H)$ is a $(k, 2^{C}\setminus \{\emptyset\})$-graph. Then
\begin{enumerate}[(a)]
\item  $\Psi(\Psi^{-1}(V,H))=(V,H)$ and 
\item $\Psi^{-1}(V,H)$ is an $\calL_{\calP}$-template. 
\end{enumerate}
\end{observation}

\begin{proposition}\label{bijection}
Suppose $G$ is an $\calL_{\calP}$-template with domain $V$ and $(V,H)=\Psi(G)$.  Then for any $(k,C)$-graph $M=(V,H')$, $M\unlhd_pG$ if and only if for all $e\in {V\choose k}$, $H'(e)\in H(e)$.
\end{proposition}

\begin{proof}
Suppose $M=(V,H')\unlhd_pG$.  We want to show that for all $e\in {V\choose 2}$, $H'(e)\in H(e)$.  Fix $e\in {V\choose k}$ and let $c=H'(e)$.  Note this means $qftp^M(\ebar)=p_c(\xbar)$ where $\ebar$ is any enumeration of $e$.  Then $M\unlhd_pG$ implies there is an enumeration $\ebar$ of $e$ such that $G\models R_{p_c}(\ebar)$.  By definition of $\Psi(G)$, $c\in H(e)$.  Thus $H'(e)\in H(e)$.  Conversely, suppose $M=(V,H')$ and for all $e\in {V\choose k}$, $H'(e)\in H(e)$.  We define a choice function $\chi$ for $G$.  Fix $e\in {V\choose k}$ and let $c=H'(e)$.  Then by assumption, $c\in H(e)$, which implies by definition of $\Psi$, $G\models R_{p_c}(\ebar)$ for some enumeration $\ebar$ of $e$.  Define $\chi(e)=p_c(c_{\ebar})$.   Then by construction, $\chi\in Ch(G)$ and $M\unlhd_{\chi}G$, thus $M\unlhd_p G$.
\end{proof}

\begin{proposition}\label{subcount}
If $V$ is a finite set, $G$ is an $\calL_{\calP}$-template with domain $V$, and $\Psi(G)=(V,H)$.  Then $sub(G)=\prod_{e\in {V\choose k}}|H(e)|$.
\end{proposition}
\begin{proof}
By Proposition \ref{bijection}, $sub(G)=|\{M: M\unlhd_pG\}|$ is the same as the number of $(k,C)$-graphs $(V,H')$ with the property that for all $e\in {V\choose k}$, $H'(e)\in H(e)$.  Clearly the number of such $(k,C)$-graphs is $\prod_{e\in {V\choose k}}|H(e)|$.
\end{proof}

\begin{proposition}\label{idk}
Let $V$ be a set.  If $(V,H)$ is a $\calP$-good $(k,2^C\setminus \{\emptyset\})$-graph, then $\Psi^{-1}(V,H)\in \calR(V,\calP)$.  If $G\in \calR(V,\calP)$, then $\Psi(G)$ is a $\calP$-good $(k,2^C\setminus \{\emptyset\})$-graph.
\end{proposition}
\begin{proof}
Suppose first $(V,H)$ is a $\calP$-good $(k,2^C\setminus \{\emptyset\})$-graph.  By part (a) of Observation \ref{obchg}, $\Psi^{-1}(V,H)$ is an $\calL_{\calP}$-template with domain $V$.  To show $\Psi^{-1}(V,H)$ is $\calP$-random, fix $M\unlhd_p \Psi^{-1}(V,H)$, say $M=(V,H')$.  We want to show $M\in \calP$.  By part (b) of Observation \ref{obchg}, $\Psi(\Psi^{-1}(V,H))=(V,H)$.  Therefore, Proposition \ref{bijection} implies for all $e\in {V\choose k}$, $H'(e)\in H(e)$.  Since $(V,H)$ is $\calP$-good, this implies $M\in \calP$.  Thus $\Psi^{-1}(V,H)\in \calR(V,\calP)$.  

Suppose $G\in \calR(V,\calP)$.  Let $\Psi(G)=(V,H)$.  To show $(V,H)$ is a $\calP$-good $(k,2^C\setminus \{\emptyset\})$-graph, let $M=(V,H')$ be such that for all $e\in {V\choose k}$, $H'(e)\in H(e)$.  By Proposition \ref{bijection}, $M\unlhd_pG$.  Since $G$ is $\calP$-random, $M\in \calP$.
\end{proof}

\begin{corollary}\label{coloredcor}
For all $n$, $\ex(n,\calP)=2^{\max(n,\calP){n\choose k}}$..
\end{corollary}
\begin{proof}
We first show $\ex(n,\calP)\leq 2^{\max(n,\calP){n\choose k}}$.  Fix $G\in \calR_{ex}([n],\calP)$ and let $\Psi(G)=([n],H)$.  By Proposition \ref{idk}, $\Psi(G)$ is a $\calP$-good $(k,2^C\setminus \{\emptyset\})$-graph, so by definition of $\max(n,\calP)$, we have that $\max(n,\calP)\geq \frac{1}{{n\choose k}}\sum_{e\in {[n]\choose k}}\log_2|H(e)|$.  Therefore 
$$
2^{\max(n,\calP){n\choose k}}\geq 2^{\sum_{e\in {[n]\choose k}}\log_2|H(e)|}=\prod_{e\in {[n]\choose k}}|H(e)|=sub(G),
$$
where the last equality is by Proposition \ref{subcount}.  This shows $2^{\max(n,\calP){n\choose k}}\geq sub(G)=ex(n,\calP)$, where the equality is because $G\in \calR_{ex}([n],\calP)$.  We now show $\ex(n,\calP)\geq 2^{\max(n,\calP){n\choose k}}$.  Let $G'=([n],H)$ be a $\calP$-good $(k, 2^{C}\setminus \{\emptyset\})$-graph such that $\max(n,\calP)=\frac{1}{{n\choose k}}\sum_{e\in {[n]\choose k}}\log_2|H(e)|$.  Let $G=\Psi^{-1}(G')$.  Proposition \ref{idk} implies $G\in \calR([n],\calP)$, so $sub(G)\leq ex(n,\calP)$.  By Proposition \ref{subcount}, $sub(G)=\prod_{e\in {[n]\choose k}}|H(e)|$.  Thus we have
$$
\prod_{e\in {[n]\choose k}}|H(e)|= 2^{\sum_{e\in {[n]\choose k}}\log_2|H(e)|}=2^{\max(n,\calP){n\choose k}},
$$
where the last equality is by choice of $G'$.  Therefore $2^{\max(n,\calP){n\choose k}}=sub(G)\leq ex(n,\calP)$, as desired.  
\end{proof}

We now give a restatement of Theorem \ref{enumeration} which will be convenient for us.

\begin{theorem}[Restatement of Theorem \ref{enumeration}]\label{enumeration2}
Suppose $\calL$ is a finite relational language satisfying $r=r_{\calL}\geq 2$ and $\calH$ is a hereditary $\calL$-property.  Then $|\calH_n|=ex(n,\calH)2^{o(n^r)}$.
\end{theorem}
\begin{proof}
By Theorem \ref{enumeration}, it suffices to show 
\[
ex(n,\calH)2^{o(n^r)}=\begin{cases} \pi(\calH)^{{n\choose r}+o(n^r)} & \text{ if } \pi(\calH)>1.\\
2^{o(n^r)} &\text{ if }\pi(\calH)=1. \end{cases}
\]
This is obvious by definition of $\pi(\calH)=lim_{n\rightarrow \infty}ex(n,\calH)^{1/{n\choose r}}$.
\end{proof}

We now see that Theorem \ref{ishigamisthm} follows easily from Theorem \ref{enumeration2}.  
\vspace{3mm}

\noindent{\bf Proof of Theorem \ref{ishigamisthm}.} Suppose $\calP$ is a hereditary property of $(k,C)$-graphs.  Theorem \ref{enumeration2} and Corollary \ref{coloredcor} imply $|\calP_n|=\ex(n,\calP)2^{o(n^k)}=2^{\max(n,\calP){n\choose k}+o(n^k)}$.

%*****************************************************************************
\section{Digraphs omitting transitive tournamets}\label{dgsec}
%*****************************************************************************

In this section we consider results by K\"{u}hn, Oshtus, Townsend, and Zhao on asymptotic enumeration and structure of digraphs omitting transitive tournaments \cite{OKTZ}.  In particular we will reprove approximate structure and enumeration theorems from \cite{OKTZ} using the main results of this paper along with extremal and stability results from \cite{OKTZ}.  We include this example to demonstrate our results apply to structures in languages with non-symmetric relations.  We would like to point out that \cite{OKTZ} also investigates oriented graphs omitting transitive tournaments as well as digraphs and oriented graphs omitting cycles.  Moreover, their results go much further than the ones we state here, proving precise structure and enumeration results in various cases. 

\subsection{Statements of results from \cite{OKTZ}.}
We begin with some preliminaries on digraphs and statements of results from \cite{OKTZ}.  A \emph{digraph} is a pair $(V,E)$ where $V$ is a set of vertices and $E\subseteq V^{\underline{2}}$ is a set of \emph{directed edges}.   A \emph{tournament} on $k$ vertices is an orientation of the complete graph on $k$ vertices.  In other words, it is a digraph $(V,E)$ such that $|V|=k$ and for all $xy\in {V\choose 2}$ exactly one of $(x,y)$ or $(y,x)$ is in $E$.  A tournament $(V,E)$ is called \emph{transitive} if for all $x,y,z\in V$, $(x,y)\in E$ and $(y,z)\in E$ implies $(x,z)\in E$.  Suppose $G=(V,E)$ and $G'=(V',E')$ are digraphs.  We say $G'$ is a \emph{subdigraph} of $G$, denoted $G'\subseteq G$, if $V'\subseteq V$ and $E'\subseteq E$.  If $V=V'$ then $G$ is a \emph{full digraph} of $G'$, denoted $G\subseteqq G'$.  We say $G$ and $G'$ are \emph{isomorphic}, denoted $G\cong G'$, if there is a bijection $f:V\rightarrow V'$ such that for all $xy\in {V\choose 2}$, $(x,y)\in E$ if and only if $(f(x),f(y))\in E'$.  Given digraphs $H$ and $G$, we say $G$ is \emph{$H$-free} if there is no $G'\subseteq G$ with $G\cong H$.  For the rest of this section, $k\geq 2$ is a fixed integer and $T_{k+1}$ is a fixed transitive tournament on $k+1$ vertices.  The subject of this section is the following set of digraphs, where $n\in \mathbb{N}$.
$$
\Forb_{di}(n, T_{k+1})=\{G=([n],E): \text{ $G$ is a $T_{k+1}$-free digraph}\}.
$$
Recall that a $k$-partite Tur\'{a}n graph on a finite vertex set $V$ is a complete balanced $k$-partite graph on $V$ such that the sizes of the parts differ by at most $1$.  Let $T_k(n)$ be the set of $k$-partite Tur\'{a}n graphs with vertex set $[n]$ and let $t_k(n)$ be the number of edges in an element of $T_k(n)$.  Let $DT_k(n)$ be the set of digraphs which can be obtained from an element of $T_k(n)$ by replacing all the edges with two directed edges.  More precisely, $DT_k(n)$ is the set of digraphs $G=([n],E)$ such that for some $G'=([n], E')\in T_k(n)$, $E=\{(x,y)\in [n]^2: xy\in E'\}$.  Given a digraph $G=(V,E)$, set
\begin{align*}
f_1(G)&=\{xy\in {V\choose 2}: \text{ exactly one of $(x,y)$ or $(y,x)$ is in $E$}\},\\
f_2(G)&=\{xy\in {V\choose 2}: \text{ $(x,y)$ and $(y,x)$ are in $E$}\},\text{ and }\\
e(G)&= f_1(G)+\log_2(3)f_2(G).
\end{align*}
Observe that in this notation the number of full subdigraphs of $G$ is $2^{e(G)}$.  Let 
$$
max(n, T_{k+1})=\max \{e(G): G\in \Forb_{di}(n, T_{k+1})\}.
$$
This notion is called ``$\ex(n,T_{k+1})$'' in \cite{OKTZ}.  We have changed the notation to avoid confusion with Definition \ref{genexdef}.  A digraph $G\in \Forb_{di}(n, T_{k+1})$ is \emph{edge-extremal} if $e(G)=max(n,T_{k+1})$.   The following is Lemma 4.1 from \cite{OKTZ}.

\begin{theorem}[{\bf K\"{u}hn-Oshtus-Townsend-Zhao \cite{OKTZ}}]\label{digext}
For all $n\in \mathbb{N}$, $max(n,T_{k+1})=t_k(n)\log_2(3)$ and $DT_k(n)$ is the set of all edge-extremal elements of $\Forb_{di}(n, T_{k+1})$.
\end{theorem}

\begin{definition}\label{deltaclosedg}
Given two digraphs $G=(V,E)$ and $G'=(V,E)$, let $\Delta(G,G')= E\Delta E'$.  Given $\delta>0$, we say $G$ and $G'$ are \emph{$\delta$-close}\footnote{Given $G$ and $G'$ two digraphs with the same vertex set of size $n$, write $G=G'\pm \delta n^2$ to denote that $G$ can be obtained from $G'$ by adding or removing at most $\beta n^2$ directed edges.  In \cite{OKTZ}, the authors deal with the notion of $\delta$-closeness given by $G=G'\pm \delta n^2$.  For large $n$, this is essentially interchangable (up to a factor of $2$) with Definition \ref{deltaclosedg}.  For convenience we will state their results using the notion in Definition \ref{deltaclosedg}.} if $|\Delta(G,G')|\leq \delta {n\choose 2}$.
\end{definition}
  Given $n$ and $\delta$, let 
\begin{align*}
DT^{\delta}_k(n)&=\{G\in \Forb_{di}(n, T_{k+1}): |\Delta(G,G')|\leq \delta {n\choose 2}\text{ some }G'\in DT_k(n)\}\text{ and }\\
\mathbb{DT}_k(n)&=\{G\in \Forb_{di}(n, T_{k+1}): G\subseteqq G'\text{ some }G'\in DT_k(n)\}\text{ and }\\
\mathbb{DT}^{\delta}_k(n)&=\{G\in \Forb_{di}(n, T_{k+1}): |\Delta(G,G')|\leq \delta {n\choose 2}\text{ some }G'\in \mathbb{DT}_k(n)\}
\end{align*}
The following is Lemma 4.3 from \cite{OKTZ}.

\begin{theorem}[{\bf K\"{u}hn-Oshtus-Townsend-Zhao \cite{OKTZ}}]\label{digstab}
Let $k\geq 2$.  For all $\delta>0$ there is $\beta>0$ such that the following holds for all sufficiently large $n$.  If $G\in \Forb_{di}(n, T_{k+1})$ satisfies that $e(G)\geq max(n,T_{k+1})-\epsilon {n\choose 2}$, then $G\in DT^{\delta}_k(n)$.
\end{theorem}

The next theorem we state is Lemma 5.1 in \cite{OKTZ}. In fact, a much stronger result is proven there, where the $o(n^2)$ error is replaces with $O(n)$. 
\begin{theorem}[{\bf K\"{u}hn-Oshtus-Townsend-Zhao \cite{OKTZ}}]\label{digen}
Let $k\geq 2$. $|\Forb_{di}(n,T_{k+1})|=3^{t_k(n)+o(n^2)}$.
\end{theorem}

The following approximate structure theorem follows from the proof of Lemma 4.5 from \cite{OKTZ}.  

\begin{theorem}[{\bf K\"{u}hn-Oshtus-Townsend-Zhao \cite{OKTZ}}]\label{digapprox}
Let $k\geq 2$.  For all $\delta>0$ there is $\beta>0$ such that the following holds for all sufficiently large $n$.  
$$
\frac{|\Forb_{di}(n, T_{k+1})\setminus \mathbb{DT}^{\delta}_k(n)|}{|\Forb_{di}(n, T_{k+1})|}\leq 2^{-\beta {n\choose 2}}.
$$
\end{theorem}

\noindent We will use the machinery of this paper along with Theorems \ref{digext} and \ref{digstab} to reprove Theorems \ref{digen} and \ref{digapprox}.

\subsection{Translation to hereditary $\calL$-property.} In this subsection we translate the combinatorial notions from \cite{OKTZ} into the setup used in this paper.  Let $\mathcal{L}=\{E(x,y)\}$ consist of a single binary relation symbol.  We consider digraphs as $\mathcal{L}$-structures in the natural way, that is, given a digraph $G=(V,E)$ and $(x,y)\in V^2$, $G\models R(x,y)$ if and only if $(x,y)\in E$.  Let $\mathcal{P}$ to be the closure of $\bigcup_{n\in \mathbb{N}}\Forb_{di}(n,T_{k+1})$ under isomorphism.  Clearly $\mathcal{P}$ is a hereditary $\calL$-property.  Since $r_{\mathcal{L}}=2$, $\calL_{\calP}=\{R_p(x,y): p\in S_2(\mathcal{P})\}$.  Set
\begin{enumerate}
\item $q_1(x,y)=\{x\neq y, E(x,y), \neg E(y,x)\}$
\item $q_2(x,y)=\{x\neq y, E(y,x), \neg E(x,y)\}$
\item $q_3(x,y)=\{x\neq y, E(y,x), E(x,y)\}$
\item $q_4(x,y)=\{x\neq y, \neg E(y,x), \neg E(x,y)\}$.
\end{enumerate}
For each $i=1,2,3,4$, let $p_i(x,y)$ be the unique complete quantifier-free $2$-type extending $q_i(x,y)$.  We leave it to the reader to verify that $S_2(\mathcal{P})=\{p_i(x,y): i\in [4]\}$.  Therefore, 
$$
\calL_{\calP}=\{R_{p_1}(x,y), R_{p_2}(x,y),R_{p_3}(x,y),R_{p_4}(x,y)\}.
$$  
We say an $\calL_{\calP}$-structure $G$ is \emph{downward closed} if $G\models \forall x\forall y (R_{p_3}(x,y)\leftrightarrow R_{p_1}(x,y)\wedge R_{p_2}(x,y))$ and $G\models \forall x \forall y (x\neq y\rightarrow R_{p_4}(x,y))$.  

\begin{definition}
Suppose $G$ is a complete $\calL_{\calP}$-structure $G$ with domain $V$.  The \emph{digraph associated to $G$} is $\Psi(G):=(V,E)$ where for each $uv\in {V\choose 2}$, $(u,v)\in E$ if and only if 
$$
G\models R_{p_1}(u,v)\vee R_{p_2}(u,v)\vee R_{p_3}(u,v)\vee R_{p_3}(u,v).
$$
In other words, $(u,v)\in E$ if and only if there is $p(c_u,c_v)\in Ch_G(uv)$ such that $E(x,y)\in p(x,y)$.
\end{definition}

Given a digraph $(V,E)$, define $\Psi^{-1}(V,E)$ to be the $\calL_{\calP}$-structure with domain $V$ such that for all $(u,v)\in V^2$, the following hold.
\begin{align*}
&G\models R_{p_1}(u,v)\Leftrightarrow  (u,v)\in E,\\
&G\models R_{p_2}(u,v)\Leftrightarrow  (v,u)\in E,\\
&G\models R_{p_3}(u,v)\Leftrightarrow  (u,v), (v,u)\in E,\text{ and }\\
&G\models R_{p_4}(u,v)\Leftrightarrow  v\neq u.  
\end{align*}

We leave the following observation to the reader.

\begin{observation}\label{dgob1}
Suppose $(V,E)$ is a digraph and $G$ is a downward closed $\calL_{\calP}$-template with domain $V$.  Then
\begin{enumerate}[(a)]
\item $\Psi(\Psi^{-1}(V,E))=(V,E)$ and $\Psi^{-1}(V,E)$ is a downward closed $\calL_{\calP}$-template.
\item $\Psi^{-1}(\Psi(G))=G$.
\end{enumerate}
\end{observation}

\begin{lemma}\label{diglem2}
If $G$ is an $\calL_{\calP}$-template, then for any digraph $G'$, $G'\unlhd_pG$ implies $G'\subseteqq \Psi(G)$.  If further $G$ is downward closed, then $\Psi(G)\unlhd_pG$ and for any $G'\subseteqq \Psi(G)$, $G'\unlhd_pG$. 
\end{lemma}
\begin{proof}
Suppose first $G$ is an $\calL_{\calP}$-template with domain $V$ and $G'\unlhd_pG$.  It is clear this implies $G'$ is a digraph with vertex set $V$.  Let $G'=(V,E')$ and $\Psi(G)=(V,E)$. We want to show $E'\subseteq E$.   Fix $(x,y)\in E'$.  Then $qftp^{G'}(x,y)\in \{p_1(x,y),p_3(x,y)\}$.  Since $G'\unlhd_pG$, this means either $G\models R_{p_1}(x,y)$ or $G\models R_{p_3}(x,y)$.  In either case, by definition of $\Psi(G)$, $(x,y)\in E$.  

Suppose now $G$ is also downward closed.  We show that $\Psi(G)\unlhd_pG$.  Fix $uv\in {V\choose 2}$.  We want to show that if $p(x,y)=qftp^{\Psi(G)}(u,v)$, then $p(c_u,c_v)\in Ch_G(uv)$.  If $p(x,y)=p_1(x,y)$, then $(u,v)\in E$ and $(v,u)\notin E$.  By definition of $\Psi(G)$, we must have $G\models R_{p_1}(u,v)\vee R_{p_2}(v,u)\vee R_{p_3}(u,v)\vee R_{p_3}(v,u)$ and $G\models \neg R_{p_3}(u,v)\wedge \neg R_{p_3}(v,u)$.  Thus $G\models R_{p_1}(u,v)\vee R_{p_2}(v,u)$.  Since $p_2(y,x)=p_1(x,y)$ this implies by definition that $p_1(c_u,c_v)\in Ch_G(uv)$, as desired.  The case when $p(x,y)=p_2(x,y)$ follows by a symmetric argument.  If $p(x,y)=p_3(x,y)$, then $(u,v)$ and $(v,u)$ are in $E$, so by definition of $\Psi(G)$ and because $G$ is downward closed, $G\models R_{p_3}(u,v)$.  Thus $p_3(c_u,c_v)\in Ch_G(uv)$ as desired.  If $p(x,y)=p_4(x,y)$, then neither $(u,v)$ nor $(v,u)$ are in $E$, so by definition of $\Psi(G)$, 
$$
G\models \bigwedge_{i=1}^3 (\neg R_{p_i}(u,v)\wedge \neg R_{p_i}(v,u)).
$$
Since $G$ is complete, this implies $G\models R_{p_4}(u,v)\vee R_{p_4}(v,u)$, which implies $p_4(c_u,c_v)\in Ch_G(uv)$, as desired.  This finishes the proof that $\Psi(G)\unlhd_pG$.  

Suppose now $G'\subseteqq \Psi(G)$.  Let $G'=(V,E')$ and $\Psi(G)=(V,E)$. Fix $uv\in {V\choose 2}$.  We want to show $p(x,y)=qftp^{G'}(u,v)$, then $G\models R_p(u,v)$.   Because $G$ is downward closed, if $p(x,y)=p_4(x,y)$, then we are done since $u\neq v$ implies $G\models R_{p_4}(u,v)$.  If $p(x,y)=p_1(x,y)$, then $(u,v)\in E'$.  Since $G'\subseteqq \Psi(G)$, this implies $(u,v)\in E$.  By definition of $\Psi(G)$, this implies 
$$
G\models R_{p_1}(u,v)\vee R_{p_2}(v,u)\vee R_{p_3}(u,v)\vee R_{p_3}(v,u).
$$
Because $G$ is a downward closed $\calL_{\calP}$-template and $p_1(x,y)=p_2(y,x)$, this implies $G\models R_{p_1}(u,v)$, as desired.  A similar argument takes care of the case when $p(x,y)=p_2(x,y)$.  Suppose now $p(x,y)=p_3(x,y)$.  Then $(u,v),(v,u)\in E'$.  Since $G'\subseteqq \Psi(G)$, this implies $(u,v),(v,u)\in E$.  By definition of $\Psi(G)$ and because $G$ is downward closed, this implies 
\begin{align*}
G\models R_{p_3}(u,v)\vee R_{p_3}(v,u).
\end{align*}
Since $p_3(x,y)=p_3(y,x)$ and $G$ is an $\calL_{\calP}$-template, this implies $G\models R_{p_3}(u,v)$, so we are done. 
\end{proof}

\begin{lemma}\label{diglem9}
If $H=(V,E)\in \Forb_{di}(n,T_{k+1})$, then $\Psi^{-1}(H)$ is a downward closed element of $\calR([n],\calP)$.
\end{lemma}
\begin{proof}
By Observation \ref{dgob1}, $\Psi^{-1}(H)$ is a downward closed $\calL_{\calP}$-template with domain $[n]$.  To show $\Psi^{-1}(H)$ is $\calP$-random, let $G\unlhd_p\Psi^{-1}(H)$.  By Lemma \ref{diglem2}, this implies $G\subseteqq\Psi(\Psi^{-1}(H))=H$, where the equality is by Observation \ref{dgob1}.  Since $H$ is $T_{k+1}$-free, so is any subdigraph of $H$.  Thus $G\in \calP$.  This shows $\Psi^{-1}(H)$ is $\calP$-random.
\end{proof}

\begin{corollary}\label{digcor2}
If $G$ is an $\calL_{\calP}$-template. Then $sub(G)\leq 2^{e(\Psi(G))}$ and equality holds if $G$ is downward closed.
\end{corollary}
\begin{proof}
Since $G$ is an $\calL_{\calP}$-template, Lemma \ref{diglem2} implies $sub(G)$ is at most the number of subdigraphs of $\Psi(G)$, which is equal to $2^{e(\Psi(G))}$ by definition of $e(\Psi(G))$.  If $G$ is downward closed, then equality holds by Lemma \ref{diglem9}.
\end{proof}

\begin{proposition}\label{digprop2}
Suppose $G$ is a finite downward closed $\calL_{\calP}$-template.  Then $G$ is $\calP$-random if and only if $\Psi(G)$ is $T_{k+1}$-free.
\end{proposition}
\begin{proof}
Suppose $\Psi(G)$ is not $T_{k+1}$-free.  Then $\Psi(G)\notin \calP$.  By Lemma \ref{diglem2}, $\Psi(G)\unlhd_pG$, so this implies $G$ is not $\calP$-random.  Conversely, suppose $G$ is not $\calP$-random.  Then there is $G'\unlhd_pG$ such that $G'\notin \calP$.  In other words $G'$ is a digraph which is not $T_{k+1}$-free.  By Lemma \ref{diglem2}, $G'\subseteq \Psi(G)$.  This implies $\Psi(G)$ is not $T_{k+1}$-free.
\end{proof}

We end this section by proving Proposition \ref{digprop3}, which tells us that elements in $\calR_{ex}([n],\calP)$ correspond to elements of $DT_k(n)$.  We first prove a lemma.

\begin{lemma}\label{dgdw}
If $G\in \calR([n],\calP)$ then there is $G^*\in \calR([n],\calP)$ which is downward closed with the property that $sub(G^*)\geq sub(G)(4/3)^{|\diff(G,G^*)|}$.
\end{lemma}
\begin{proof}
Suppose $G\in \calR([n],\calP)$.  If $G$ is downward closed, set $G^*=G$.  If $G$ is not downward closed, let $G^*$ be an $\calL_{\calP}$-template with domain $[n]$ and satisfing the following.
\begin{itemize}
\item $G^*\models \forall x \forall y (x\neq y \rightarrow R_{p_4}(x,y))$.
\item If $G\models R_{p_3}(u,v)$ then $G^*\models R_{p_3}(u,v)\wedge R_{p_1}(u,v)\wedge R_{p_2}(u,v)$.
\item If $G\models R_{p_1}(u,v)$ then $G^* \models R_{p_1}(u,v)$ and if $G\models R_{p_2}(u,v)$ then $G^* \models R_{p_2}(u,v)$.
\end{itemize}
We leave it to the reader to verify that such a $G^*$ exists and that $G^*$ has the property that $Ch_{G^*}(uv)\supseteq Ch_G(uv)$, for all $(u,v)\in [n]^2$.  If $uv\in \diff(G,G^*)$, then $Ch_{G^*}(uv)\supsetneq Ch_G(uv)$ and 
$$
4=|S_2(\calP)|\geq |Ch_{G^*}(uv)|>|Ch_G(uv)|\geq 1
$$
 imply $\frac{|Ch_{G^*}(uv)|}{|Ch_G(uv)|}\geq 4/3$.  Thus Corollary \ref{digcor2} implies
\begin{align*}
sub(G^*)=\prod_{uv\in {V\choose 2}}|Ch_{G^*}(uv)|=sub(G)\prod_{\{uv\in  \diff(G^*,G)\}}\frac{|Ch_{G^*}(uv)|}{|Ch_{G}(uv)|}\geq sub(G)\Big(4/3\Big)^{|\diff(G,G^*)|}.
\end{align*}
We have only left to show that $G^*\in \calR([n],\calP)$.  Let $H\unlhd_pG^*$.  We want to show $H\in \calP$.  By definition of $G^*$ and $\Psi$, $\Psi(G)$ is the same digraph as $\Psi(G^*)$.  By Proposition \ref{diglem2}, $H$ is a subdigraph of $\Psi(G)=\Psi(G^*)$.  By Proposition \ref{diglem2}, $\Psi(G)\unlhd_pG$.  Since $G$ is $\calP$-random, this implies $\Psi(G)\in \calP$, which implies any subdigraph of $\Psi(G)$ is in $\calP$.  In particular, $H\in \calP$.
\end{proof}

\begin{corollary}\label{dgcor}
If $G\in \calR_{ex}([n],\calP)$, then $G$ is downward closed.
\end{corollary}
\begin{proof}
Suppose $G\in \calR_{ex}([n],\calP)$.  By definition $sub(G)=\ex(n,\calP)$.  By Lemma \ref{dgdw}, if $G$ is not downward closed, then there is $G^*\in \calR([n],\calP)$ which is downward closed and which satisfies $sub(G^*)\geq sub(G)(4/3)^{|\diff(G^*,G)|}$.  Since $G^*$ is downward closed and $G$ is not downward closed, $G^*\neq G$ implies $|\diff(G^*,G)|\geq 1$.  Thus $sub(G^*)>sub(G)$, contradicting that $G\in \calR_{ex}([n],\calP)$.
\end{proof}

\begin{proposition}\label{digprop3}
Suppose $G$ is an $\calL_{\calP}$-template with domain $[n]$.  Then $G\in \calR_{ex}([n],\calP)$ if and only if $G$ is downward closed and $\Psi(G)\in DT_k(n)$.
\end{proposition}
\begin{proof}
Suppose first $G\in \calR_{ex}([n],\calP)$.  By Corollary \ref{dgcor}, $G$ is downward closed.  So Corollary \ref{digcor2} implies $sub(G)=2^{e(\Psi(G))}$.  Suppose $\Psi(G)\notin DT_k(n)$.  Then Theorem \ref{digext} implies that for any $H\in DT_k(n)$, $2^{e(H)}>2^{e(\Psi(G))}$.  By Lemma \ref{diglem9}, $\Psi^{-1}(H)\in \calR([n],\calP)$ and is downward closed, so Corollary \ref{digcor2} implies $sub(\Psi^{-1}(G))=2^{e(H)}>2^{e(\Psi(G))}=sub(G)$, contradicting that $G\in \calR_{ex}([n],\calP)$.

Suppose now that $G$ is downward closed and $\Psi(G)\in DT_k(n)$.  Observation \ref{dgob1}(b) implies $\Psi^{-1}(\Psi(G))=G$, so Lemma \ref{diglem9} implies $\Psi^{-1}(\Psi(G))=G\in \calR([n],\calP)$.  Suppose towards a contradiction that $G\notin \calR_{ex}([n],\calP)$.  Then there is $G'\in \calR([n],\calP)$ such that $sub(G')>sub(G)$.  By applying Lemma \ref{dgdw}, we may assume $G'$ is downward closed.  By Lemma \ref{diglem2}, $\Psi(G')\in \calP$.  By Corollary \ref{digcor2}, $2^{e(\Psi(G'))}=sub(G')>sub(G)=2^{e(\Psi(G)}$, which implies $e(\Psi(G'))>e(\Psi(G))$, contradicting that $\Psi(G)$ is edge extremal (since by Theorem \ref{digext}, elements of $DT_k(n)$ are edge-extremal).
\end{proof}

\subsection{Proofs of Results} In this section we prove Theorems \ref{digen} and \ref{digapprox}.  We first compute $\pi(\calP)$ using Theorem \ref{digext} and Proposition \ref{digprop3}. Theorem \ref{digen} will follow immediately from this and Theorem \ref{enumeration2}.

\begin{corollary}\label{digcor1}
$\ex(n,\calP)=2^{max(n,T_{k+1})}= 3^{t_k(n)}$.  Consequently $\pi(\calP)=3^{(1-\frac{1}{k})\frac{1}{2}}$.
\end{corollary}

\begin{proof}
Let $G\in \calR_{ex}([n],\calP)$.  Then by definition, $\ex(n,\calP)=sub(G)$.  Note Proposition \ref{digprop3} implies $\Psi(G)\in DT_k(n)$ and Corollary \ref{digcor2} implies $sub(G)=2^{e(\Psi(G))}$.   Combining these facts with Theorem \ref{digext} implies $\ex(n,\calP)=sub(G)=2^{\max(n,T_{k+1})}= 3^{t_k(n)}$, as desired.  By definition of $\pi(\calP)$ and $t_k(n)$, this implies $\pi(\calP)=3^{(1-\frac{1}{k})\frac{1}{2}}$.
\end{proof}

\noindent{\bf Proof of Theorem \ref{digen}.} By Theorem \ref{enumeration2} and Corollary \ref{digcor1}, $|\calP_n|=ex(n,\calP)2^{o(n^2)}=3^{t_k(n)+o(n^2)}$.  Since $|\calP_n|=|\Forb_{di}(n,T_{k+1})|$, we are done.
\qed

\vspace{3mm}

We now prove Theorem \ref{digapprox} using Theorem \ref{digstab} and Theorem \ref{stab}.  We need the following observation which allows us to translate between the different notions of $\delta$-closeness.

\begin{observation}\label{dgob2}
Suppose $n$ is an integer, $M$ and $N$ are in $\Forb_{di}(n,T_{k+1})$, and $G_1$, $G_2$ are downward closed $\calL_{\calP}$-templates with domain $[n]$.  Then the following hold.
\begin{enumerate}[(a)]
\item $\diff(M,N)=\{xy\in {[n]\choose 2}: (x,y)\text{ or }(y,x)\in \Delta(M,N)\}$.  Therefore $|\Delta(M,N)|=2|\diff(M,N)|$.
\item $\diff(G_1,G_2)=\{xy\in {V\choose 2}: (x,y)\text{ or }(y,x)\in \Delta(\Psi(G_1),\Psi(G_2))\}$.  Therefore 
$$
2|\diff(G_1,G_2)|=|\Delta(\Psi(G_1),\Psi(G_2))|.
$$
\end{enumerate}
\end{observation}
\begin{proof}
That $(a)$ holds is immediate from the definitions.  For (b), since $G_1$ and $G_2$ are $\calL_{\calP}$-templates, Lemma \ref{templatelem} implies $\diff(G_1,G_2)=\{xy\in {V\choose 2}: Ch_{G_2}(xy)\neq Ch_{G_2}(xy)\}$.  We leave it to the reader to verify that since $G_1$ and $G_2$ are both downward closed, 
\begin{align*}
\{xy\in {V\choose 2}: Ch_{G_2}(xy)\neq Ch_{G_2}(xy)\}&=\{xy\in {V\choose 2}: (x,y)\text{ or }(y,x)\in E(\Psi(G_1))\Delta E(\Psi(G_2))\}\\
&=\{xy\in {V\choose 2}: (x,y) \text{ or }(y,x)\in \Delta(\Psi(G_1),\Psi(G_2))\}.
\end{align*}
\end{proof}

\newpage
\begin{proposition}\label{digprop4}
$\calP$ has a stability theorem.
\end{proposition}
\begin{proof}
Fix $\delta>0$.  Choose $\epsilon$ so that for sufficiently large $n$, Theorem \ref{digstab} implies that for all $G\in \Forb_{di}(n,T_{k+1})$, if $e(G)\geq max(n,T_{k+1})(1-\epsilon)$, then $G\in DT^{\delta/2}_k(n)$.  Suppose $G\in \calR([n],\calP)$ satisfies $sub(G)\geq ex(n,\calP)^{1-\epsilon}$.  We want to show there is $G'\in \calR_{ex}([n],\calP)$ such that $\dist(G,G')\leq \delta$.  Apply Lemma \ref{diglem9} to obtain $G^*\in \calR([n],\calP)$ which is downward closed and which satisfies that $sub(G^*)\geq sub(G)(4/3)^{|\diff(G^*,G)|}$.  Then our assumptions imply
$$
\ex(n,\calP)^{1-\epsilon}\leq sub(G)(4/3)^{|\diff(G^*,G)|}\leq sub(G^*)\leq \ex(n,\calP).
$$
Rearranging this, we obtain that $(4/3)^{|\diff(G^*,G)|}\leq \ex(n,\calP)^{\epsilon}$.  Assume $n$ is sufficiently large so that $\ex(n,\calP)\leq \pi(\calP)^{2n^2}$.  Then we have $(4/3)^{|\diff(G^*,G)|}\leq \pi(\calP)^{2\epsilon n^2}$ (see Observation \ref{ob5}(a)).  Taking logs of both sides and rearranging yields
$$
|\diff(G^*,G)|\leq \epsilon \pi(\calP)n^2 / \log(4/3).
$$
We may assume $\epsilon$ was chosen sufficiently small so that $4\epsilon \pi(\calP)n^2 / \log(4/3)\leq \delta/2$.  Then we have $\dist(G,G^*)\leq \delta/2$. 

Proposition \ref{digprop2} implies $\Psi(G^*)\in \Forb_{di}(n,T_{k+1})$ and Corollary \ref{digcor2} implies 
$$
2^{e(\Psi(G^*))}=sub(G^*)\geq ex(n,\calP)^{1-\epsilon}=2^{max(n,T_{k+1})(1-\epsilon)},
$$
where the last equality is by Corollary \ref{digcor1}. This implies $e(\Psi(G^*))\geq max(n,T_{k+1})(1-\epsilon)$.  By Theorem \ref{digstab}, this implies $\Psi(G^*)\in DT_k^{\delta/2}(n)$.  Let $H\in DT_k(n)$ satisfy $|\Delta(\Psi(G^*),H)|\leq \delta/2 {n\choose 2}$.  By Proposition \ref{diglem9}, we have that $\Psi^{-1}(H)$ is a downward closed element of $\in \calR([n],\calP)$.  By Observation \ref{dgob1}, $\Psi(\Psi^{-1}(H))=H\in DT_k(n)$, so by Proposition \ref{digprop3}, $\Psi^{-1}(H)\in \calR_{ex}([n],\calP)$.  We show $\dist(G^*,\Psi^{-1}(H))\leq \delta/2$.  Since $G^*$ and $\Psi^{-1}(H)$ are downward closed, Observation \ref{dgob1}(a) and Observation \ref{dgob2}(b) imply
\begin{align*}
|\diff(G^*,\Psi^{-1}(H))|=\frac{1}{2}|\Delta(\Psi(G^*),\Psi(\Psi^{-1}(H))|=\frac{1}{2}|\Delta(\Psi(G^*),H)|\leq \delta/2{n\choose 2}.
\end{align*}
Thus $\dist(G^*,\Psi^{-1}(H))\leq \delta/2$.  Combining all this, we have that 
$$
\dist(G,\Psi^{-1}(H))\leq \dist(G,G^*)+\dist(G^*,\Psi^{-1}(H))\leq \delta.
$$
This shows $G$ is $\delta$-close to an element of $\calR_{ex}([n],\calP)$, as desired.
\end{proof}

\noindent {\bf Proof of Theorem \ref{digstab}.} Recall we want to show that for all $\delta>0$, there is a $\beta>0$ so that for sufficiently large $n$, 
\begin{align*}
\frac{|\Forb_{di}(n,T_{k+1})\setminus \mathbb{D}\mathbb{T}^{\delta}_k(n)|}{|\Forb_{di}(n,T_{k+1})|}\leq 2^{-\beta {n\choose 2}}.
\end{align*}
Fix $\delta>0$.  Proposition \ref{digprop4} and Theorem \ref{stab} imply that there exist $\beta>0$ such that for all sufficiently large $n$,
\begin{align}\label{digineq}
\frac{|\calP_n\setminus E^{\delta}(n,\calP)|}{|\calP_n|}\leq 2^{-\beta {n\choose 2}},
\end{align}
where recall, $E(n,\calP)=\{G\in \Forb_{di}(n,T_{k+1}): G\unlhd_pG'\text{ some }G'\in \calR_{ex}([n],\calP)\}$, and
$$
E^{\delta}(n,\calP)=\{G\in \Forb_{di}(n,T_{k+1}): \dist(G,G')\leq \delta \text{ for some }G'\in E(n,\calP)\}.
$$
Therefore it is clearly sufficient to show that for all $n$, $E^{\delta}(n,\calP)=\mathbb{DT}_k^{\delta/2}(n)$.  We first show that $E(n,\calP)=\mathbb{DT}_k(n)$.  Suppose $G\in E(n,\calP)$.  Then there is $G'\in \calR_{ex}([n],\calP)$ such that $G\unlhd_pG'$.  By Proposition \ref{digprop3}, $G'$ is downward closed and $\Psi(G')\in DT_k(n)$. By Lemma \ref{diglem2}, $G\unlhd_pG'$ implies that $G\subseteqq \Psi(G')\in DT_k(n)$.  By definition this shows $G\in \mathbb{DT}_k(n)$.  Suppose now $G\in \mathbb{DT}_k(n)$.  By definition, there is $G'\in DT_k(n)$ such that $G\subseteqq G'$.  By Observation \ref{dgob1}(a), $\Psi^{-1}(G')$ is a downward closed $\calL_{\calP}$-template and $\Psi(\Psi^{-1}(G'))=G'\in DT_k(n)$.  Thus Proposition \ref{digprop3} implies $\Psi^{-1}(G)\in \calR_{ex}([n],\calP)$.  By Lemma \ref{diglem2}, $G\subseteqq G'=\Psi(\Psi^{-1}(G'))$ implies $G\unlhd_p\Psi^{-1}(G')$.  Thus by definition, $G\in E(n,\calP)$.  We have now shown $E(n,\calP)=\mathbb{DT}_k(n)$.  Observation \ref{dgob2}(a) implies $|\Delta(M,N)|=2|\diff(M,N)|$ for any $M,N\in \calP_n$.  Therefore $E^{\delta}(n,\calP)=\mathbb{DT}^{\delta/2}_{k}(n)$.
\qed

%***********************************************************************************
\section{Triangle-free Triple Systems}\label{trifreesec}

%%%%%%%%%%%%%%%%%%%%%%%%%%%%%%%%%%%%%%%%%

Let $F$ be the hypergraph with vertex set $\{1,2,3,4,5\}$ and edge set $\{123,124,345\}$, where $xyz$ denotes the set $\{x,y,z\}$.  A $3$-uniform hypergraph is called \emph{triangle-free} (or $F$-free) if it contains no subgraph (non-induced) isomorphic to $F$.  In this section, we consider results from \cite{BOLLOBAS}, \cite{KM}, and \cite{BM2} about triangle-free $3$-uniform hypergraphs.  In particular, we will reprove approximate structure and enumeration results from \cite{BM2} using the machinery of this paper, along with extremal results from \cite{BOLLOBAS} and a stability theorem from \cite{KM}.  

%%%%%%%%%%%%%%%%%%%%%%%%%%%%%%%%%%%%%%%%%
\subsection{Statements of Results from \cite{BOLLOBAS}, \cite{KM}, and \cite{BM2}.}
%%%%%%%%%%%%%%%%%%%%%%%%%%%%%%%%%%%%%%%%%
In this subsection we state the results of interest from \cite{BOLLOBAS}, \cite{KM}, and \cite{BM2}.  Recall that a $3$-uniform hypergraph is a pair $(V,E)$ where $V$ is a set of vertices and $E\subseteq {V\choose 3}$.  Suppose $G=(V,E)$ and $G'=(V',E')$ are $3$-uniform hypergraphs.  We say $G$ is a \emph{subgraph} of $G'$ if $V'\subseteq V$ and $E'\subseteq E$.  $G$ is a \emph{full} subgraph of $G'$, denoted $G\subseteqq G'$, if further, $V'=V$.  $G$ and $G'$ are \emph{isomorphic} if there is a bijection $f:V'\rightarrow V$ such that for all $xyz\in {V'\choose 3}$, $xyz\in E'$ if and only if $f(x)f(y)f(z)\in E$.  A $3$-uniform hypergraph $G=(V,E)$ is called \emph{tripartite} if and only if there is some partition $U_1, U_2, U_3$ of $V$ such that $xyz\in E$ implies $x$, $y$, and $z$ are all in different parts of the partition. A tripartite $3$-uniform hypergraph is called \emph{balanced} if the partition $U_1,U_2,U_3$ can be chosen to be an equipartition.  Given $n$, let $E(n)$ denote the set of balanced tripartite $3$-uniform hypergraph on $[n]$ and let $e(n)$ be the number of edges in an element of $E(n)$.  Let $F(n)$ be the set of $F$-free $3$-uniform hypergraphs with vertex set $[n]$.  The following is a consequence of the main theorem in \cite{BOLLOBAS}.

\begin{theorem}[{\bf Bollob\'{a}s \cite{BOLLOBAS}}]\label{trifree5}
Suppose $G=(V,E)$ is a triangle-free $3$-uniform hypergraph on $n$ vertices.  Then $|E|\leq e(n)$.  If $|E|=e(n)$, then $G$ is isomorphic to an element of $E(n)$. 
\end{theorem}

It is observed in \cite{BM2} that 
\begin{align}\label{noHG}
e(n)=\Big \lfloor \frac{n}{3}\Big\rfloor \Big\lfloor \frac{n+1}{3}\Big\rfloor \Big\lfloor \frac{n+2}{3}\Big\rfloor =\frac{n^3}{27}+o(n^3).
\end{align}
Combining this with a theorem of Nagle and R\"{o}dl in \cite{NR} implies the following theorem (stated in \cite{BM2} as Corollary 1).

\begin{theorem}[{\bf Nagle-R\"{o}dl \cite{NR}, Bollob\'{a}s \cite{BOLLOBAS}}]\label{fghext}
$|F(n)|=2^{e(n)+o(n^2)}=2^{\frac{n^3}{27}+o(n^3)}$.
\end{theorem}

From now on, in this section, we will just say ``hypergraph'' in place of ``$3$-uniform hypergraph.''  Given a hypegraph $G=(V,E)$ and a partition $U_1,U_2,U_3$ of $V$, a \emph{non-crossing edge} for the partition is an edge $xyz\in E$ such that for some $i\in \{1,2,3\}$, $|xyz\cap U_i|\geq 2$.  The following is Theorem 5 in \cite{BM2}, and was proved in \cite{KM}.

\begin{theorem}[{\bf Keevash-Mubayi \cite{KM}}]\label{hgstab}
For every $\delta>0$ there is an $\epsilon>0$ such that for all sufficiently large $n$, if $G=(V,E)$ is an $F$-free hypergraph with $|V|=n$ and $|E|\geq (1-\epsilon)\frac{n^3}{27}$, then there is a partition $U_1,U_2,U_3$ of $V$ such that $E$ contains at most $\delta n^3$ crossing edges with respect to this partition.
\end{theorem}

Given an element $G$ of $F(n)$, an \emph{optimal partition} of $G=(V,E)$ is a partition $U_1,U_2,U_3$ of $[n]$ so that $E$ contains the minimal number of crossing edges for $U_1,U_2,U_3$.  Given $\delta>0$, let $F(n,\delta)$ be the set of $G\in F(n)$ such that there is an optimal partition for $G$ with at most $\delta n^3$ crossing edges.  Then Theorem \ref{hgstab} and a hypergraph regularity lemma are used in \cite{BM2} to prove the following.  

\begin{theorem}[{\bf Balogh-Mubayi \cite{BM2}}]\label{hgappst}
For all $\delta>0$ there is $\beta>0$ such that for all sufficiently large $n$, 
$$
\frac{|F(n)\setminus F(n,\delta)|}{|F(n)|}\leq 2^{-\beta {n\choose 3}}.
$$
\end{theorem}

\begin{definition}\label{deltaclosehg}
Suppose $G=(V,E)$ and $G'=(V,E')$ are hypergraphs.  Set $\Delta(G,G')=E\Delta E'$.  Given $\delta>0$, we say $G$ and $G'$ are \emph{$\delta$-close} if $|\Delta(G,G')|\leq \delta {n\choose 3}$.  
\end{definition}

\begin{comment}
We leave it as an exercise to check that $\Delta(G,G')=\{xyz\in {V\choose 3}: (x,y,z)\in \diff(G,G')\}$.  Thus 
$$
|\Delta(G,G')|= \frac{1}{3!}\dist(G,G')=\frac{1}{6}\dist(G,G').
$$
\end{comment}

We now give a restatement of Theorem \ref{hgstab} using this notion of $\delta$-closeness.  The arguments are either standard or appear somewhere in \cite{BM2}.

\begin{theorem}\label{trifree7}
For all $\delta>0$ there is an $\epsilon>0$ such that for all sufficiently large $n$, the following holds. If $G=([n],E)\in F(n)$ and $|E|\geq (1-\epsilon)\frac{n^3}{27}$, then $G$ is $\delta$-close (in the sense of Definition \ref{deltaclosehg}) to an element of $E(n)$.  
\end{theorem}
\begin{proof}
Fix $\delta>0$ and choose $\epsilon_1\leq \delta^2$ so that Theorem \ref{hgstab} holds for $\delta^2/2$.  Fix $n$ sufficiently large and let $G=([n],E)\in F(n)$ be such that $|E|\geq (1-\epsilon_1)\frac{n^3}{27}$.  By Theorem \ref{hgstab}, there is a partition $U_1,U_2,U_3$ of $V$ such that $E$ contains at most $(\delta^2/2) n^3$ crossing edges with respect to this partition.  Let $\epsilon_2=6\delta$.  Suppose towards a contradiction that for some $i\in \{1,2,3\}$, $|U_i|> n/3+\epsilon_2n$, say $|U_1|> n/3+\epsilon_2n$.  Let $x=|U_1|-n/3$.  Then we have $|U_2|+|U_3|\leq 2n/3-x$, which implies by the AM-GM inequality that $|U_2||U_3|\leq (n/3-x/2)^2$.  So 
$$
|U_1||U_2||U_3|=(n/3+x)\Big(\frac{n}{3}-\frac{x}{2}\Big)^2=(n/3+x)(n^2/9-nx/3+x^2/4)= \frac{n^3}{27}-\frac{nx^2}{4}+\frac{x^3}{4}.
$$
Since $|U_1|=n/3+x\leq n$, we have that $x\leq 2n/3$.  Thus 
$$
\frac{n^3}{27}-\frac{nx^2}{4}+\frac{x^3}{4}=\frac{n^3}{27}+\frac{x^2}{4}(x-n)\leq \frac{n^3}{27}+\frac{x^2}{4}(-n/3)= \frac{n^3}{27}-\frac{x^2n}{12}.
$$
Since $x\geq \epsilon_2 n$, this implies $|U_1||U_2||U_3|\leq \frac{n^3}{27}-\frac{\epsilon_2^2 n^3}{12}$.  But now the total number of edges in $G$ is by assumption at most
$$
(\delta^2 /2) n^3 + \frac{n^3}{27}-\frac{\epsilon_2^2n^3}{12}=\frac{n^3}{27}\Big(1+\frac{27\delta^2}{2}-\frac{\epsilon_2^227}{12}\Big) < (1-\epsilon_1)\frac{n^3}{27},
$$
where the last inequality is because $\epsilon_2=6\delta$ and $\epsilon_1\leq \delta^2 $.  This contradicts that $|E|>(1-\epsilon_1)\frac{n^3}{27}$.  Thus for each $i$, $||U_i|-n/3|\leq \epsilon_2 n$.  If $n$ is sufficiently large, this implies there is an equipartition $V_1,V_2,V_3$ of $[n]$ such that for each $i$, $|V_i\Delta U_i|\leq 2\epsilon_2 n$.  Let $G'=(V,E')$ be the complete tripartite hypergraph with parts $V_1, V_2, V_3$.  Consider the following subsets of $\Delta(G,G')$. 
\begin{itemize}
\item Let $\Gamma_1$ be the set of $e\in \Delta(G,G')$ such that $e$ contains a vertex in $V_i\Delta U_i$ for some $i\in\{1,2,3\}$.  For each $i=1,2,3$, $|V_i\Delta U_i|\leq 2\epsilon_2n$ so there are at most $2\epsilon_2n^2$ edges containing a vertex in $V_i\Delta U_i$.  Thus $|\Gamma_1|\leq 6\epsilon_2 n^3$.
\item Let $\Gamma_2$ be the set of $e\in E$ which are crossing edges for $U_1, U_2, U_3$ which are also crossing edges for $V_1,V_2,V_3$. By assumption on the $U_i$, $|\Gamma_2|\leq \delta^2 n^3/2$ edges.
\item Let $\Gamma_3$ be the set of $e\in {n\choose 3}$ which are non-corssing in $U_1,U_2, U_3$ and non-crossing in $V_1,V_2,V_3$ but which are not in $E$.  Since $G$ is $F$-free, we must have that $|\Gamma_3|\leq e(n)-|E|$.  Since $|E|\geq (1-\epsilon_1)\frac{n^3}{27}$ and $e(n)=\frac{n^3}{27}+o(n^3)$, we may assume $n$ is sufficiently large so that $|\Gamma_3|\leq |E|-e(n)\leq 2\epsilon_1 n^3/27$.
\end{itemize}
We claim $\Delta(G,G')=\Gamma_1\cup \Gamma_2\cup \Gamma_3$.  Indeed, suppose $e\in \Delta(G,G')\setminus \Gamma_1$.  Then either $e$ is crossing for $U_1, U_2,U_3$ and for $V_1,V_2,V_3$ or $e$ is non-crossing for $U_1, U_2,U_3$ and for $V_1,V_2,V_3$.  If $e$ is crossing for $U_1, U_2,U_3$ and for $V_1,V_2,V_3$, then by definition of $G'$, $e\in E'$, so $e\in \Delta(G,G')$ implies $e\notin E$.  This shows $e\in \Gamma_3$.  On the other hand, if $e$ is non-crossing for $U_1, U_2,U_3$ and for $V_1,V_2,V_3$, then by definition of $G'$, $e\notin E'$, so $e\in \Delta(G,G')$ implies $e\in E$.  This shows $e\in \Delta_2$.  Thus $\Delta(G,G')=\Gamma_1\cup \Gamma_2\cup \Gamma_3$.  Therefore our bounds above for the $|\Gamma_i|$ imply the following.
$$
|\Delta(G,G')|\leq n^3(6\epsilon_2+\delta^2/2+ 2\epsilon_1/27)<74\delta n^3,
$$
where the last inequality is because $\epsilon_2=6\delta$ and $\epsilon_1\leq \delta^2/2$.  Thus $|\diff(G,G')|<6(74)\delta n^3$, so $\dist(G,G')<6(74)\delta$.  Clearly by scaling the $\delta$ we start with, we can obtain the conclusion of the corollary.
\end{proof}

Let $\mathbb{E}(n)=\{G\in F(n): G\subseteqq G',\text{ for some }G'\in E(n)\}$.  Given $\delta>0$, let $\mathbb{E}^{\delta}(n)$ be the set of $G \in F(n)$ which are $\delta$-close (in the sense of Definition \ref{deltaclosehg}) to an element of $\mathbb{E}(n)$.  Then we consider the following to be a restatement of Theorem \ref{hgappst}.  

\begin{theorem}\label{trifree10}
For all $\delta>0$ there is $\beta>0$ such that for sufficiently large $n$,
$$
\frac{|F(n)\setminus \mathbb{E}^{\delta}(n)|}{|F(n)|}\leq 2^{-\beta {n\choose 3}}.
$$
\end{theorem}

We will use the main theorems of this paper along with Corollary \ref{trifree7} and Theorem \ref{trifree5} to reprove Theorems \ref{fghext} and \ref{trifree10}.

%%%%%%%%%%%%%%%%%%%%%%%%%%%%%%%%%%%%%%%%%
\subsection{Preliminaries}
%%%%%%%%%%%%%%%%%%%%%%%%%%%%%%%%%%%%%%%%%

In this section we give translations of the combinatorial notions from \cite{BM2} to the setup of this paper.  Let $\calL=\{E(x,y,z)\}$.  Let $\calP$ be the class of all finite triangle-free $3$-uniform hypergraphs, considered as $\calL$-structures.  It is clear that $\calP$ is a hereditary $\calL$-property.  Since $r_{\calL}=3$, we have that $\calL_{\calP}=\{E_p(x_1,x_2,x_3): p(\xbar)\in S_3(\calP)\}$.  Let $\xbar=(x_1,x_2,x_3)$ and set
\begin{align*}
q_1(\xbar)&=\{x_i\neq x_j :i\neq j\}\cup \{E(x_i,x_j,x_k): |\{i,j,k\}|=3\}\cup \{\neg E(x_i,x_j,x_k): |\{i,j,k\}|<3\}\text{ and }\\
q_1(\xbar)&=\{x_i\neq x_j :i\neq j\}\cup \{\neg E(x_i,x_j,x_k): |\{i,j,k\}|=3\}\cup \{\neg E(x_i,x_j,x_k): |\{i,j,k\}|<3\}.
\end{align*}
Then let $p_1(x_1,x_2,x_3)$ and $p_2(x_1,x_2,x_3)$ be the unique quantifier free $3$-types containing $q_1(\xbar)$ and $q_2(\xbar)$, respectively.  We leave it to the reader to verify that $S_3(\calP)=\{p_1(\xbar), p_2(\xbar)\}$, so consequently $\calL_{\calP}=\{p_1(x_1,x_2,x_3), p_2(x_1,x_2,x_3)\}$.

\begin{definition}
Given an $\calL_{\calP}$-structure $G$ with domain $V$, the hypergraph associated to $G$ is $\Psi(G):=(V,E)$, where $E=\{A\in {V\choose 3}:\text{ for some enumeration $\abar$ of $A$, $G\models p_1(\abar)\}$}$.
\end{definition}

We say that an $\calL_{\calP}$-template is \emph{downward closed} if $G\models \forall x \forall y \forall z R_{p_1}(x,y,z)\rightarrow R_{p_2}(x,y,z)$.  Given a hypergraph $(V,E)$, define $\Psi^{-1}(V,E)$ to be the $\calL_{\calP}$-structure $G$ which has domain $V$ and for each $(x,y,z)\in V^3$, $G\models R_{p_1}(x,y,z)$ if and only if $xyz\in E$, and $G\models R_{p_2}(x,y,z)$ if and only if $x$, $y$, $z$ are pairwise distinct.  

\begin{observation}\label{hgob1}
Suppose $(V,E)$ is a hypergraph and $G$ is a downward closed $\calL_{\calP}$-template with domain $V$.  Then $\Psi(\Psi^{-1}(V,E))=(V,E)$ and $\Psi^{-1}(V,E)$ is a downward closed $\calL_{\calP}$-template.
\end{observation}

\begin{lemma}\label{trifree2}
Suppose $G$ is a complete $\calL_{\calP}$-structure with domain $V$.  Then $\Psi(G)\unlhd_pG$ and for all $G'\unlhd_pG$, $G'$ is a subhypergraph of $\Psi(G)$.  If $G$ is also downward closed, then for all $G'$ a subhypergraph of $\Psi(G)$, $G'\unlhd_pG$.
\end{lemma}
\begin{proof}
Let $\Psi(G)=(V,E)$.  We first show $\Psi(G)\unlhd_pG$ (where here we are considering $\Psi(G)$ as an $\calL$-structure).  Let $uvw\in {V\choose 3}$ and let $p(x,y,z)=qftp^{\Psi(G)}(u,v,w)$.  We want to show that $G\models R_p(\mu(u,v,w))$ for some $\mu \in Perm(3)$.  Suppose $p=p_1$, so $uvw\in E$.  Then by definition of $\Psi(G)$, there is $\mu\in Perm(3)$ such that $G\models R_{p_1}(\mu(u,v,w))$, as desired.  Suppose now $p=p_2$, so $uvw\notin E$.  Since $G$ is a complete $\calL_{\calP}$-structure, there is some $R_q\in \calL_{\calP}$ and $\mu\in Perm(3)$ such that $G\models R_{q}(\mu(u,v,w))$.  By definition of $\Psi(G)$, $xyz\notin E$ implies $q\neq p_1$.  Thus we must have $q=p_2$, so $G\models R_{p_2}(\mu(u,v,w))$ for some $\mu\in Perm(3)$, as desired.

Suppose $G'=(V,E')\unlhd_pG$ and let $\Psi(G)=(V,E)$.  We want to show $E'\subseteq E$.  Let $uvw\in E'$.  Then $qftp^{G'}(u,v,w)=p_1(x,y,z)$.  Since $G'\unlhd_pG$, this implies $G\models R_{p_1}(\mu(u,v,w))$ for some $\mu \in Perm(3)$.  By definition of $\Psi(G)$, this implies $uvw\in E$, as desired.  Suppose now that $G$ is also downward closed.  Let $G'=(V,E')$ be a subhypergraph of $\Psi(G)=(V,E)$.  Fix $uvw\in {n\choose 3}$.  If $p(x,y,z)=qftp^{G'}(u,v,w)$, we want to show $G\models R_p(\mu(u,v,w))$ for some $\mu \in Perm(3)$.  If $p(x,y,z)=p_1(x,y,z)$, then $uvw\in E'\subseteq E$, so $G\models R_{p_1}(\mu(u,v,w))$ for some $\mu\in Perm(3)$ by definition of $\Psi(G)$.  Thus $G\models R_p(\mu(u,v,w))$ for some $\mu \in Perm(3)$ as desired.  Suppose now $p(x,y,z)=p_2(x,y,z)$. Because $G$ is complete, $G\models R_{q}(\mu(u,v,w))$ for some $q\in \{p_1,p_2\}$ and $\mu \in Perm(3)$.  If $q=p_2$, then we are done.  If $q=p_1$, then because $G$ is downward closed, $G\models R_{p_1}(\mu(u,v,w))$ implies $G\models R_{p_2}(\mu(u,v,w))$.  This finishes the proof. 
\end{proof}

\begin{corollary}\label{corhg1}
If $(V,E)\in \calP$, then $\Psi^{-1}(V,E)$ is a downward closed element of $\calR(V,\calP)$.
\end{corollary}
\begin{proof}
By Observation \ref{hgob1}(a), $\Psi^{-1}(V,E)$ is a downward closed $\calL_{\calP}$-template.  To show $\Psi^{-1}(V,E)$ is $\calP$-random, let $H\unlhd_p\Psi^{-1}(V,E)$.  By Lemma \ref{trifree2}, $H\subseteqq \Psi(\Psi^{-1}(V,E))=(V,E)$, where the equality is by Observation \ref{hgob1}(a).  Since $(V,E)\in \calP$, and elements of $\calP$ are closed under taking subhypergraphs, this implies $H\in \calP$.  
\end{proof}

\begin{corollary}\label{trifree3}
Suppose $G$ is a finite $\calL_{\calP}$-template with domain $V$. Then $sub(G)\leq 2^{e(\Psi(G))}$, with equality holding if $G$ is downward closed.  
\end{corollary}
\begin{proof}
By Lemma \ref{trifree2}, the full subpatterns of $G$ are subhypergraphs of $\Psi(G)$.  Since the number of subhypergraphs of $\Psi(G)$ is $2^{e(\Psi(G))}$, this shows $sub(G)\leq 2^{e(\Psi(G)}$.  If $G$ is downward closed, then every subhypergraph of $\Psi(G)$ is also a subpattern of $G$, so equality holds. 
\end{proof}

%%%%%%%%%%%%%%%%%%%%%%%%%%%%%%%%%%%%%%%%%
\subsection{Proofs of Results.}
%%%%%%%%%%%%%%%%%%%%%%%%%%%%%%%%%%%%%%%%%

In this section we Theorems \ref{fghext} and \ref{trifree10}.  We begin with some preliminary results.

\begin{lemma}\label{Gstarhg}
Suppose $G\in \calR([n],\calP)$ is not downward closed.  Then there is $G^*\in \calR([n],\calP)$ which is downward closed such that $sub(G^*)\geq sub(G)2^{|\diff(G,G^*)|}$.
\end{lemma}
\begin{proof}
Define $G^*$ to agree with $G$ everywhere except on 
$$
\Gamma=\{uvw\in {[n]\choose 3}:Ch_G(uvw)=\{p_1(c_u,c_v,c_w)\}\}.
$$
For $uvw\in \Gamma$, define $G^*\models R_{p_1}(\mu(u,v,w))\wedge R_{p_2}(\mu(u,v,w))$ for all $\mu\in Perm(3)$.  We leave it to the reader to verify that $G^*$ is an $\calL_{\calP}$-template (by definition and because $G$ is).  By definition of $G^*$, for all $xyz\in {[n]\choose 3}$, $Ch_{G^*}(xyz)\supseteq Ch_G(xyz)$.  By Lemma \ref{templatelem}, for all $xyz\in {[n]\choose 3}$, $xyz\in \diff(G,G^*)$ if and only if $Ch_G(xyz)\neq Ch_{G^*}(xyz)$.  Suppose $xyz\in \diff(G,G^*)$.  Then $Ch_G(xyz)\neq Ch_{G^*}(xyz)$ and $Ch_G(xyz)\subseteq Ch_{G^*}(xyz)$ implies  $Ch_G(xyz)\subsetneq Ch_{G^*}(xyz)$.  Since $S_3(\calP)$ contains only two elements, and because $G$ complete implies $|Ch_G(xyz)|\geq 1$, we must have $|Ch_G(xyz)|=1$ and $|Ch_{G^*}(xyz)|=2$.  Therefore
$$
sub({G^*})=\Big(\prod_{xyz\in {[n]\choose 3}}|Ch_G(xyz)|\Big)\Bigg(\prod_{\{xyz: (x,y,z)\in \diff(G,G^*)\}}\frac{|Ch_{G^*}(xyz)|}{|Ch_G(xyz)|}\Bigg)\geq sub(G)2^{|\diff(G,G^*)|},
$$
where the inequality is by Corollary \ref{trifree3} and because for all $xyz\in \diff(G,G^*)$, $\frac{|Ch_{G^*}(xyz)|}{|Ch_G(xyz)|}=\frac{2}{1}$.  We have only left to show that $G^*$ is $\calP$-random.  Suppose $H\unlhd_pG^*$.  By Proposition \ref{trifree2}, $H$ is a subhypergraph of $\Psi(G^*)$.  Observe that by definition of $G^*$ and $\Psi$, $\Psi(G^*)=\Psi(G)$.  By Proposition \ref{trifree2}, $\Psi(G)\unlhd_pG$, so since $G$ is $\calP$-random, $\Psi(G)=\Psi(G^*)\in \calP$.  Then $H\subseteqq \Psi(G^*)=\Psi(G)\in \calP$.  Since $\calP$ is closed under subhypergraphs, this implies $H\in \calP$, so $G^*$ is $\calP$-random.
\end{proof}

\begin{proposition}\label{trifree4}
For all integers $n\geq 2$, the following holds.  Suppose $G\in \calR_{ex}([n],\calP)$.  Then $\Psi(G)\in E(n)$.  Consequently, $\ex(n,\calP)=2^{e(n)}$ and $\pi(\calP)=2^{6/27}$.
\end{proposition}
\begin{proof}
Suppose $G\in \calR_{ex}([n],\calP)$.  By Porposition \ref{trifree2}, $\Psi(G)\unlhd_pG$, so since $G$ is $\calP$-random, we have $\Psi(G)\in \calP_n=F(n)$.  By Lemma \ref{Gstarhg}, if $G$ is not downward closed, then there is $G^*\in \calR([n],\calP)$ which is downward closed and such that $sub(G^*)\geq sub(G)2^{|\diff(G,G^*)|}$. Since $G$ not downward closed and $G^*$ is downward closed, $\diff(G,G^*)\neq \emptyset$ implies $sub(G^*)>sub(G)$, contradicting that $G\in \calR_{ex}([n],\calP)$.  Thus $G$ is downward closed, so Corollary \ref{trifree3} implies $sub(G)=2^{e(\Psi(G))}$.  Suppose towards a contradiction that $\Psi(G)\notin E(n)$.  Then by Theorem \ref{trifree5}, we have that for any $H\in E(n)$, $2^{e(H)}>2^{e(\Psi(G))}$.  By Corollary \ref{corhg1}, $\Psi^{-1}(H)$ is a downward closed element of $\calR([n],\calP)$.  By Corollary \ref{trifree5}, $sub(\Psi^{-1}(H))=2^{e(H)}>2^{e(\Psi(G))}=sub(G)$, contradicting that $G\in \calR_{ex}([n],\calP)$.  Thus we must have $\Psi(G)\in E(n)$.  Consequently, we have shown if $G\in \calR_{ex}([n],\calP)$, then 
$$
sub(G)=\ex(n,\calP)=2^{e(\Psi(G))}=2^{e(n)}.
$$
By definition of $\pi(\calP)=\lim_{n\rightarrow \infty}\ex(n,\calP)^{1/{n\choose 3}}$ and (\ref{noHG}), this implies$\pi(\calP)=2^{6/27}$.
\end{proof}

We now can give a very quick proof of Theorem \ref{fghext}.
\vspace{3mm}

\noindent {\bf Proof of Theorem \ref{fghext}.} Proposition \ref{trifree4} and Theorem \ref{enumeration} imply the following.
$$
|\calP_n|=|F(n)|=\pi(\calP)^{{n\choose 3}+o(n^3)}=2^{\frac{6}{27}{n\choose 3}+o(n^3)}=2^{\frac{n^3}{27}+o(n^3)}.
$$
\qed

We now prove $\calP$ has a stability theorem in the sense of Definition \ref{stabdef}, and use this along with Theorem \ref{stab} to prove Theorem \ref{trifree10}.

\begin{lemma}\label{trifree8}
Suppose $G$ and $G'$ are in $\calR([n],\calP)$ and are downward closed.  Then for all $\delta>0$, $\dist(G,G')\leq \delta$ if and only if $|\Delta(\Psi(G),\Psi(G')|\leq \delta{n\choose 3}$.
\end{lemma}

\begin{proof}
Let $\Psi(G)=([n],E)$ and $\Psi(G')=([n],E')$.  Because $G$ and $G'$ are both $\calL_{\calP}$-templates, Lemma \ref{templatelem} implies that for all $(x,y,z)\in [n]^3$, $xyz\in \diff(G,G')$ if and only if $Ch_G(xyz)\neq Ch_{G'}(xyz)$.  By definition of $\Psi$ and because $G$ and $G'$ are downward closed, for all $(x,y,z)\in [n]^3$, we have that $Ch_G(xyz)\neq Ch_{G'}(xyz)$ if and only if $xyz\in E\Delta E'$.  This shows $\diff(G,G')=\Delta(\Psi(G), \Psi(G'))$.  This shows $\dist(G,G')\leq \delta$ if and only if $|\Delta(\Psi(G), \Psi(G'))|\leq \delta{n\choose 3}$.
\end{proof}

\begin{proposition}\label{trifree9}
$\calP$ has a stability theorem.
\end{proposition}

\begin{proof}
Fix $\delta>0$ and choose $\epsilon<\delta /4$ sufficiently small so that for sufficiently large $n$, the conclusion of Proposition \ref{trifree7} holds for $\delta/4$.  Suppose $G\in \calR([n],\calP)$ satisfies $sub(G)\geq ex(n,\calP)^{1-\epsilon}$.  We want to show there exists an element of $\calR_{ex}([n],\calP)$ which is $\delta$-close to $G$ in the sense of Definition \ref{deltaclose1}.  Choose $G^*$ to be a downward closed element of $\calR([n],\calP)$ as in Lemma \ref{Gstarhg}.  Then 
\begin{align}\label{ffree}
\ex(n,\calP)\geq sub(G^*)\geq sub(G)2^{|\diff(G,G^*)|}\geq ex(n,\calP)^{1-\epsilon}2^{|\diff(G,G^*)|}.
\end{align}
Assume $n$ is sufficiently large so that $\ex(n,\calP)\leq \pi(\calP)^{2n^3}$.  Rearranging (\ref{ffree}), we obtain that $2^{|\diff(G,G^*)|}\leq \ex(n,\calP)^{\epsilon}\leq \pi(\calP)^{2\epsilon n^3}$.  Taking logs of both sides and rearranging, we obtain that $|\diff(G,G^*)|\leq C\epsilon n^3$ where $C=2\log(\pi(\calP))/\log 2$.  Assume we chose $\epsilon$ sufficiently small so that $C\epsilon \leq \delta/4$.   Then $|\diff(G,G^*)|\leq \delta/4{n\choose 3}$, so $\dist(G,G^*)\leq \delta/4$.  

Proposition \ref{trifree2} implies $\Psi(G^*)\unlhd_pG^*$, so since $G^*$ is $\calP$-random, $\Psi(G^*)\in \calP_n$.  Corollary \ref{trifree3} and (\ref{noHG})  imply 
$$
2^{e(\Psi(G^*))}=sub(G^*)\geq sub(G)\geq 2^{(1-\epsilon)\frac{6}{27}{n\choose 3}}.
$$
This implies $e(\Psi(G^*))\geq (1-\epsilon)\frac{6}{27}{n\choose 3}$, so by Proposition \ref{trifree7}, $\Psi(G^*)$ is $\delta/2$-close in the sense of Definition \ref{deltaclosehg} to some $H\in E(n)$.  By Corollary \ref{corhg1}, $\Psi^{-1}(H)$ is a downward closed element of $\calR([n],\calP)$. Thus Lemma \ref{trifree8} implies that because $\Psi(G^*)$ and $\Psi(\Psi^{-1}(H))=H$ are $\delta/2$-close in the sense of Definition \ref{deltaclosehg}, $\dist(G^*,\Psi^{-1}(H))\leq \delta/2$.  By Corollary \ref{trifree3}, Proposition \ref{trifree4}, and because $H\in E(n)$, $sub(\Psi^{-1}(H))=2^{e(H)}=\ex(n,\calP)$.  Thus $\Psi^{-1}(H)\in \calR_{ex}([n],\calP)$ and 
$$
\dist(G,\Psi^{-1}(H))\leq \dist(G,G^*)+\dist(G^*,\Psi^{-1}(H))\leq \delta.
$$
This finishes the proof.
\end{proof}

\noindent {\bf Proof of Theorem \ref{trifree10}.}
Fix $\delta>0$.  Choose $\beta>0$ such that Theorem \ref{stab} holds for $\delta$.  Proposition \ref{trifree9} and Theorem \ref{stab} imply that for sufficiently large $n$, 
$$
\frac{|\calP_n\setminus E^{\delta}(n,\calP)|}{|\calP_n|}\leq 2^{-\beta {n\choose 3}},
$$
where recall $E(n,\calP)=\{G\in \calP_n: G\unlhd_pG'$ for some $G'\in \calR_{ex}([n],\calP)\}$.  Thus to finish the proof, if suffices to show that $\mathbb{E}(n)=E(n,\calP)$.  By Proposition \ref{trifree4}, if $G'\in \calR_{ex}([n],\calP)$, then $\Psi(G')\in E(n)$. Lemma \ref{trifree4} implies $G\unlhd_pG'$ if and only if $G$ is a subhypergraph of $\Psi(G')$.  Thus 
$$
E(n,\calP)=\{G\in \calP_n: G\text{ is a subhypergraph of some }G'\in E(n)\}.
$$
Therefore by definition,  $E(n,\calP)=\mathbb{E}(n)$.
\qed

\bibliography{/Users/Lab/terry/Desktop/science1.bib}
%\bibliography{/Users/carolineterry/Desktop/science1.bib}
\bibliographystyle{amsplain}

\end{document}